\documentclass[10pt]{article}
\usepackage{latexsym}
\usepackage{amssymb}
\usepackage{amsthm}
\usepackage{amsmath}
\usepackage{graphicx}
\usepackage[latin1]{inputenc}
\usepackage{hyperref}
\newtheorem{definition}{Definition}[section]
\newtheorem{theorem}{Theorem}[section]
\newtheorem{prop}[theorem]{Proposition}
\newtheorem{coro}[theorem]{Corollary}
\newtheorem{lemma}[theorem]{Lemma}
\newtheorem{remark}[theorem]{Remark}

\newtheorem{exem}{Example}[section]

\newcommand{\R}{\mathbb{R}}             
\newcommand{\N}{\mathbb{N}}             
\newcommand{\Z}{\mathbb{Z}}             %
\newcommand{\C}{\mathbb{C}}             

\author{Damien Gobin\footnote{Laboratoire de Math\'ematiques Jean Leray, Universit\'e de Nantes, 2, rue de la Houssini\`ere, BP
     92208, 44322 Nantes Cedex 03. Email adress: damien.gobin@univ-nantes.fr. Research supported by the French National Research Project
AARG, No. ANR-12-BS01-012-01.}}
\title{Inverse scattering at fixed energy on three-dimensional asymptotically hyperbolic St\"ackel manifolds}
\date{\today}

\begin{document}

\maketitle


\begin{abstract}
In this paper, we study an inverse scattering problem at fixed energy on three-dimensional asymptotically hyperbolic St\"ackel manifolds
having the topology of toric cylinders and satisfying the Robertson condition. On these manifolds the Helmholtz equation can be separated into a system of a radial ODE and two angular 
ODEs. We can thus decompose the full scattering operator onto generalized harmonics and the resulting partial scattering matrices consist in a countable 
set of $2 \times 2$ matrices whose coefficients are the so-called transmission and reflection coefficients. It is shown that the reflection coefficients are nothing but generalized 
Weyl-Titchmarsh functions associated with the radial ODE. Using a novel multivariable version of the Complex Angular Momentum method, we show that the knowledge of the scattering 
operator at a fixed non-zero energy is enough to determine uniquely the metric of the three-dimensional St\"ackel manifold up to natural obstructions.

\vspace{0.5cm}
\end{abstract}


\section{Introduction and statement of the main result}

In this work we are interested in an inverse scattering problem at fixed energy for the Helmholtz equation on three-dimensional St\"ackel manifolds satisfying the 
Robertson condition. The St\"ackel manifolds were introduced in 1891 by St\"ackel in \cite{Sta1} and are of main interest in the theory of variable separation. Indeed, it is known 
that the additive separability of the Hamilton-Jacobi equation for the geodesic flow on a riemannian manifold is equivalent to the fact that the metric is in St\"ackel form.
However, to obtain the multiplicative separability of the Helmholtz equation, we also have to assume that the Robertson condition is satisfied. 
As we will see in a brief review of the theory of variable separation, there also exist some 
intrinsic characterizations of the separability of the Hamilton-Jacobi and Helmholtz equations in terms of Killing tensors (which correspond to hidden symmetries) or symmetry operators.
We refer to \cite{BCR1,BCR2,CR,Eis,Eis2,KM1,KM3,KM2,Sta} for important contributions in this domain and to \cite{Benen,Mi} for surveys on these questions.
We emphasize that the description of the riemannian manifolds given by St\"ackel is local. We shall obtain a global description of these manifolds by choosing a global St\"ackel system of coordinates and
we shall thus use the name of ``St\"ackel manifold'' in our study. We choose to work on a St\"ackel manifold $(\mathcal{M},g)$ having the topology of a toric cylinder and in order to define the 
scattering matrix on this manifold, we add an asymptotically hyperbolic structure, introduced in \cite{IK} (see also \cite{GuiSB,JSB,SB}), at the two radial ends of our 
cylinder. We can then define the scattering 
operator $S_g(\lambda)$ at a fixed energy $\lambda \neq 0$ associated with the Helmholtz equation on $(\mathcal{M},g)$ which is the object of main interest of this paper. The question we address is the following:
\begin{center}
\emph{Does the scattering operator $S_g(\lambda)$ at a fixed energy $\lambda \neq 0$ uniquely determine the metric $g$ of the St\"ackel manifold?}
\end{center}
We recall that inverse scattering problems at fixed energy on asymptotically hyperbolic manifolds are closely related to the anisotropic Calder\'on problem on compact riemannian
manifolds with boundary. We refer to the surveys \cite{GT,IK,KKL,KS,S,U} for the current state of the art on this question. One of the aim of this paper is thus to give examples of 
manifolds on which we can solve the inverse scattering problem at fixed energy 
but do not have one of the particular structures for which the uniqueness for the anisotropic Calder\'on problem on compact manifolds with boundary is known.
Note that the result we prove here is an uniqueness result. We are also interested in stability result i.e. in study the continuity of the application $g \mapsto 
S_g(\lambda)$. This question will be the object of a future work.

The main tool of this paper consists in complexifying the coupled angular momenta that appear in the variable separation procedure.
Indeed, thanks to variable separation, the scattering operator at fixed energy can be decomposed onto scattering coefficients indexed by a discrete set of \emph{two} angular 
momenta that correspond to the \emph{two} constants of separation.
Roughly speaking, the aim of the Complexification of the Angular Momentum method is the following: from a discrete set of data (here the equality of the 
reflection coefficients on the coupled spectrum) we want to obtain a continuous regime of informations (here the equality of these functions on $\C^2$).
This method consists in three steps. We first allow the angular momentum to be a complex number. We then use uniqueness results for functions in certain analytic classes 
to obtain the equality of the non-physical corresponding data on the complex plane $\C$. Finally, we use this new information to solve our inverse problem thanks to the 
Borg-Marchenko Theorem.
The general idea of considering complex angular momentum originates from a paper of Regge 
(see \cite{Re}) as a tool in the analysis of the scattering matrix of Schr\"odinger operators in $\R^3$ with spherically symmetric potentials.
We also refer to \cite{AR,New} for books dealing with this method. This tool was already used in the field of inverse problems for one angular momentum in
\cite{Pap1,DaKaNi,DNK2,DN3,DN,DN4,G,Ra} and we note that this method is also used in high energy physics (see \cite{Collins}).
In this work we use a novel multivariable version 
of the Complexification of 
the Angular Momentum method for \emph{two angular momenta} which correspond to the constants of separation of the Helmholtz equation. Note that we have to use this multivariable version since 
these two angular moments (which are also coupled eigenvalues of two commuting operators) are not independent and cannot be considered separately.
This work is a continuation of the paper \cite{DaKaNi} of Daud\'e, Kamran and Nicoleau 
in which the authors treat the same question on Liouville surfaces which correspond to St\"ackel manifolds in two dimensions.

\subsection{Review of variable's separation theory}\label{reviewsep}

In this Subsection, we propose a brief review of variable's separation theory. We refer to \cite{Benen,Mi} and to the introduction of \cite{BCR1} for surveys of this
theory. Let $(\mathcal{M},g)$ be a riemannian manifold of dimension $n$. On $(\mathcal{M},g)$, we are first interested in the Hamilton-Jacobi equation
\begin{equation}\label{eqHJ}
\nabla W . \nabla W = E,
\end{equation}
where $E$ is a constant parameter and $\nabla$ is the gradient operator
\[ (\nabla W)^i = g^{ij} \partial_j W,\]
where we use the Einstein summation convention. We are also interested in the Helmholtz equation
\begin{equation}\label{eqHel}
-\Delta_g \psi = E \psi.
\end{equation}
where $\Delta_g$ is the Laplace-Beltrami operator
\[ \Delta_g \psi = g^{ij} \nabla_i \nabla_j \psi,\]
where $\nabla_j$ is the covariant derivative with respect to the Levi-Civita connection. We note that, as it is shown in \cite{BCR1}, we can add a potential $V$ satisfying 
suitable conditions to these equations without more complications in the theory we describe here.
It is known that, in many interesting cases, these equations admit local separated solutions of the form
\[ W(\underline{x},\underline{c}) = \sum_{i=1}^n W_i(x^i,\underline{c}),\]
for the Hamilton-Jacobi equation and of the form
\[ \psi(\underline{x},\underline{c}) = \prod_{i=1}^n \psi_i(x^i,\underline{c}),\]
for the Helmholtz equation, where $\underline{x} = \{ x^i \}$ is a suitable coordinate system on $\mathcal{M}$ and $\underline{c}$ denotes a set of constant parameters, whose 
number depends on an appropriate definition of separation (see below). The reason why we are interested in such solutions is that
it happens that for solutions of this kind, Equations (\ref{eqHJ})-(\ref{eqHel}) become equivalent to a system of ordinary differential separated equations, each one involving a single 
coordinate.
In this work we study the particular case of the orthogonal separation, i.e. when $g^{ij} = 0$ for $i \neq j$. We now recall the definition of separation of variables 
for the Hamilton-Jacobi and the Helmholtz equations.

\begin{definition}[\cite{BCR1}]
 The Hamilton-Jacobi equation is separable in the coordinates $\underline{x} = \{ x^i \}$ if it admits a complete separated solution, i.e. a solution of the kind
 \[ W(\underline{x},\underline{c}) = \sum_{i=1}^n W_i(x^i,\underline{c}),\]
 depending on $n$ parameters $\underline{c} = (c_j)$ satisfying the completeness condition
 \[ \det \begin{pmatrix} \frac{\partial p_i}{\partial c_j}  \end{pmatrix} \neq 0, \quad \textrm{where} \quad p_i = \partial_i W.\]
\end{definition}

\begin{definition}[\cite{BCR1}, Definition 4.1]
 The Helmholtz equation is separable in the coordinates $\underline{x} = \{ x^i \}$ if it admits
 a complete separated solution, i.e. a solution of the form 
 \[\psi(\underline{x},\underline{c}) = \prod_{i=1}^{n} \psi_i(x^i,\underline{c}),\] 
 depending on $2n$ parameters $\underline{c} = (c_j)$ satisfying the completeness condition
 \[ \det \begin{pmatrix} \frac{\partial u_i}{\partial c_j} \\ \frac{\partial v_i}{\partial c_j} \end{pmatrix} \neq 0, \quad \textrm{where} \quad u_i = \frac{\psi_i'}{\psi_i} \quad 
 \textrm{and} \quad v_i = \frac{\psi_i''}{\psi_i}.\]
\end{definition}

We now recall the results proved by St\"ackel, Robertson and Eisenhart at the end of the eighteenth century and at the beginning of the nineteenth century which:
\begin{enumerate}
 \item Characterize the riemannian manifolds admitting orthogonal variable separation.
 \item Make the link between the variable separation for Hamilton-Jacobi and Helmholtz equations.
\end{enumerate}

\begin{definition}[St\"ackel matrix]
 A St\"ackel matrix is a regular $n \times n$ matrix $S(\underline{x}) = (s_{ij}(x^i))$ whose components $s_{ij}(x^i)$ are functions depending on the variable 
 corresponding to the line number only.
\end{definition}

\begin{theorem}[St\"ackel 1893, \cite{Sta}]
 The Hamilton-Jacobi equation is separable in orthogonal coordinates $\underline{x}$ if and only if the metric $g$ is of the form
 \[ g = H_1^2 (dx^1)^2 + H_2^2 (dx^2)^2 + H_3^2 (dx^3)^2,\]
where
\[H_i^2 = \frac{\det(S)}{s^{i1}}, \quad \forall i \in \{1,2,3\},\]
where $s^{i1}$ is the minor associated with the coefficient $s_{i1}$, for all $i \in \{1,2,3\}$.
\end{theorem}

\begin{theorem}[Robertson 1927, \cite{Rob}]\label{thmRob}
 The Helmholtz equation is separable in orthogonal coordinates $\underline{x}$ if and only if in these coordinates the Hamilton-Jacobi equation is separable and moreover the 
 following condition is satisfied
\begin{equation}\label{condRobintro}
  \frac{\det(S)^2}{|g|} = \frac{\det(S)^2}{\prod_{i=1}^n H_i^2} = \prod_{i=1}^n f_i(x^i),
\end{equation}
 where $f_i(x^i)$ are arbitrary functions of the corresponding coordinate only and $|g|$ is the determinant of the metric $g$.
\end{theorem}

Thanks to this Theorem we see that a full understanding of the separation theory for the Helmholtz equation depends on an understanding of the corresponding problem for 
the Hamilton-Jacobi equation and we note that the separability of the Helmholtz equation is more demanding.
The additional condition (\ref{condRobintro}) in Theorem \ref{thmRob} is called \emph{Robertson condition}. This condition has a geometrical meaning given by the 
following characterization dues to Eisenhart.

\begin{theorem}[Eisenhart 1934, \cite{Eis}]
 The Robertson condition (\ref{condRobintro}) is satisfied if and only if in the orthogonal coordinates system $\underline{x}$ the Ricci tensor is diagonal:
 \[ R_{ij} = 0, \quad \forall i \neq j.\]
\end{theorem}

\begin{remark}
We note that the Robertson condition is satisfied for Einstein manifolds. Indeed, an Einstein manifold is a riemannian manifold whose Ricci tensor is proportional to the 
metric which is diagonal in the orthogonal case we study.
\end{remark}

As shown by Eisenhart in \cite{Eis,Eis2} and by Kalnins and Miller in \cite{KM3} the separation of the Hamilton-Jacobi equation for the geodesic flow is related to the 
existence of Killing tensors of order two
(whose presence highlights the presence of hidden symmetries).
We thus follow \cite{Benen,KM3} in order to study this relation.
We use the natural 
symplectic structure on the cotangent bundle $T^{\star} \mathcal{M}$ of the manifold $(\mathcal{M},g)$. Let $\{x^i\}$ be local coordinates on $\mathcal{M}$ and $\{x^i,p_i\}$
the associated coordinates on $T^{\star}\mathcal{M}$.
Let
\[ H = g^{ij}p_i p_j,\]
be the geodesic hamiltonian on $T^{\star} \mathcal{M}$. The Hamilton-Jacobi equation can thus be written as
\[ H(x^i,\partial_i W) = E.\]
Thanks to this formalism, an other necessary and sufficient condition for the separability of the Hamilton-Jacobi equation has been proved by Levi-Civita. We state 
the version given in \cite{Benen}.

\begin{theorem}[Levi-Civita 1904, \cite{LC}]
 The Hamilton-Jacobi equation
 \[ H(x^i,\partial_i W) = E,\]
 admits a separated solution if and only if the differential equations, known as the separability equations of Levi-Civita,
 \[ L_{ij}(H) := \partial_i H \partial_j H \partial^i \partial^j H + \partial^i H \partial^j \partial_i \partial_j H - \partial_i H \partial^j H \partial^i \partial_j H
  - \partial^i H \partial_j H \partial_i \partial^j H = 0, \]
 where $\partial_i = \frac{\partial}{\partial x^i}$ and $\partial^i = \frac{\partial}{\partial p_i}$, are identically satisfied.
\end{theorem}

\begin{remark}
 This Theorem gives us a simple method for testing whether a coordinate system is separable or not. Moreover, it also provides us the basis for the geometrical (i.e. 
 intrinsic) characterization of the separation.
\end{remark}


Since we are interested in a characterization of the separation of the Hamilton-Jacobi equation using Killing tensors, we recall here some basic properties of these objects 
following the second Section of \cite{Benen}. We first recall that the contravariant symmetric tensors $\mathbf{K} = (K^{i...j})$ on $\mathcal{M}$ are in one-to-one correspondence with the 
homogeneous polynomials on $T^{\star} \mathcal{M}$ by the correspondence,
\[ P_{\mathbf{K}} = P(\mathbf{K}) = K^{i...j}p_i...p_j.\]

\begin{exem}
 The hamiltonian $H$ corresponds to the contravariant metric tensor $\mathbf{G}$.
\end{exem}
\noindent
The space of these polynomial functions is closed with respect to the canonical Poisson bracket
\[ \{ A,B \} := \partial^i A \partial_i B - \partial^i B \partial_i A.\]
We then defined the Lie-algebra structure $[.,.]$ on the space of the symmetric contravariant tensors by setting,
\[ P([\mathbf{K}_1,\mathbf{K}_2]) = \{ P(\mathbf{K}_1),P(\mathbf{K}_2) \}.\]

\begin{definition}
 Let $\mathbf{G}$ be a metric tensor. We say that $\mathbf{K}$ is a Killing tensor if $P(\mathbf{K})$ is in involution with $P_{\mathbf{G}} = P(\mathbf{G}) = H$, i.e.
 \[ \{ P(\mathbf{K}),P(\mathbf{G}) \} = 0,\]
 which, by definition, is equivalent to the Killing equation,
 \[ [\mathbf{K},\mathbf{G}] = 0.\]
\end{definition}

\begin{remark}
\begin{enumerate}
 \item This means that if $\mathbf{K}$ is a Killing tensor, $P(\mathbf{K})$ is a first integral of the geodesic flow.
 \item A vector field $\mathbf{X}$ is a Killing vector field, i.e. $[\mathbf{X},\mathbf{G}] =0$, if 
 and only if its flow preserves the metric.
\end{enumerate}
\end{remark}

We now consider the case of a symmetric $2$-tensor $\mathbf{K}$. Since there is a metric tensor, $\mathbf{K}$ can be represented in components as a tensor of type $(2,0)$, 
$(1,1)$ or $(0,2)$, respectively noted $\mathbf{K} = (K^{ij}) = (K_j^i) = (K_{ij})$. As a symmetric tensor of type $(1,1)$, $\mathbf{K}$ defines an endomorphism on the space 
$\chi(\mathcal{M})$ of the vector fields on $\mathcal{M}$ and an endomorphism on the space $\Phi^1(\mathcal{M})$ of the one-forms on $\mathcal{M}$. We denote by $\mathbf{K} \mathbf{X}$
the vector field image of $\mathbf{X} \in \chi(\mathcal{M})$ by $\mathbf{K}$ and by $\mathbf{K} \phi$ the one-form image of $\phi \in \Phi^1(\mathcal{M})$ by $\mathbf{K}$. In other 
words,
\[ \mathbf{K} \mathbf{X} = K_j^i X^j \partial_i \quad \textrm{and} \quad \mathbf{K} \phi = K_j^i \phi_i dx^j.\]
Then a $2$-tensor $\mathbf{K}$ gives rise to eigenvalues, eigenvectors or eigenforms according to the equations $\mathbf{K} \mathbf{X} = \rho \mathbf{X}$ and 
$\mathbf{K} \phi = \rho \phi$.
%
Finally, we denote by  $\mathbf{K}_1 \mathbf{K}_2$ the product of the two endomorphisms $\mathbf{K}_1$ and $\mathbf{K}_2$ whose expression in components is 
$\mathbf{K}_1 \mathbf{K}_2 = K_1^{ih} K_{2h}^j$. The algebraic commutator of two tensors is then denoted by
\[ [\![ \mathbf{K}_1 , \mathbf{K}_2  ]\!] := \mathbf{K}_1 \mathbf{K}_2 - \mathbf{K}_2  \mathbf{K}_1.\]

%

The first link between separation of variables for Hamilton-Jacobi equation and Killing tensors is then given by the following Proposition.

\begin{prop}[\cite{BenF,KM3,Koo}]
 To every orthogonal coordinate system $\{ x^i \}$ which permits additive separation of variables in the Hamilton-Jacobi equation, there correspond $n-1$ second order Killing 
 tensors $\mathbf{K}_1$,...,$\mathbf{K}_{n-1}$ in involution, i.e.
$ [ \mathbf{K}_i , \mathbf{K}_j  ] = 0$, and such that $\{H,P_{\mathbf{K}_1},...,P_{\mathbf{K}_{n-1}}\}$ is linearly independent. The separable solutions $W = \sum_{k=1}^n W^{(k)}(x^k)$ 
 are characterized by the relations
 \[ H(x^i,p_i) = E \quad \textrm{and} \quad P_{\mathbf{K}_l}(x^i,p_i) = \lambda_l, \quad \textrm{where} \quad l = 1,...,n-1 \quad \textrm{and} \quad p_i = \partial_{x^i}W,\]
 where $\lambda_1$,...,$\lambda_{n-1}$ are the separation constants.
\end{prop}

\begin{remark}
 In language of hamiltonian mechanics, Killing tensors correspond to constants of the motion. The link mentioned in the previous Proposition thus states that if the Hamilton-Jacobi 
 equation is separable, then the corresponding hamiltonian system is completely integrable (see \cite{Arn}).
\end{remark}

Whereas it corresponds a family of $n-1$ Killing tensors in involution to every separable system, there also exist families of Killing tensors in involution that are not 
related to separable systems. We thus need to find additional conditions to characterize the admissible families of Killing tensors. Such conditions are given in the 
following two Theorems.

\begin{theorem}[\cite{BCR1} Theorem 7.15]\label{thmsepHJ}
 The Hamilton-Jacobi equation is orthogonally separable if and only if there exist $n$ pointwise independent Killing tensors $(\mathbf{K}_i)$ one other
 commuting as linear operators, i.e. $ [\![ \mathbf{K}_i , \mathbf{K}_j  ]\!] = 0$, for all $i \neq j$, and in involution, i.e.
$ [ \mathbf{K}_i , \mathbf{K}_j  ] = 0$, for all $i \neq j$.
\end{theorem}

\begin{theorem}[\cite{BCR1} Theorem 7.16]
 The Helmholtz equation is orthogonally separable if and only if there exist $n$ pointwise independent Killing tensors $(\mathbf{K}_i)$ one other
 commuting as linear operators, in involution and commuting with the Ricci tensor (Robertson condition).
\end{theorem}
\noindent
As it is shown in \cite{BCR1,BCR2,KM3} and recalled in \cite{Benen} there exist some equivalent reformulations of Theorem \ref{thmsepHJ}. We first recall the definition of 
normal vector field.

\begin{definition}
 On a riemannian manifold $\mathcal{M}$ a vector field $\mathbf{X}$ is called normal if it is orthogonally integrable or surface forming, i.e. if it is orthogonal to a
 foliation of $\mathcal{M}$ by hypersurfaces.
\end{definition}

\begin{theorem}[\cite{KM3,Benen2,Benen}]\label{thmsepHJ1tenseur}
 The geodesic Hamilton-Jacobi equation is separable in orthogonal coordinates if and only if there exists a Killing $2$-tensor with simple eigenvalues and normal eigenvectors.
\end{theorem}

\begin{definition}
 A Killing tensor having these properties is called a characteristic Killing tensor.
\end{definition}

We can show that a characteristic Killing tensor generates a Killing-St\"ackel space, i.e. a $n$-dimensional linear space $\mathcal{K}_n$
of Killing $2$-tensors whose elements commute as linear operators and are in involution. Hence it is immediate from Theorem \ref{thmsepHJ1tenseur} that Theorem 
\ref{thmsepHJ} holds.


%

We conclude this Subsection by an intrinsic characterization of separability for the Hamilton-Jacobi and the Helmholtz equations using symmetry operators given in \cite{KM2}.
We first recall the one-to-one correspondence between a second-order operator and a second order Killing tensor. We have already seen that the contravariant symmetric 
tensors are in one-to-one correspondence with the homogeneous polynomials on $T^{\star} \mathcal{M}$.
Moreover, the operator $\hat{P}_{\mathbf{K}}$ corresponding to a second-degree homogeneous polynomial 
\[ P_{\mathbf{K}} = K^{ij}p_i p_j\]
associated with a symmetric contravariant two-tensor $\mathbf{K}$ on $\mathcal{M}$ is defined by 
\[ \hat{P}_{\mathbf{K}} \psi  = - \Delta_{\mathbf{K}} \psi = -\nabla_i (K^{ij} \nabla_j \psi).\]
We note that, if $\mathbf{K} = \mathbf{G}$ is the contravariant metric tensor, we obtain the Laplace-Beltrami operator $\Delta_g = \hat{P}_{\mathbf{G}}$.


\begin{definition}[\cite{KM2}]
 We say that a second-order symmetry operator $\hat{P}$ is in self-adjoint form if $\hat{P}$ can be expressed into the form
 \[ \hat{P} = \frac{1}{\sqrt{|g|}} \sum_{i,j} \partial_i \left( \sqrt{|g|} a^{ij} \partial_j \right) + c,\]
 where $|g|$ is the determinant of the metric, $a^{ij} = a^{ji}$, for all $(i,j) \in \{1,...,n\}^2$, and $c$ is a real constant.
\end{definition}

\begin{remark}
 If $\hat{P}$ is a second-order selfadjoint operator we can associate to $\hat{P}$ the quadratic form on $T^{\star}\mathcal{M}$ defined by $P = a^{ij}p_i p_j$.
\end{remark}

\begin{theorem}[\cite{KM2} Theorem 3]\label{thmKM}
 Necessary and sufficient conditions for the existence of an orthogonal separable coordinate system $\{ x^i \}$ for the Helmholtz equation are that there exists a linearly 
 independent set $\{ \hat{P}_1 = \Delta_g ,\hat{P}_2,...,\hat{P}_n\}$ of second-order differential operators on $\mathcal{M}$ such that:
 \begin{enumerate}
  \item $[\hat{P}_i, \hat{P}_j] := \hat{P}_i \hat{P}_j - \hat{P}_j \hat{P}_i  = 0$, for all $1 \leq i,j \leq n$.
  \item Each $\hat{P}_i$ is in self-adjoint form.
  \item There is a basis $\{ \omega_{(j)}, \, \, 1 \leq j \leq n \}$ of simultaneous eigenforms for the $\{ P_i \}$.
  \item (Robertson condition) The associated second order Killing tensors $(\mathbf{K}_i)$ commute with the Ricci tensor.
 \end{enumerate}
 If these conditions are satisfied then there exist functions $g^i(x)$ such that $\omega_{(j)} g^j dx^j$, for $j=1,...,n$.
\end{theorem}

In our work we will give explicitly the operators $\hat{P}_1 = \Delta_g$, $\hat{P}_2$ and $\hat{P}_3$ corresponding to our study. We finally note that there exists a more 
general notion of separability called the R-separation (see for instance \cite{KM2,BCR1,BCR2}). Our notion of separability corresponds to the case $R=1$. The study of the 
R-separability in our framework will be the object of a future work.

\subsection{Description of the framework}


The aim of this Subsection is to introduce the framework of our paper. Precisely, we give the definition of the St\"ackel manifolds we are dealing with and we show 
that we can make, without loss of generality, some assumptions on the St\"ackel manifolds we consider.


We start this Subsection by the description of the manifolds we study and the definition of the St\"ackel structure.
We first emphasize that the description given by St\"ackel is local. We obtain a global description by choosing a global St\"ackel system of coordinates
which justify the name ``St\"ackel manifold''. We thus consider manifolds with two ends having the topology of toric cylinders with a global chart
\[\mathcal{M} = (0,A)_{x^1} \times \mathcal{T}_{x^2,x^3}^2,\]
where $\mathcal{T}^{2}_{x^2,x^3} =: \mathcal{T}^{2}$ denotes the 2-dimensional torus. The variable $x^1$ is the radial variable whereas $(x^2,x^3)$ denotes the angular 
variables. We emphasize that we choose to work with angular variables living in a $2$-torus but it is also possible to choose other topologies. For instance the case of 
angular variables living in a $2$-sphere could be the object of a future work.
We define a St\"ackel matrix which is a $3\times 3$ matrix of the form
\[ S =  \begin{pmatrix}
                   s_{11}(x^1) & s_{12}(x^1) & s_{13}(x^1) \\ s_{21}(x^2) & s_{22}(x^2) & s_{23}(x^2) \\ s_{31}(x^3) & s_{32}(x^3) & s_{33}(x^3)
                  \end{pmatrix},\]
where the coefficients $s_{ij}$ are smooth functions. Let $\mathcal{M}$ be endowed with the riemannian metric
\begin{equation}\label{met}
 g = H_1^2 (dx^1)^2 + H_2^2 (dx^2)^2 + H_3^2 (dx^3)^2,
\end{equation}
with
\[H_i^2 = \frac{\det(S)}{s^{i1}}, \quad \forall i \in \{1,2,3\},\]
where $s^{i1}$ is the minor (with sign) associated with the coefficient $s_{i1}$ for all $i \in \{1,2,3\}$.

\begin{remark}
 The metric $g$ is riemannian if and only if the determinant of the St\"ackel matrix $S$ and the minors $s^{11}$, $s^{21}$ and $s^{31}$ have the same sign.
 Moreover, if we develop the determinant with respect to the first column, we note that 
 if we assume that $s_{11}$, $s_{21}$ and $s_{31}$ are positive functions and if the minors $s^{11}$, $s^{21}$ and $s^{31}$ have the same sign, then the sign of the 
 determinant of $S$ is necessary the same as the sign of these minors.
\end{remark}

We emphasize that the application $S \mapsto g$ is not one-to-one. Indeed, we 
describe here two invariances of the metric which will be useful in the resolution of our inverse problem.

\begin{prop}[Invariances of the metric]\label{propinv}
 Let $S$ be a St\"ackel matrix.
 \begin{enumerate}
  \item Let $G$ be a $2 \times 2$ constant invertible matrix. The St\"ackel matrix 
    \[ \hat{S} =  \begin{pmatrix} s_{11}(x^1) & \hat{s}_{12}(x^1) & \hat{s}_{13}(x^1) \\ s_{21}(x^2) & \hat{s}_{22}(x^2) & \hat{s}_{23}(x^2)  \\ s_{31}(x^3) & \hat{s}_{32}(x^3) & \hat{s}_{33}(x^3) \end{pmatrix},\]
  satisfying
  \[\begin{pmatrix} s_{i2} & s_{i3} \end{pmatrix} = \begin{pmatrix} \hat{s}_{i2} & \hat{s}_{i3} \end{pmatrix} G, \quad \forall i \in \{1,2,3\},\]
  leads to the same metric as $S$.
  \item The St\"ackel matrix
    \[ \hat{S} =  \begin{pmatrix} \hat{s}_{11}(x^1) & s_{12}(x^1) & s_{13}(x^1) \\ \hat{s}_{21}(x^2) & s_{22}(x^2) & s_{23}(x^2)  \\ \hat{s}_{31}(x^3) & s_{32}(x^3) & s_{33}(x^3) \end{pmatrix},\]
where,
\begin{equation}\label{inv1}
 \begin{cases}
 \hat{s}_{11}(x^1) = s_{11}(x^1) + C_1 s_{12}(x^1) + C_2 s_{13}(x^1) \\
 \hat{s}_{21}(x^2) = s_{21}(x^2) + C_1 s_{22}(x^2) + C_2 s_{23}(x^2) \\
 \hat{s}_{31}(x^3) = s_{31}(x^3) + C_1 s_{32}(x^3) + C_2 s_{33}(x^3)
\end{cases},
\end{equation}
where $C_1$ and $C_2$ are real constants, leads to the same metric as $S$.
 \end{enumerate}
\end{prop}

\begin{proof}
  We recall that 
\[g = \sum_{i=1}^3 H_i^2 (dx^i)^2 \quad \textrm{with} \quad H_i^2 = \frac{\det(S)}{s^{i1}} \quad \forall i \in \{1,2,3\}.\]
 \begin{enumerate}
  \item The result is an easy consequence of the equalities,
\[  s^{i1} = \det(G) \hat{s}^{i1}, \quad \forall i \in \{1,2,3\}, \quad \textrm{and} \quad \det(S) = \det(G) \det(\hat{S}).\]
  \item The result follows from, 
  \[  s^{i1} = \hat{s}^{i1}, \quad \forall i \in \{1,2,3\}, \quad \textrm{and} \quad \det(S) = \det(\hat{S}).\]
 \end{enumerate}
\end{proof}
\noindent
Since we are only interested in recovering the metric $g$ of the St\"ackel manifold, we can choose any representative of the equivalence class described by the invariances 
given in the previous Proposition. This fact allows us to make some assumptions on the St\"ackel matrix we consider as we can see in the following Proposition.

\begin{prop}\label{propcadre}
 Let $S$ be a St\"ackel matrix with corresponding metric $g_S$. There exists a St\"ackel matrix $\hat{S}$
 with $g_{\hat{S}} = g_S$
 and such that
 \begin{equation}\tag{C}\label{cadre}
 \begin{cases}
 \hat{s}_{12}(x^1) > 0 \quad \textrm{and} \quad \hat{s}_{13}(x^1) > 0, \quad \forall x^1\\
 \hat{s}_{22}(x^2) < 0 \quad \textrm{and} \quad \hat{s}_{23}(x^2) > 0, \quad \forall x^2\\
 \hat{s}_{32}(x^3) > 0 \quad \textrm{and} \quad \hat{s}_{33}(x^3) < 0, \quad \forall x^3\\
 \lim\limits_{x^1 \to 0} s_{12}(x^1) = \lim\limits_{x^1 \to 0} s_{13}(x^1) = 1
\end{cases}.
\end{equation}
\end{prop}

\begin{proof}
 See Appendix \ref{ap1}.
\end{proof}

\begin{remark}\label{rkrie}
The condition (\ref{cadre}) has some interesting consequences which will be useful in our later analysis.
\begin{enumerate}
 \item We note that under the condition (\ref{cadre}), $s^{21} = s_{13}s_{32} - s_{12}s_{33}$ and $s^{31}= s_{12}s_{23} - s_{13}s_{22}$ are strictly positive. Thus, since 
 the metric $g$ has to be a riemannian metric we must also have $\det(S) > 0$ and $s^{11} > 0$.
 \item We note that, since $s_{22}, \, s_{33} < 0$ and $s_{23}, \, s_{32} > 0$,
 \[ s^{11} > 0 \quad \Leftrightarrow \quad s_{22} s_{33} > s_{23}s_{32} \quad \Leftrightarrow  \quad \frac{s_{22}}{s_{23}} < \frac{s_{32}}{s_{33}}.\]
 We will use these facts later in the study of the coupled spectrum of the operators $H$ and $L$ corresponding to the symmetry operators of $\Delta_g$ introduced in Subsection
 \ref{reviewsep}.
\end{enumerate}
\end{remark}
\noindent
From now on and without loss of generality, we assume that the St\"ackel matrix $S$ we consider satisfies the condition (\ref{cadre}).

On the St\"ackel manifold $(\mathcal{M},g)$ we are interested in studying of the Helmholtz equation
\[ - \Delta_g f = -\lambda^2 f.\]
As mentionned in Subsection \ref{reviewsep} the St\"ackel structure is not enough to obtain the multiplicative separability of the Helmholtz equation. Indeed, we have to
assume that the Robertson condition is satisfied.
We recall that this condition can be defined as follows: for all $i \in \{1,2,3\}$ there exists $f_i(x^i)$, function of $x^i$ alone, such that
\begin{equation}\label{Rob1}
\frac{s^{11} s^{21} s^{31}}{\det(S)} = f_1 f_2 f_3.
\end{equation}
We can easily reformulate this condition into the form
\begin{equation}\label{Rob2}
\frac{\det(S)^2}{H_1^2 H_2^2 H_3^2} = f_1 f_2 f_3.
\end{equation}

\begin{remark}
 We note that the functions $f_i$, $i \in \{1,2,3\}$, are defined up to positive multiplicative constants whose product is equal to one.
 In the following we will choose, without loss of generality, these constants equal to one.
\end{remark}


\subsection{Asymptotically hyperbolic structure and examples}

The aim of this Subsection is to define the asymptotically hyperbolic structure we add on the previously defined St\"ackel manifolds and to give
three examples which illustrate the diversity of the manifolds we consider.

We say that a riemannian manifold $(\mathcal{M},g)$ with boundary $\partial \mathcal{M}$ is asymptotically hyperbolic if its sectional curvatures tends to $-1$ at the boundary. In this paper, we put an asymptotically 
hyperbolic structure at the two radial ends of our St\"ackel cylinders in the sense given by Isozaki and Kurylev in \cite{IK} (Section 3.2)\footnote{Note that the asymptotically hyperbolic structure 
introduced in \cite{IK} is slightly more general than the one used by Melrose, Guillarmou, Joshi and S\'a Barreto in \cite{GuiSB,JSB,Mel,SB}.}.
We give now the definition of this structure in our framework.

\begin{definition}[Asymptotically hyperbolic St\"ackel manifold] \label{defAH}
A St\"ackel manifold with two asymptotically hyperbolic ends having the topology of a toric cylinder is a St\"ackel manifold $(\mathcal{M},g)$ whose St\"ackel matrix $S$ satisfies the condition 
(\ref{cadre}) with a global chart
\[\mathcal{M} = (0,A)_{x^1} \times \mathcal{T}_{x^2,x^3}^2,\]
where $x^1 \in (0,A)_{x^1}$ corresponds to a boundary defining function for the two asymptotically hyperbolic ends $\{x^1=0\}$ and $\{x^1=A\}$ and $(x^2,x^3) \in [0,B]_{x^2} \times [0,C]_{x^3}$
are angular variables on the $2$-torus $\mathcal{T}_{x^2,x^3}^2$, satisfying the following conditions.
\begin{enumerate}
 \item The St\"ackel metric $g$ has the form (\ref{met}).
 \item The coefficients $s_{ij}$, $(i,j) \in \{1,2,3\}^2$, of the St\"ackel matrix are smooth functions.
 \item The coefficients of the St\"ackel matrix satisfy:
\begin{enumerate}
 \item $H_i^2 > 0$ for $i \in \{1,2,3\}$ (riemannian metric).
 \item $s_{2j}(0) = s_{2j}(B)$, $s_{2j}'(0) = s_{2j}'(B)$, $s_{3j}(0) = s_{3j}(C)$ and $s_{3j}'(0) = s_{3j}'(C)$ for $j \in \{1,2,3\}$ (Periodic conditions in angular variables).
 \item Asymptotically hyperbolic ends at $\{x^1=0\}$ and $\{x^1=A\}$:
 \begin{enumerate}
        \item At $\{x^1=0\}$: 
there exist $\epsilon_0 > 0$ and $\delta > 0$ such that $\forall n \in \mathbb{N}$ there exists $C_{n} > 0$ such that:
$\forall x^1 \in (0,A-\delta)$,
  \[ \|  (x^1 \partial_{x^1})^n ((x^1)^2 s_{11}(x^1) -1) \| \leq C_{n} (1 + |\log(x^1)|)^{-\min(n,1)-1-\epsilon_0},\]
  \[ \|  (x^1 \partial_{x^1})^n ( s_{12}(x^1)- 1) \| \leq C_{n} (1 + |\log(x^1)|)^{-\min(n,1)-1-\epsilon_0},\]
  \[ \|  (x^1 \partial_{x^1})^n ( s_{13}(x^1)- 1) \| \leq C_{n} (1 + |\log(x^1)|)^{-\min(n,1)-1-\epsilon_0}.\]
        \item At $\{x^1=A\}$: 
there exist $\epsilon_1>0$ and $\delta>0$ such that $\forall n \in \mathbb{N}$ there exists $C_{n} > 0$ such that:
$\forall x^1 \in (\delta,A)$,
  \[ \|  ((A-x^1) \partial_{x^1})^n ((A-x^1)^2 s_{11}(x^1)-1) \| \leq C_{n} (1 + |\log((A-x^1))|)^{-\min(n,1)-1-\epsilon_1},\]
  \[ \|  ((A-x^1) \partial_{x^1})^n ( s_{12}(x^1)- 1) \| \leq C_{n} (1 + |\log((A-x^1))|)^{-\min(n,1)-1-\epsilon_1},\]
  \[ \|  ((A-x^1) \partial_{x^1})^n ( s_{13}(x^1)- 1) \| \leq C_{n} (1 + |\log((A-x^1))|)^{-\min(n,1)-1-\epsilon_1}.\]
  \end{enumerate}
\end{enumerate}
\end{enumerate}

\end{definition}

\begin{remark}
 We know that, thanks to the condition (\ref{cadre}), $s_{12}$ and $s_{13}$ tend to $1$ when $x^1$ tends to $0$. However, at the end $\{x^1 = A \}$, we can just say that there exist two positive 
 constants $\alpha$ and $\beta$ such that $s_{12}$ and $s_{13}$ tend to $\alpha$ and $\beta$ respectively. Thus, at the end $\{x^1 = A \}$, we should assume that
  \[ (A-x^{1})^2 s_{11}(x^1) = [1]_{\epsilon_1}, \quad s_{12}(x^1) = \alpha [1]_{\epsilon_1} \quad \textrm{and} \quad s_{13}(x^1) = \beta [1]_{\epsilon_1},\]
 where,
  \[[1]_{\epsilon_1} = 1 + O((1 + |\log(A-x^1)|)^{-1-\epsilon_1}).\]
 However, we can show (see the last point of Remark \ref{Robsimpl}) that, if $s_{22}$ or $s_{33}$ are not constant functions, then $\alpha = \beta = 1$.
\end{remark}

\noindent
Let us explain the meaning of asymptotically hyperbolic ends for St\"ackel manifolds we put here\footnote{We refer to \cite{IK}, Section 3 p.99-101, for a justification of the name ``asymptotically 
hyperbolic''.}. Since the explanation is similar at the end $\{x^1 = A \}$ we just study the end $\{ x^1 = 0 \}$. 
We first write the metric (\ref{met}) in a neighbourhood of $\{x^1 = 0\}$ into the form
\[ g = \sum_{i=1}^{3} H_i^2 (dx^i)^2 = \frac{ \displaystyle{\sum_{i=1}^{3}} (x^1)^2 H_i^2 (dx^i)^2}{(x^1)^2}.\]
By definition,
\begin{equation}\label{condhypH}
 \begin{cases}
(x^1)^2 H_1^2 = (x^1)^2 s_{11} + (x^1)^2 \left( s_{12} \frac{s^{12}}{s^{11}} + s_{13} \frac{s^{13}}{s^{21}} \right) \\
(x^1)^2 H_2^2 = (x^1)^2 s_{11} \frac{s^{11}}{s_{32}s_{13} - s_{33}s_{12}} + (x^1)^2 \left(s_{12} \frac{s^{12}}{s_{32}s_{13}- s_{33}s_{12}} + s_{13} \frac{s^{13}}{s_{32}s_{13} - s_{33}s_{12}} \right)\\
(x^1)^2 H_3^2 = (x^1)^2 s_{11} \frac{s^{11}}{s_{23}s_{12} - s_{22}s_{13}} + (x^1)^2 \left( \frac{s^{12}}{s_{23}s_{12} - s_{22}s_{13}} + \frac{s_{13}}{s_{12}} \frac{s^{13}}{s_{23}s_{12} - s_{22}s_{13}} \right)\\
\end{cases}.
\end{equation}
As it is shown in \cite{MM}, we know that at the end $\{x^1 = 0\}$, the sectional curvature of $g$ approaches $-|dx^1|_h$ where
\[ h = \sum_{i=1}^{3} (x^1)^2 H_i^2 (dx^i)^2.\]
In other words, the opposite of the sectional curvature at the end $\{ x^1 = 0 \}$ is equivalent to
\[ (x^1)^2 H_1^2 = (x^1)^2 s_{11} + (x^1)^2 \left( s_{12} \frac{s^{12}}{s^{11}} + s_{13} \frac{s^{13}}{s^{21}} \right).\]
Thus, since an asymptotically hyperbolic structure corresponds to a sectional curvature which tends to $-1$, we want that this last quantity tends to $1$.
This is ensured by the third assumption of Definition \ref{defAH} which entails that (for $n=0$):
\begin{equation}\label{condhyp}
 (x^{1})^2 s_{11}(x^1) = [1]_{\epsilon_0}, \quad s_{12}(x^1) = [1]_{\epsilon_0} \quad \textrm{and} \quad s_{13}(x^1) = [1]_{\epsilon_0},
\end{equation}
where
\[ [1]_{\epsilon_0} = 1 + O((1 + |\log(x^1)|)^{-1-\epsilon_0}).\]
We also note that, under these conditions we can write thanks to (\ref{condhypH}) the metric $g$, in a neighbourhood of $\{x^1 = 0\}$, into the form
\begin{equation}\label{nvg}
 g = \frac{ (dx^1)^2 + d\Omega_{\mathcal{T}^2}^2 + P(x^1,x^2,x^3,dx^1,dx^2,dx^3)}{(x^1)^2},
\end{equation}
where 
\[d\Omega_{\mathcal{T}^2}^2 = \frac{s^{11}}{s_{32} -s_{33}} (dx^2)^2 + \frac{s^{11}}{s_{23}-s_{22}} (dx^3)^2,\]
is a riemannian metric on the $2$-torus $\mathcal{T}^{2}$ (since $s^{11}$, $s^{21}$ and $s^{31}$ have the same sign) and $P$ is a 
remainder term which is, roughly speaking, small as $x^1 \to 0$. Hence, in the limit $x^1 \to 0$, we see that
\[g \sim \frac{ (dx^1)^2 + d\Omega_{\mathcal{T}^2}^2}{(x^1)^2},\]
that is, $g$ is a small perturbation of a hyperbolic like metric.

\begin{remark}
 \begin{enumerate}
  \item According to the previous definition, we also need conditions on the derivatives of $s_{1j}$, $j \in \{1,2,3\}$, to be in the framework of \cite{IK}.
  \item By symmetry, we can do the same analysis at the end $\{x^1 = A\}$.
 \end{enumerate}
\end{remark}

From the conditions (\ref{condhyp}) and the Robertson condition (\ref{Rob1}) we can obtain more informations on the functions $f_1$, $f_2$ and $f_3$. We first remark that
\begin{eqnarray*}
 f_1 f_2 f_3 &=& \frac{s^{11} s^{21} s^{31}}{\det(S)}\\
 &=& \frac{s^{11} (s_{13}s_{32} - s_{12}s_{33}) ( s_{12} s_{23} - s_{13}s_{22})}{s_{11}s^{11} + s_{12} s^{12} + s_{13}s^{13}}.
\end{eqnarray*}
Thus, using the conditions (\ref{condhyp}), we obtain
\[ f_1 f_2 f_3 \sim (x^1)^2 (s_{23} - s_{22}) (s_{32} - s_{33}), \quad \textrm{when} \quad x^1 \to 0.\]
Hence, we can say that there exist three positive constants $c_1$, $c_2$ and $c_3$ such that $c_1c_2c_3 = 1$ and
\begin{equation}\label{Rob3}
 f_{1}(x^1) = c_1(x^1)^2 [1]_{\epsilon_0}, \quad f_2(x^2) = c_2(s_{23}-s_{22}) \quad \textrm{and} \quad f_3(x^3) = c_3 (s_{32}-s_{33}).
\end{equation}
We thus note that the functions $f_i$, $i \in \{1,2,3\}$, are defined up to positive constants $c_1$, $c_2$ and $c_3$ whose product is equal to $1$.
However, as mentioned previously, we can choose these constants as equal to $1$. Of course, the corresponding result on $f_{1}$ at the end $\{x^1 = A\}$ is also true.

\begin{remark}\label{Robsimpl}
The previous analysis allows us to simplify the Robertson condition and thus the expression of the riemannian metric on the $2$-torus.
 \begin{enumerate}
  \item We first note that, if we make a Liouville change of variables in the $i^{\mathrm{th}}-$variable,
  \begin{equation}\label{chgmtLiou}
    X^i = \int_0^{x^i} \sqrt{g_i(s)} \, ds,
  \end{equation}
  where $g_i$ is a positive function of the variable $x^i$, the corresponding coefficient $H_i^2$ of the metric is also divided by $g_i(x^i)$. The same modification of the metric happens when we 
  divide the $i^{\mathrm{th}}-$line of the St\"ackel matrix by the function $g_i$. Thus, to proceed to a Liouville change of variables is equivalent to divide the 
  $i$-th line of the St\"ackel matrix by the corresponding function.
  \item We now remark that, if we divide the $i^{\mathrm{th}}-$line of the St\"ackel matrix by a function $g_i$ of the variable $x^i$, the quantity
  \[\frac{s^{11} s^{21} s^{31}}{\det(S)}\]
  is divided by $g_i$.
  Thus, recalling the form of the Robertson condition (\ref{Rob1}), we can always assume that $f_2 = f_3 = 1$ by choosing appropriate coordinates on $\mathcal{T}^2$.
  However, we do not divide the first line by $f_1$ because it changes the description of the hyperbolic structure (i.e. the condition (\ref{condhyp})). Nevertheless, it remains a degree of 
  freedom on the first line. For instance, we can divide the first line by $s_{12}$ or $s_{13}$ and we then obtain that the radial part depends only on the two scalar functions 
  $\frac{s_{11}}{s_{13}}$ and $\frac{s_{12}}{s_{13}}$. As we will see at the end of Section \ref{invrad}, these quotients are exactly the scalar functions we recover in our study 
  of the radial part. However, since it does not simplify our study we do not use this reduction for the moment.
  \item From now on, $f_2=1$ and $f_3 = 1$ and we can thus rewrite (\ref{Rob3}) into the form
  \[s_{23}-s_{22} = 1 \quad \textrm{and} \quad s_{32}-s_{33} = 1.\]
  Thanks to these equalities, we can also write
  \[d\Omega^2_{\mathcal{T}^2} = s^{11}( (dx^2)^2 + (dx^3)^2),\]
  for the induced metric on the compactified boundary $\{ x^1 = 0 \}$.
  \item Generally, we know that $s_{12}$ and $s_{13}$ tend to $1$ when $x^1$ tends to $0$ but we do not know that this is also true when $x^1$ tends 
  to $A$. However, the St\"ackel structure allows us to show that the asymptotically hyperbolic structure has to be the same at the two ends (under a mild additional assumption). 
  Assume that the behaviour of the first line at the two ends is the following: at the end $\{x^1 = 0\}$
  \[ (x^{1})^2 s_{11}(x^1) = [1]_{\epsilon_0}, \quad s_{12}(x^1) = [1]_{\epsilon_0} \quad \textrm{and} \quad s_{13}(x^1) = [1]_{\epsilon_0},\]
  and at the end $\{x^1 = A\}$
  \[ (A-x^{1})^2 s_{11}(x^1) = [1]_{\epsilon_1}, \quad s_{12}(x^1) = \alpha [1]_{\epsilon_1} \quad \textrm{and} \quad s_{13}(x^1) = \beta [1]_{\epsilon_1},\]
where,
\[ [1]_{\epsilon_0} = 1 + O((1 + |\log(x^1)|)^{-1-\epsilon_0}) \quad \textrm{and} \quad [1]_{\epsilon_1} = 1 + O((1 + |\log(A-x^1)|)^{-1-\epsilon_1})\]
and $\alpha$ and $\beta$ are real positive constants. Using the Robertson condition at the end $\{x^1 = 0\}$ and the end $\{x^1 = A \}$ we obtain
\[ 1 = f_2 = s_{23}-s_{22} = \alpha s_{23}- \beta s_{22},\]
and
\[ 1 = f_3 = s_{32}-s_{33} = \beta s_{32}- \alpha s_{33}.\]
Thus, using that $s_{23} = 1 + s_{22}$ and $s_{32} = 1 + s_{33}$, we obtain
\[ (\alpha - \beta) s_{22} = 1-\alpha \quad \textrm{and} \quad (\beta - \alpha) s_{33} = 1-\beta.\]
Hence, if we assume that $s_{22}$ or $s_{33}$ are not constant functions, we obtain $\alpha = \beta = 1$. 
 \end{enumerate}
\end{remark}

Finally, we now give three examples of St\"ackel manifolds.

\begin{exem}\label{exemple}
We give here three examples of St\"ackel manifolds which illustrate the diversity of the manifolds we consider.
 \begin{enumerate}
  \item We can first choose the St\"ackel matrix
  \[ S =  \begin{pmatrix} s_{11}(x^1) & s_{12}(x^1) & s_{13}(x^1) \\ a & b & c \\ d & e & f \end{pmatrix},\]
  where $s_{11}$, $s_{12}$ and $s_{13}$ are smooth functions of $x^1$ and $a$, $b$, $c$, $d$, $e$ and $f$ are real constants. The metric $g$ can thus be written as
  \[ g = \sum_{i=1}^{3} H_i^2 (dx^i)^2,\]
  where $H_i^2$, for $i \in \{1,2,3\}$, are functions of $x^1$ alone. Therefore, $g$ trivially satisfies the Robertson condition and we can add the asymptotically hyperbolic 
  structure given in Definition \ref{defAH}. We note that, as explained in the previous 
  Remark, $g$ depends only on two arbitrary functions (after a Liouville change of variables in the variable $x^1$). Moreover, we can show that $\partial_{x^2}$ and 
  $\partial_{x^3}$ are Killing vector fields and the existence of these Killing vector fields traduces the symmetries with respect to the translation in $x^2$ and $x^3$.
  \item We can also choose the St\"ackel matrix
  \[ S =  \begin{pmatrix} s_{11}(x^1) & s_{12}(x^1) & a s_{12}(x^1) \\ 0 & s_{22}(x^2) & s_{23}(x^2) \\ 0 & s_{32}(x^3) & s_{33}(x^3) \end{pmatrix},\]
  where $s_{11}$ and $s_{12}$ are smooth functions of $x^1$, $s_{22}$ and $s_{23}$ are smooth functions of $x^2$, $s_{32}$ and $s_{33}$ are smooth functions of $x^3$ and 
  $a$ is a real constant. We can add the asymptotically hyperbolic structure given in Definition \ref{defAH} and the metric $g$ can be written as
  \[ g = s_{11}(dx^1)^2 + \frac{s_{11}}{s_{12}} \left( \frac{s^{11}}{as_{32}-s_{33}} (dx^2)^2 + \frac{s^{11}}{s_{23}-as_{22}} (dx^3)^2 \right).\]
  Therefore, $g$ satisfied the Robertson condition. We note that, after Liouville transformations in the three variables, $g$ depends on three arbitrary functions.
  Moreover, thanks to the Liouville transformation
  \[ X^1 = \int_0^{x^1} \sqrt{s_{11}(s)} \, ds,\]
  we see that there exists a system of coordinates in which the metric $g$ takes the form
  \[ g = (dx^1)^2 + f(x^1) g_0,\]
  where $g_0$ is a metric on the $2$-torus $\mathcal{T}^2$. In other words, $g$ is a warped product. In particular, $g$ is conformal to a metric which can be written as a sum of 
  one euclidean direction and a metric on a compact manifold. We recall that in this case, under some additional assumptions on the compact part, the uniqueness of the 
  anisotropic Calder\'on problem on compact manifolds with boundary has been proved in \cite{DKSU,DKLS}.
  \item At last, we can choose the St\"ackel matrix
  \[ S =  \begin{pmatrix} s_{1}(x^1)^2 & -s_{1}(x^1) & 1 \\ -s_2(x^2)^2 & s_{2}(x^2) & -1 \\ s_3(x^3)^2 & -s_{3}(x^3) & 1 \end{pmatrix},\]
  where $s_{1}$ is a smooth function of $x^1$, $s_{2}$ is a smooth function of $x^2$ and $s_{3}$ is a smooth function of $x^3$. This model was studied in \cite{Benen,BM} and 
  is of main interest in the field of geodesically equivalent riemannian manifolds i.e. of manifold which share the same unparametrized geodesics 
  (see \cite{BM}). The associated metric
  \[ g = (s_1-s_2)(s_1-s_3) (dx^1)^2 + (s_2-s_3)(s_1-s_2)(dx^2)^2 + (s_3-s_2)(s_3-s_1)(dx^3)^2,\]
  satisfies the Robertson condition and $g$ has, a priori, no symmetry, is not a warped product and depends on three arbitrary functions that satisfy $s_1 > s_2 > s_3$.
  To put an asymptotically 
  hyperbolic structure in the sense given in Definition \ref{defAH} we first multiply the second and the third column of the St\"ackel matrix on the right by the invertible
  matrix
  \[ G =  \begin{pmatrix} -1 & -1 \\ 0 & -1 \end{pmatrix},\]
  since it does not change the metric. We thus obtain the new St\"ackel matrix
  \[ \begin{pmatrix} s_{1}(x^1)^2 & s_{1}(x^1) & s_{1}(x^1) -1 \\ -s_2(x^2)^2 & -s_{2}(x^2) & -s_{2}(x^1) + 1 \\ s_3(x^3)^2 & s_{3}(x^3) & s_{3}(x^1) -1 \end{pmatrix}.\]
  In a second time, we use the Liouville change of variables in the first variable
  \[ X^1 = \int_0^{x^1} \sqrt{s_{1}(s)} \, ds,\]
  and we obtain the St\"ackel matrix
  \[ S =  \begin{pmatrix} s_{1}(X^1) & 1 & 1 - \frac{1}{s_1(X^1)} \\ -s_2(x^2)^2 & -s_{2}(x^2) & -s_{2}(x^1) + 1 \\ s_3(x^3)^2 & s_{3}(x^3) & s_{3}(x^1) -1 \end{pmatrix}.\]
  Finally, to put the asymptotically hyperbolic structure on the first line, we assume that
  \[ s_{1}(X^1) =  \frac{1}{(X^1)^2} (1 + O((1 + |\log(X^1)|)^{-1-\epsilon_0})), \quad \textrm{when} \quad X^1 \to 0\]
  and
  \[ s_{1}(X^1) =  \frac{1}{(A^1-X^1)^2} (1 + O((1 + |\log(A^1-X^1)|)^{-1-\epsilon_1})), \quad \textrm{when} \quad X^1 \to A^1,\]
  where $A^1 = \int_0^{A} \sqrt{s_{1}(s)} \, ds$.
 \end{enumerate}
\end{exem}

\subsection{Scattering operator and statement of the main result}

We recall here the construction of the scattering operator given in \cite{IK,IKL} for asymptotically hyperbolic manifolds. This construction has been used in
\cite{DaKaNi} in the case of asymptotically hyperbolic Liouville surfaces. Roughly speaking, in the neighbourhood of the asymptotically hyperbolic
ends we can compare the global dynamics with a simpler comparison dynamics, i.e. we establish the existence and the asymptotically completeness of the wave operators
\[ W_k^{\pm} = s-\lim e^{itH} J_k e^{-itH_0^k},\]
where $J_k$ is a cutoff function that isolate the $k^{\mathrm{th}}$ asymptotically hyperbolic end and $H_0^k$ is a simpler hamiltonian which governs the free wave dynamics
in this end. The scattering operator $S_g$ is then defined by
\[ S_g = (W^+)^{\star} W^-, \quad \textrm{where} \quad W^{\pm} = \sum_k W_k^{\pm}.\]
This operator makes the link between the asymptotic (scattering) data in the past and the asymptotic (scattering) data in the future. The scattering matrix $S_g(\lambda)$ is then the 
restriction of the scattering operator $S_g$ on an energy level $\lambda^2$.
This corresponds to the time-dependent approach to scattering theory.
There is also an equivalent stationary definition.
To define the scattering matrix in a stationary way, we take the Fourier transform of the wave equation $\partial_t^2 u - \Delta_g u = 0$ with respect to $t$, and instead of 
studying the asymptotic behaviour of $u(t,x)$ at late times, we study the spatial asymptotic behaviour of solutions of the Helmholtz equation $-\Delta_g u = \lambda^2 u$ 
as $|x| \to \infty$.
We thus obtain an other but equivalent definition of the scattering matrix at energy $\lambda^2$ thanks to the Helmholtz equation (see Theorem \ref{defscat}).

In our particular model, there are two ends and so we introduce two cutoff functions $\chi_0$ and $\chi_1$, smooth on $\mathbb{R}$, defined by
\begin{equation}\label{defchi}
 \chi_0 = 1 \quad \textrm{on} \quad \left(0,\frac{A}{4}\right), \quad \chi_1 = 1 \quad \textrm{on} \quad \left(\frac{3A}{4},A\right), \quad \chi_0 + \chi_1 = 1 \quad \textrm{on}\quad (0,A),
\end{equation}
in order to separate these two ends. We consider the shifted stationary Helmholtz equation
\[  -(\Delta_g +1) f = \lambda^2 f,\]
where $\lambda^2 \neq 0$ is a fixed energy, which is usually studied in case of asymptotically hyperbolic manifolds (see \cite{Bort,IK,IKL,JSB}). Indeed, it is known (see \cite{IKL})
that the essential spectrum of $- \Delta_g$ is $[1,+\infty)$ and thus, we shift the bottom of the essential 
spectrum in order that it becomes $0$. It is known that the operator $- \Delta_g -1$ has no eigenvalues embedded into the essential spectrum $[0,+\infty)$ (see 
\cite{Bou,IK,IKL}).
It is shown in \cite{IKL} that the solutions of the shifted stationary equation 
\[-(\Delta_g +1) f = \lambda^2 f,\]
  are unique when we impose on $f$ some radiation conditions at 
infinities. Precisely, as in \cite{DaKaNi}, we define some Besov spaces that encode these radiation conditions at infinities as follows. To motivate our definitions,
we first recall that the compactified boundaries $\{x^1 = 0\}$ and $\{x^1 = A\}$ are endowed with the induced metric
  \[d\Omega^2_{\mathcal{T}^2} = s^{11}( (dx^2)^2 + (dx^3)^2).\]

\begin{definition}
 Let $\mathcal{H}_{\mathcal{T}^2} = L^2(\mathcal{T}^2,s^{11} dx^2dx^3)$. Let the intervals $(0,+\infty)$ and $(-\infty,A)$ be decomposed as
 \[(0,+\infty) = \cup_{k \in \mathbb{Z}} I_k \quad \textrm{and} \quad (-\infty,A) = \cup_{k \in \mathbb{Z}} J_k,\]
 where
 \[I_k = 
   \begin{cases}
(\exp(e^{k-1}),\exp(e^k)] \quad \textrm{if} \quad k \geq 1\\
(e^{-1},e] \quad \textrm{if} \quad k = 0 \\
(\exp(-e^{|k|}),\exp(-e^{|k|-1})] \quad \textrm{if} \quad k \leq -1\\
\end{cases}
 \]
 and
  \[J_k = 
   \begin{cases}
(A-\exp(e^{k-1}),A-\exp(e^k)] \quad \textrm{if} \quad k \geq 1\\
(A-e^{-1},A-e] \quad \textrm{if} \quad k = 0 \\
(A-\exp(-e^{|k|}),A-\exp(-e^{|k|-1})] \quad \textrm{if} \quad k \leq -1\\
\end{cases}.
 \]
We define the Besov spaces $\mathcal{B}_0 = \mathcal{B}_0(\mathcal{H}_{\mathcal{T}^2})$ and $\mathcal{B}_1 = \mathcal{B}_1(\mathcal{H}_{\mathcal{T}^2})$ to be the Banach spaces of $\mathcal{H}_{\mathcal{T}^2}$-valued 
functions on $(0,+\infty)$ and $(-\infty,A)$ satisfying respectively
\[ \Vert f \Vert_{\mathcal{B}_0} = \sum_{k \in \mathbb{Z}} e^{\frac{|k|}{2}} \left( \int_{I_k} \Vert f(x) \Vert_{\mathcal{H}_{\mathcal{T}^2}}^2 \, \frac{dx}{x^2} \right)^{\frac{1}{2}} < \infty\]
and
\[ \Vert f \Vert_{\mathcal{B}_1} = \sum_{k \in \mathbb{Z}} e^{\frac{|k|}{2}} \left( \int_{J_k} \Vert f(x) \Vert_{\mathcal{H}_{\mathcal{T}^2}}^2 \, \frac{dx}{(A-x)^2} \right)^{\frac{1}{2}} < \infty.\]
The dual spaces $\mathcal{B}_0^{\star}$ and $\mathcal{B}_1^{\star}$ are then identified with the spaces equipped with the norms
\[ \Vert f \Vert_{\mathcal{B}_0^{\star}} = \left( \sup_{R>e} \frac{1}{\log(R)}\int_{\frac{1}{R}}^R \Vert f(x) \Vert_{\mathcal{H}_{\mathcal{T}^2}}^2 \, \frac{dx}{x^2} \right)^{\frac{1}{2}} < \infty\]
and
\[ \Vert f \Vert_{\mathcal{B}_1^{\star}} =  \left( \sup_{R>e} \frac{1}{\log(R)}\int_{A-R}^{A-\frac{1}{R}} \Vert f(x) \Vert_{\mathcal{H}_{\mathcal{T}^2}}^2 \, \frac{dx}{(A-x)^2} \right)^{\frac{1}{2}} < \infty.\]
\end{definition}

\begin{remark}
 As shown in \cite{IK}, we can compare the Besov spaces $\mathcal{B}_0$ and $\mathcal{B}_0^{\star}$ to weighted $L^2$-spaces. Indeed, if we define $L_0^{2,s}((0,+\infty),\mathcal{H}_{\mathcal{T}^2})$ for $s \in \mathbb{R}$ by
 \[ \Vert f \Vert_{s} = \left( \int_{0}^{+\infty} (1+ |\log(x)|)^{2s} \Vert f(x) \Vert_{\mathcal{H}_{\mathcal{T}^2}}^2 \, \frac{dx}{x^2} \right)^{\frac{1}{2}} < \infty,\]
 then for $s > \frac{1}{2}$,
 \[ L_0^{2,s} \subset \mathcal{B}_0 \subset L_0^{2,\frac{1}{2}} \subset L_0^2 \subset L_0^{2,-\frac{1}{2}} \subset \mathcal{B}_0^{\star} \subset L_0^{2,-s}.\]
 There is a similar result for the Besov spaces $\mathcal{B}_1$ and $\mathcal{B}_1^{\star}$.
\end{remark}

\begin{definition}
 We define the Besov spaces $\mathcal{B}$ and $\mathcal{B}^{\star}$ as the Banach spaces of $\mathcal{H}_{\mathcal{T}^2}$-valued functions on $(0,A)$ with norms
 \[ \Vert f \Vert_{\mathcal{B}} = \Vert \chi_0 f \Vert_{\mathcal{B}_0} + \Vert \chi_1 f \Vert_{\mathcal{B}_1}\]
 and
  \[ \Vert f \Vert_{\mathcal{B}^{\star}} = \Vert \chi_0 f \Vert_{\mathcal{B}_0^{\star}} + \Vert \chi_1 f \Vert_{\mathcal{B}_1^{\star}}.\]
  We also define the Hilbert space of scattering data
  \[\mathcal{H}_{\infty} = \mathcal{H}_{\mathcal{T}^2} \otimes \mathbb{C}^2 \simeq \mathcal{H}_{\mathcal{T}^2} \oplus \mathcal{H}_{\mathcal{T}^2}.\]
\end{definition}
\noindent
In \cite{IK} (see Theorem 3.15) the following theorem is proved.

\begin{theorem}[Stationary construction of the scattering matrix]
 \begin{enumerate}
  \item For any solution $f \in \mathcal{B}^{\star}$ of the shifted stationary Helmholtz equation at non-zero energy $\lambda^2$ 
  \begin{equation}\label{Hel1}
  -(\Delta_g +1) f = \lambda^2 f,
\end{equation}
  there exists a unique 
  $\psi^{(\pm)} = (\psi_0^{(\pm)},\psi_{1}^{(\pm)}) \in \mathcal{H}_{\infty}$ such that
  \begin{align}\label{scat}
   f \simeq& \quad \omega_-(\lambda) \left( \chi_0 \, \, (x^1)^{\frac{1}{2}+i\lambda} \psi_0^{(-)} + \chi_1 \, \, (A-x^1)^{\frac{1}{2}+i\lambda} \psi_1^{(-)} \right) \nonumber\\
    &- \omega_+(\lambda) \left( \chi_0 \, \, (x^1)^{\frac{1}{2}-i\lambda} \psi_0^{(+)} + \chi_1 \, \, (A-x^1)^{\frac{1}{2}-i\lambda} \psi_1^{(+)} \right),
  \end{align}
  where
    \begin{equation}\label{omega}
 \omega_{\pm}(\lambda) = \frac{\pi}{(2 \lambda \sinh(\pi \lambda))^{\frac{1}{2}} \Gamma(1 \mp i \lambda)}.
  \end{equation}
 \item For any $\psi^{(-)} \in \mathcal{H}_{\infty}$, there exists a unique $\psi^{(+)} \in \mathcal{H}_{\infty}$ and $f \in \mathcal{B}^{\star}$ satisfying (\ref{Hel1})
 for which the decomposition (\ref{scat}) above holds. This defines uniquely the scattering operator $S_g(\lambda)$ as the $\mathcal{H}_{\infty}$-valued operator such that for all 
 $\psi^{(-)} \in \mathcal{H}_{\infty}$,
 \begin{equation}\label{scat2}
  \psi^{(+)} = S_g(\lambda) \psi^{(-)}.
 \end{equation}
 \item The scattering operator $S_g(\lambda)$ is unitary on $\mathcal{H}_{\infty}$.  
 \end{enumerate}\label{defscat}
\end{theorem}
\noindent
Note that in our model with two asymptotically hyperbolic ends the scattering operator has the structure of a $2 \times 2$ matrix whose components are $\mathcal{H}_{\mathcal{T}^2}$-valued operators. Precisely, we write
\[S_g(\lambda) = \begin{pmatrix}
   L(\lambda) & T_R(\lambda) \\
   T_L(\lambda) & R(\lambda)
  \end{pmatrix},\]
where $T_L(\lambda)$ and $T_R(\lambda)$ are the transmission operators and $L(\lambda)$ and $R(\lambda)$ are the reflection operators from the right and from the left
respectively. The transmission operators measure what is transmitted from one end to the other end in a scattering experiment, while 
the reflection operators measure the part of a signal sent from one end that is reflected to itself.

The main result of this paper is the following:

\begin{theorem}\label{main}
 Let $(\mathcal{M},g)$ and $(\mathcal{M},\tilde{g})$, where $\mathcal{M} = (0,A)_{x^1} \times \mathcal{T}_{x^2,x^3}^2$, be two three-dimensional 
 St\"ackel toric cylinders, i.e. endowed with the metrics $g$ and $\tilde{g}$ defined in (\ref{met}) respectively. We assume that these manifolds satisfy the Robertson condition and are endowed with 
 asymptotically hyperbolic structures at the two ends $\{x^1 = 0 \}$ and $\{x^1 = A\}$ as defined in Definition \ref{defAH}. We denote by $S_g(\lambda)$ and $S_{\tilde{g}}(\lambda)$ the corresponding scattering operators at a
 fixed energy $\lambda \neq 0$ as defined in Theorem \ref{defscat}. Assume that
 \[ S_g(\lambda) = S_{\tilde{g}}(\lambda).\]
 Then, there exists a diffeomorphism $\Psi : \mathcal{M} \to \mathcal{M}$, equals to the identity at the compactified ends $\{x^1 = 0\}$ and $\{x^1 = A\}$, such that
 $\tilde{g}$ is the pull back of $g$ by $\Psi$, i.e.
 \[ \tilde{g} = \Psi^{\star} g.\]
\end{theorem}


For general Asymptotically Hyperbolic Manifolds (AHM in short) with no particular (hidden) symmetry, direct and inverse scattering results for scalar waves have been 
proved by Joshi and S\'a Barreto in \cite{JSB}, by S\'a Barreto in \cite{SB}, by Guillarmou and S\'a Barreto in \cite{GuiSB,GSB} and by Isozaki and Kurylev in \cite{IK}. In 
\cite{JSB}, it is shown that the asymptotics of the metric of an AHM are uniquely determined (up to isometries) by the scattering matrix $S_g(\lambda)$ at a 
fixed energy $\lambda$ off a discrete subset of $\R$. In \cite{SB}, it is proved that the metric of an AHM is uniquely determined (up to isometries) by the 
scattering matrix $S_g(\lambda)$ for every $\lambda \in \R$ off an exceptional subset. Similar results are obtained recently in \cite{IK} for even more general classes of AHM. 
In \cite{GuiSB}, it is proved that, for connected conformally compact Einstein manifolds of even dimension $n+1$, the scattering matrix at energy $n$ on an open subset of 
its conformal boundary determines the manifold up to isometries.
In \cite{GSB}, the authors study direct and inverse scattering problems for asymptotically complex hyperbolic manifolds and show that the topology and the metric of such a 
manifold are determined (up to invariants) by the scattering matrix at all energies.
We also mention the work \cite{Mara} of Marazzi in which the author study inverse scattering for the stationary Schr\"odinger equation with smooth potential not vanishing 
at the boundary on a conformally compact manifold with sectional curvature $-\alpha^2$ at the boundary. The author then shows that the scattering matrix at two
fixed energies $\lambda_1$ and $\lambda_2$, $\lambda_1 \neq \lambda_2$, in a suitable subset of $\C$, determines $\alpha$ and the Taylor series of both the potential and the metric at the boundary.
At last, we also mention \cite{BP} where related inverse problems - inverse resonance problems - are studied in certain subclasses of AHM.

This work must also be put into perspective with the anisotropic Calder\'on problem on compact manifolds with boundary. We recall here, the definition
of this problem. Let $(\mathcal{M},g)$ be a riemannian compact manifold with smooth boundary $\partial \mathcal{M}$. We denote by $-\Delta_g$ the Laplace-Beltrami operator on 
$(\mathcal{M},g)$ and we recall that this operator with Dirichlet boundary conditions is selfadjoint on $L^2(\mathcal{M},dVol_g)$ and has a pure point spectrum $\{ \lambda_j^2 \}_{j \geq 1}$.
We are interested in the solutions $u$ of
\begin{equation}\label{Dirichlet}
 \begin{cases}
-\Delta_g u = \lambda^2 u, \quad \textrm{on} \quad \mathcal{M},\\
\quad \quad u = \psi , \quad \textrm{on} \quad \partial \mathcal{M}.\\
\end{cases}
\end{equation}
It is known (see for instance \cite{S}) that for any $\psi \in H^{\frac{1}{2}}(\partial \mathcal{M})$ there exits a unique weak solution $u \in H^1(\mathcal{M})$ 
of (\ref{Dirichlet}) when $\lambda^2$ does not belong to the Dirichlet spectrum $\{ \lambda_i^2 \}$ of $-\Delta_g$. This allows us to define the Dirichlet-to-Neumann (DN)
map as the operator $\Lambda_g(\lambda^2)$ from $H^{\frac{1}{2}}(\partial \mathcal{M})$ to 
$H^{-\frac{1}{2}}(\partial \mathcal{M})$ defined for all $\psi \in H^{\frac{1}{2}}(\partial \mathcal{M})$ by
\[ \Lambda_g(\lambda^2) (\psi) = (\partial_{\nu} u)_{|\partial \mathcal{M}},\]
where $u$ is the unique solution of (\ref{Dirichlet}) and $(\partial_{\nu} u)_{|\partial \mathcal{M}}$ is its normal derivative with respect to the unit outer normal vector 
$\nu$ on $\partial \mathcal{M}$. The anisotropic Calder\'on problem can be stated as:
\begin{center}
\emph{Does the knowledge of the DN map $\Lambda_g(\lambda^2)$ at a frequency $\lambda^2$ determine uniquely the metric $g$?}
\end{center}
We refer for instance to \cite{DKSU,DKLS,GSB,GT2,KS2,LTU,LU,LeeU} for important contributions to this subject and to the surveys
\cite{GT,KS,S,U} for the current state of the art.

In dimension two, the anisotropic Calder\'on problem with $\lambda^2 = 0$ was shown to be true for smooth connected riemannian surfaces in \cite{LU,LeeU}. A positive answer for 
zero frequency $\lambda^2 = 0$ in dimension $3$ or higher has been given for compact connected real analytic riemannian manifolds with real analytic boundary first in
\cite{LeeU} under some topological assumptions relaxed later in \cite{LTU,LU} and for compact connected Einstein manifolds with boundary in \cite{GuiSB}. The general 
anisotropic Calder\'on problem in dimension $3$ or higher remains a major open problem. Results have been obtained recently in \cite{DKSU,DKLS} for some classes of smooth 
compact riemannian manifolds with boundary that are conformally transversally anisotropic, i.e. riemannian manifolds $(\mathcal{M},g)$ such that
\[ \mathcal{M} \subset \subset \R \times \mathcal{M}_0, \quad g = c(e \oplus g_0),\]
where $(\mathcal{M}_0,g_0)$ is a $n-1$ dimensional smooth compact riemannian manifold with boundary, $e$ is the euclidean metric on the real line, and $c$ is a smooth positive 
function on the cylinder $\R \times \mathcal{M}_0$. Under some conditions on the transverse manifold $(\mathcal{M}_0,g_0)$ (such as simplicity), the riemannian manifold 
$(\mathcal{M},g)$ is said to be admissible. In that framework, the authors of \cite{DKSU,DKLS} were able to determine uniquely the conformal factor $c$ from the knowledge of 
the DN map at zero frequency $\lambda^2 = 0$. One of the aim of this paper is thus to give an example of manifolds on which we can solve the inverse scattering problem at
fixed energy but do not have one of the particular structures we just described before for which the uniqueness for the anisotropic Calder\'on problem on compact manifolds with
boundary is known (see Example \ref{exemple}, 3)).

\subsection{Overview of the proof}

The proof of Theorem \ref{main} is divided into four steps which we describe here.\\

\noindent
\underline{Step 1:} The first step of the proof consists in solving the direct problem. This will be done in Section \ref{dirpb}. In this Section we first use
the structure of St\"ackel manifold satisfying the Robertson condition to proceed to the separation of variables for 
the Helmholtz equation. We obtain that the shifted Helmholtz equation
\[ -(\Delta_g + 1)f = \lambda^2 f,\]
can be rewritten as
\[ A_1 f + s_{12} L f + s_{13} H f = 0,\]
where $A_1$ is a differential operator in the variable $x^1$ alone and $L$ and $H$ are commuting, elliptic, semibounded selfadjoint operators on 
$L^2(\mathcal{T}^2,s^{11} dx^2 dx^3)$ that only depend on the variables $x^2$ and $x^3$.
 Since the operators $L$ and $H$ commute, there exists a common Hilbertian basis of eigenfunctions for $H$ and $L$. Moreover, the ellipticity property on a 
 compact manifold shows that the spectrum is discrete and the selfadjointness proves that the spectrum is real. We consider generalized harmonics $\{ Y_m \}_{m \geq 1}$ which form a
 Hilbertian basis of $L^2(\mathcal{T}^2, s^{11} dx^2 dx^3)$ associated with the coupled spectrum $(\mu_m^2,\nu_m^2)$ of $(H,L)$.
We decompose the solutions $f = \displaystyle{\sum_{m \geq 1} u_m(x^1)Y_m(x^2,x^3)}$ of the Helmholtz equation on the common basis of harmonics $\{Y_m\}_{m \geq 1}$ and we then conclude that the Helmholtz equation 
separates into a system of three ordinary differential equations:
 \[\begin{cases}
-u_m''(x^1) + \frac{1}{2} (\log(f_1)(x^1))' u_m'(x^1) + [-(\lambda^2 +1)s_{11}(x^1) + \mu_m^2 s_{12}(x^1) + \nu_m^2 s_{13}(x^1)]u_m(x^1) = 0\\
-v_m''(x^2) + [-(\lambda^2 +1)s_{21}(x^2) + \mu_m^2 s_{22}(x^2) + \nu_m^2 s_{23}(x^2)]v_m(x^2) = 0\\
-w_m''(x^3) + [-(\lambda^2 +1)s_{31}(x^3) + \mu_m^2 s_{32}(x^3) + \nu_m^2 s_{33}(x^3)]w_m(x^3) = 0
\end{cases},\]
where $f_1$ is the function appearing in the Robertson condition and $Y_m(x^2,x^3) = v_m(x^2) w_m(x^3)$.
In this system of ODEs there is one ODE in the radial variable $x^1$ and two 
ODEs in the angular variables $x^2$ and $x^3$.
We emphasize that the angular momenta $\mu_m^2$ and $\nu_m^2$ which are the separation constants correspond also to the coupled spectrum of the two angular operators $H$ and $L$.
The fact that the angular momenta $(\mu_m^2,\nu_m^2)$ are coupled has an important consequence in the use of the Complexification of 
the Angular Momentum method. Indeed, we cannot work separately with one angular momentum and we thus have to use a multivariable version of this method.

In a second time, we define the characteristic and Weyl-Titchmarsh functions following the construction 
given in \cite{DaKaNi,FY,KST}. We briefly recall here the definition of these objects and the reason why we use them.
Using a Liouville change of variables $X^1 = g(x^1)$, $X^1 \in (0,A^1)$ where $A^1 = \int_{0}^{A} g(x^1) dx^1$, we can write the radial equation as
\begin{equation}\label{eqSL}
 - \ddot{U} + q_{\nu_m^2} U = -\mu_m^2 U, 
\end{equation}
where $-\mu_m^2$ is now the spectral parameter and $q_{\nu_m^2}$ satisfies at the end $\{X^1 = 0\}$,
 \[q_{\nu_m^2}(X^1,\lambda) = - \frac{\lambda^2 + \frac{1}{4}}{(X^1)^2} + q_{0,\nu_m^2}(X^1,\lambda),\]
where $X^1 q_{0,\nu_m^2}(X^1,\lambda)$ is integrable at the end $\{X^1 = 0\}$ (the potential $q_{\nu_m^2}$ also has the same property at the other end).
We are thus in the framework of \cite{FY}.
We can then define the characteristic and Weyl-Titchmarsh functions associated with this singular non-selfadjoint Schr\"odinger equation. To do this, we follow the method given in 
\cite{DaKaNi}. We thus define two fundamental systems of solutions $\{ S_{10},S_{20} \}$ and $\{ S_{11},S_{21} \}$ defined by
\begin{enumerate}
 \item When $X^1 \to 0$,
 \[ S_{10}(X^1,\mu^2,\nu^2) \sim (X^1)^{\frac{1}{2}-i\lambda} \quad \textrm{and} \quad S_{20}(X^1,\mu^2,\nu^2) \sim \frac{1}{2i\lambda} (X^1)^{\frac{1}{2}+i\lambda} \]
 and when $X^1 \to A^1$,
\[ S_{11}(X^1,\mu^2,\nu^2) \sim (A^1-X^1)^{\frac{1}{2}-i\lambda} \quad \textrm{and} \quad S_{21}(X^1,\mu^2,\nu^2) \sim -\frac{1}{2i\lambda} (A^1-X^1)^{\frac{1}{2}+i\lambda}.  \]
 \item $W(S_{1n},S_{2n}) = 1$ for $n \in \{0,1\}$.
 \item For all $X^1 \in (0,A^1)$, $\mu \mapsto S_{jn}(X^1,\mu^2,\nu^2)$ is an entire function for $j \in \{1,2\}$ and $n \in \{0,1\}$.
\end{enumerate}
We add some singular separated boundary conditions at the two ends (see (\ref{condbordsing})) and we consider 
the new radial equation as an eigenvalue problem. Finally, we define the two characteristic functions of this radial equation as Wronskians of functions of the fundamental 
systems of solutions:
\[\Delta_{q_{\nu_m^2}}(\mu_m^2) = W(S_{11},S_{10}) \]
and
\[\delta_{q_{\nu_m^2}}(\mu_m^2) = W(S_{11},S_{20})\]
and we also define the Weyl-Titchmarsh function by:
\begin{equation}\label{deffonctionWTintro}
M_{q_{\nu_m^2}}(\mu_m^2) = -\frac{\delta_{q_{\nu_m^2}}(\mu_m^2)}{\Delta_{q_{\nu_m^2}}(\mu_m^2)}.
\end{equation}
The above definition generalizes the usual definition of classical Weyl-Titchmarsh functions for regular Sturm-Liouville differential operators. We refer to \cite{KST} 
for the theory of selfadjoint singular Sturm-Liouville operators and the definition and main properties of Weyl-Titchmarsh functions.
In our case the 
boundary conditions make the Sturm-Liouville equation non-selfadjoint. The generalized Weyl-Titchmarsh function can nevertheless be defined by the same recipe as 
shown in \cite{DaKaNi,FY} and recalled above. Our interest in considering the generalized Weyl-Titchmarsh function $M_{q_{\nu_m^2}}(\mu_m^2)$ comes from the fact that it is a powerful 
tool to prove uniqueness results for one-dimensional inverse problems.
Indeed, roughly speaking, the Borg-Marchenko Theorem states (see \cite{KST}) that if $M_q$ and $M_{\tilde{q}}$
are two generalized Weyl-Titchmarsh functions associated with the equations
\[ -u'' + q(x)u = -\mu^2 u \quad \textrm{and} \quad -u'' + \tilde{q}(x)u = -\mu^2 u,\]
where $q$ and $\tilde{q}$ satisfy the previous quadratic singularities at the ends, then if
\begin{equation}\label{BorgMar1}
 M_q(\mu^2) = M_{\tilde{q}}(\mu^2), \quad \mu \in \C \setminus \{\textrm{poles}\}, 
\end{equation}
then
\begin{equation}\label{BorgMar2}
 q = \tilde{q}.
\end{equation}
We refer to \cite{Be,GS,Te} for results in the case of 
regular Weyl-Titchmarsh functions and to the recent results \cite{FY,KST} in the case of singular Weyl-Titchmarsh functions corresponding to possibly non-selfadjoint 
equation.

We note that, the characteristic and generalized Weyl-Titchmarsh functions obtained for each one-dimensional equation (\ref{eqSL}) can be summed up over the span of each of 
the harmonics $Y_m$, $m \geq 1$, in order to define operators from $L^2(\mathcal{T}^2, s^{11} dx^2 dx^3)$ onto itself. Precisely, recalling that
\[  L^2(\mathcal{T}^2,s^{11}dx^2dx^3) = \displaystyle{\bigoplus_{m \geq 1}} \langle Y_m \rangle,\]
we define:

\begin{definition}
 Let $\lambda \neq 0$ be a fixed energy. The characteristic operator $\Delta(\lambda)$ and the generalized Weyl-Titchmarsh operator $M(\lambda)$ are defined as operators 
 from $L^2(\mathcal{T}^2, s^{11} dx^2 dx^3)$ onto itself that are diagonalizable on the Hilbert basis of eigenfunctions $\{Y_m\}_{m \geq 1}$ associated 
 with the eigenvalues $\Delta_{q_{\nu_m^2}}(\mu_m^2)$ and $M_{q_{\nu_m^2}}(\mu_m^2)$. More precisely, for all $v \in L^2(\mathcal{T}^2, s^{11} dx^2 dx^3)$, $v$ can be decomposed as
 \[ v = \sum_{m \geq 1} v_m Y_m, \quad v_m \in \C\]
 and we have
 \[ \Delta(\lambda)v = \sum_{m \geq 1} \Delta_{q_{\nu_m^2}}(\mu_m^2) v_m Y_m \quad \textrm{and} \quad M(\lambda)v = \sum_{m \geq 1} M_{q_{\nu_m^2}}(\mu_m^2) v_m Y_m.\]
\end{definition}

We emphasize that the separation of the variables allows us to ``diagonalize'' the reflection and the transmission operators into a countable family of 
multiplication operators by numbers $R_g(\lambda,\mu_m^2,\nu_m^2)$, $L_g(\lambda,\mu_m^2,\nu_m^2)$ and $T_g(\lambda,\mu_m^2,\nu_m^2)$ called reflection and transmission
coefficients respectively. We will show (see Equations (\ref{lienLM})-(\ref{lienRM})) that the characteristic 
and Weyl-Titchmarsh functions are nothing but the transmission and the reflection coefficients respectively. The aim of this identification is to use the Borg-Marchenko 
theorem from the equality of the scattering matrix at fixed energy.\\

\noindent
\underline{Step 2:} The second step of the proof consists in solving the inverse problem for the angular part of the St\"ackel matrix. We begin our proof by a first reduction of our problem.
Indeed, our main assumption is
\[S_g(\lambda) = S_{\tilde{g}}(\lambda)\]
and these operators act on $L^2(\mathcal{T}^2, s^{11} dx^2 dx^3)$ and $L^2(\mathcal{T}^2, \tilde{s}^{11} dx^2 dx^3)$ respectively. To compare these objects we thus must have
\[ s^{11} = \tilde{s}^{11}.\]
Thanks to this equality and the gauge choice $f_2 = f_3 = 1$, we will show easily that
\[ \begin{pmatrix} s_{22} & s_{23} \\ s_{32} & s_{33} \end{pmatrix} = \begin{pmatrix} \tilde{s}_{22} & \tilde{s}_{23} \\ \tilde{s}_{32} & \tilde{s}_{33} \end{pmatrix} G,\]
where $G$ is a constant matrix of determinant $1$. As mentioned in the Introduction, the presence of the matrix $G$ is due to an invariance of the metric
$g$ under a particular choice of the 
St\"ackel matrix $S$. We can then assume that $G = I_2$ and we thus obtain
\begin{equation}\label{eqbloc1}
 \begin{pmatrix} s_{22} & s_{23} \\ s_{32} & s_{33} \end{pmatrix} = \begin{pmatrix} \tilde{s}_{22} & \tilde{s}_{23} \\ \tilde{s}_{32} & \tilde{s}_{33} \end{pmatrix}.
\end{equation}
Secondly, we want to show that $s_{21} = \tilde{s}_{21}$ and $s_{31} = \tilde{s}_{31}$. Using the particular structures of the operators $H$ and $L$, we can easily show
that
\begin{equation}\label{eqover}
 \begin{pmatrix} \partial_2^2 \\ \partial_3^2 \end{pmatrix} = - \begin{pmatrix}
               s_{23} & s_{22} \\ 
               s_{33} & s_{32}
              \end{pmatrix} \begin{pmatrix} H \\ L \end{pmatrix} + (\lambda^2+1) \begin{pmatrix} s_{21} \\ s_{31} \end{pmatrix}.
\end{equation}
We then apply Equation (\ref{eqover}) on a vector of generalized harmonics
\[ \begin{pmatrix} Y_m \\ Y_m \end{pmatrix}.\]
We use the decomposition onto the generalized harmonics to write $Y_m = \sum_{p \in E_m} c_p \tilde{Y}_p$, $m \geq 1$, on the Hilbertian basis of generalized harmonics 
$\{ \tilde{Y}_m \}_{m \geq 1}$ and we identify for each $p \in E_m$ the coefficient of the harmonic $\tilde{Y}_p$. Hence, we obtain, thanks to (\ref{eqbloc1}), that
\[ - \begin{pmatrix}
               s_{23} & s_{22} \\ 
               s_{33} & s_{32}
              \end{pmatrix} \begin{pmatrix} \mu_m^2 \\ \nu_m^2 \end{pmatrix} + (\lambda^2 +1) \begin{pmatrix} s_{21} \\ s_{31} \end{pmatrix} 
  = - \begin{pmatrix}
               s_{23} & s_{22} \\ 
               s_{33} & s_{32}
              \end{pmatrix} \begin{pmatrix} \tilde{\mu}_p^2 \\ \tilde{\nu}_p^2 \end{pmatrix} + (\lambda^2 +1) \begin{pmatrix} \tilde{s}_{21} \\ \tilde{s}_{31} \end{pmatrix}, \quad \forall p \in E_m.\]
We put at the left-hand side the constants terms with respect to the variables $x^2$ and $x^3$ and the other terms at the right-hand side. We thus obtain that
\[\begin{cases}
s_{21}(x^2) = \tilde{s}_{21}(x^2) - C_1 s_{23}(x^2) - C_2 s_{22}(x^2) \\
s_{31}(x^3) = \tilde{s}_{31}(x^3) - C_1 s_{33}(x^3) - C_2 s_{32}(x^3)
\end{cases},\]
where $C_1$ and $C_2$ are real constants.
We note that, as mentioned previously in the Introduction, these equalities describe an invariance of the metric $g$ under the definition of the St\"ackel matrix 
$S$ and we can choose $C_1=C_2=0$. Finally, we obtain
\[ \begin{pmatrix}
               s_{21} \\ 
               s_{31}
              \end{pmatrix} = \begin{pmatrix}
               \tilde{s}_{21} \\ 
               \tilde{s}_{31}  \end{pmatrix}.\]
We conclude Section \ref{invang} by noticing that, thanks to these results, $H = \tilde{H}$ and $L = \tilde{L}$. As a consequence, since the generalized harmonics only 
depend on $H$ and $L$, we can choose $Y_m = \tilde{Y}_m$ and
\[ \begin{pmatrix}  \mu_{m}^2  \\ \nu_{m}^2 \end{pmatrix} = \begin{pmatrix}  \tilde{\mu}_{m}^2  \\ \tilde{\nu}_{m}^2 \end{pmatrix}, \quad \forall m \geq 1.\]
We emphasize that the choice of the generalized harmonics is not uniquely defined in each eigenspace associated with an eigenvalue with multiplicity higher than two. However, 
the scattering matrix does not depend on the choice of the $Y_m$ on each eigenspace.\\

\noindent
\underline{Step 3:} In a third step, we solve in Section \ref{invrad} the inverse problem for the radial part of the St\"ackel matrix. The main tool of this Section is a 
multivariable version of the 
Complex Angular Momentum method. The main assumption of Theorem \ref{main} implies that,
\[ M(\mu_m^2,\nu_m^2) = \tilde{M}(\mu_m^2,\nu_m^2), \quad \forall m \geq 1.\]
Roughly speaking, the aim of the Complexification of the Angular Momentum method is the following: from a discrete set of informations (here the equality of the Weyl-Titchmarsh functions 
on the coupled spectrum) we want to obtain a continuous regime of informations (here the equality of these functions on $\C^2$). In other words,
we want to extend the previous equality on $\C^2$, i.e. that we want to show that
\[ M(\mu^2,\nu^2) = \tilde{M}(\mu^2,\nu^2), \quad \forall (\mu,\nu) \in \C^2 \setminus P,\]
where $P$ is the set of points $(\mu,\nu) \in \C^2$ such that the Weyl-Titchmarsh functions does not exists, i.e. such that the denominator vanishes.
We proceed as follows. Recalling the definition of the Weyl-Titchmarsh function given in (\ref{deffonctionWTintro}), we consider the application
\[\begin{array}{ccccl}
\psi & : & \C^2 & \to & \C \\
 & & (\mu,\nu) & \mapsto & \tilde{\Delta}(\mu^2,\nu^2)\delta(\mu^2,\nu^2) - \Delta(\mu^2,\nu^2)\tilde{\delta}(\mu^2,\nu^2)\\
\end{array}\]
and we want to show that $\psi$ is identically zero on $\C^2$. To obtain this fact we use an uniqueness result for multivariable holomorphic functions given in \cite{Bl} which says 
roughly speaking that a holomorphic function which satisfies good estimates on a certain cone and which has enough zeros in this cone is identically zero. We thus first show that the
function $\psi$ is holomorphic and of exponential type with respect to $\mu$ and $\nu$, i.e. that we can find three positive constants $A$, $B$ and $C$ such that $|\psi(\mu,\nu)|$ 
is less than $C \exp(A|\textrm{Re}(\mu)| + B |\textrm{Re}(\nu)|)$. Up to an exponential correction, we then obtain that $\psi$
is holomorphic and bounded on a certain cone of $(\R^+)^2$.
Finally, we quantify the zeros of $\psi$ in this cone using the knowledge of the distribution of the coupled
spectrum (on which the function $\psi$ vanishes) given in the works of Colin de Verdi\`ere \cite{CdV1,CdV2}.
We can then conclude that $\psi = 0$, i.e.
\[ M(\mu^2,\nu^2) = \tilde{M}(\mu^2,\nu^2), \quad \forall (\mu,\nu) \in \C^2 \setminus P\]
and, by definition, we deduce from this equality that,
\[  M_{q_{\nu^2}}(\mu^2) =  M_{\tilde{q}_{\nu^2}}(\mu^2), \quad \forall (\mu,\nu) \in \C^2 \setminus P.\]

\noindent
\underline{Step 4:} We use the celebrated Borg-Marchenko Theorem (see \cite{DaKaNi,FY}) to obtain
\[ q_{\nu_m^2} =  \tilde{q}_{\nu_m^2}, \quad \forall m \geq 1.\]
Since this equality is true for all $m \geq 1$, we can ``decouple'' the potential
\[ q_{\nu_m^2} = -(\lambda^2+1) \frac{s_{11}}{s_{12}} + \nu_m^2 \frac{s_{13}}{s_{12}}
 +\frac{1}{16} \left( \dot{\left( \log \left( \frac{f_{1}}{s_{12}} \right) \right)}\right)^2 - \frac{1}{4} \ddot{\left( \log \left( \frac{f_{1}}{s_{12}} \right) \right)}.\]
and we thus obtain the uniqueness of the quotient
\[ \frac{s_{13}}{s_{12}}\]
and one ODE on the quotients 
\[ \frac{f_{1}}{s_{12}}, \quad \frac{\tilde{f}_{1}}{\tilde{s}_{12}} \quad \textrm{and} \quad \frac{s_{11}}{s_{12}}, \quad \frac{\tilde{s}_{11}}{\tilde{s}_{12}}.\]
We then rewrite this last ODE as a non-linear ODE on the function
\[ u = \left( \frac{s_{12}}{f_{1}} \frac{\tilde{f}_{1}}{\tilde{s}_{12}} \right)^{\frac{1}{4}},\]
given by
\begin{equation}\label{eqdiffuintro}
 u'' + \frac{1}{2} (\log(\tilde{h}))' u' + (\lambda^2 + 1) \tilde{h}(ls_{32} - s_{33})(s_{23}-ls_{22}) (u^5 - u)= 0, 
\end{equation}
where
\[f = \frac{s_{11}}{s_{12}}, \quad h = \frac{s_{12}}{f_1}\quad \textrm{and} \quad l = \frac{s_{13}}{s_{12}} = \tilde{l}.\]
Moreover, $u$ satisfies Cauchy conditions at the ends $0$ and $A$ given by
\[u(0) = u(A) = 1 \quad \textrm{and} \quad u'(0) = u'(A)  = 0.\]
We note that $u = 1$ is a solution of this system and by uniqueness of the Cauchy problem we conclude that $u = 1$. We then have shown that
  \[\frac{f_{1}}{s_{12}} = \frac{\tilde{f}_{1}}{\tilde{s}_{12}}.\]
Finally, using the Robertson condition, we conclude that
\begin{equation*}
 \frac{s_{11}}{s_{12}} = \frac{\tilde{s}_{11}}{\tilde{s}_{12}} \quad \textrm{and} \quad \frac{s_{11}}{s_{13}} = \frac{\tilde{s}_{11}}{\tilde{s}_{13}}.
\end{equation*}
This finishes the proof of Step 4 and together with the previous steps, the proof of Theorem \ref{main}.
We emphasize that we transformed the implicit non-linear problem of determining the metric from the knowledge of the scattering matrix at fixed energy
into an explicit non-linear problem consisting in solving the Cauchy problem associated with the non-linear ODE (\ref{eqdiffuintro}).

This paper is organized as follows. In Section \ref{dirpb} we solve the direct problem. In this Section we study the separation of variables for the Helmholtz equation, 
we define the characteristic and Weyl-Titchmarsh functions for different choices of spectral parameters and we make the link between these different functions and the 
scattering coefficients. In Section \ref{invang} we solve the inverse problem for the angular part of the St\"ackel matrix.
In Section \ref{invrad} we solve the inverse problem for the radial part of the St\"ackel matrix using a multivariable version of the 
Complex Angular Momentum method. Finally, in Section \ref{conclusion}, we finish the proof of our main Theorem \ref{main}.

\section{The direct problem}\label{dirpb}

In this Section we will study the direct scattering problem for the shifted Helmholtz equation (\ref{Hel}). 
We first study the separation of the Helmholtz equation. Secondly, we define several characteristic and generalized Weyl-Titchmarsh functions associated with 
unidimensional Schr\"odinger equations in the radial variable corresponding to different choices of spectral parameters and we study the link between these functions and the scattering operator associated with the Helmholtz 
equation.

\subsection{Separation of variables for the Helmholtz equation}

We consider (see \cite{Bort,IK,IKL,JSB}) the shifted stationary Helmholtz equation
\begin{equation}\label{Hel}
  -(\Delta_g +1) f = \lambda^2 f,
\end{equation}
where $\lambda \neq 0$ is a fixed energy, which is usually studied in case of asymptotically hyperbolic manifolds (see \cite{Bort,IK,IKL,JSB}). Indeed, it is known
(see \cite{IKL}) that the essential spectrum of $- \Delta_g$ is $[1,+\infty)$ and thus, we shift the bottom of the essential 
spectrum to $0$. It is known that the operator $- \Delta_g -1$ has no eigenvalues embedded into the essential spectrum $[0,+\infty)$ (see 
\cite{Bou,IK,IKL}).
We know that there exists a coordinates system separable for the Helmholtz equation (\ref{Hel}) if and only if the metric (\ref{met}) is in 
St\"ackel form and furthermore if the Robertson condition (\ref{Rob1}) is satisfied. 
We emphasize that, contrary to the case $n =2$ studied 
in \cite{DaKaNi}, we really need the Robertson condition in the case $n = 3$.

\begin{lemma}\label{introHL}
 The Helmholtz equation (\ref{Hel}) can be rewritten as
 \begin{equation}\label{Helint3}
 A_1 f + s_{12} Lf + s_{13} Hf = 0,
\end{equation}
where,
\begin{equation}\label{defAi}
 A_i = - \partial_i^2 + \frac{1}{2} \partial_i \left( \log(f_i) \right) \partial_i - (\lambda^2 +1) s_{i1}, \quad \textrm{for} \quad i \in \{1,2,3\},
\end{equation}
and
\begin{equation}\label{defLetH}
 L = -\frac{s_{33}}{s^{11}} A_2 + \frac{s_{23}}{s^{11}} A_3 \quad \textrm{and} \quad H = \frac{s_{32}}{s^{11}} A_2 - \frac{s_{22}}{s^{11}} A_3.
\end{equation}
\end{lemma}

\begin{proof}
 We recall that the Laplace-Beltrami operator is given in the global coordinates system $(x^i)_{i=1,2,3}$ by
\[ \Delta_g = \frac{1}{\sqrt{|g|}} \partial_i (\sqrt{|g|} g^{ij} \partial_j),\]
where $|g|$ is the determinant of the metric and $(g^{ij})$ is the inverse of the metric $(g_{ij})$. 
Using the fact that
\[g^{ii} = \frac{1}{H_i^2}, \quad \sqrt{|g|} = H_1 H_2 H_3,\]
and the Robertson condition (\ref{Rob2}) we easily show that
\begin{equation}\label{express}
 \Delta_g = \sum_{i=1}^3 \frac{1}{H_i^2} \left( \partial_i^2 - \frac{1}{2} \partial_i \left( \log(f_i) \right) \partial_i \right). 
\end{equation}

\begin{remark}
 We note that the Robertson condition is equivalent to the existence of three functions $f_i = f_i(x^i)$ such that
\[ \partial_i \log \left( \frac{H_i^4}{H_1^2 H_2^2 H_3^2} \right) = \partial_i \log( f_i), \quad \forall i \in \{1,2,3\}.\]
This equality is interesting since it gives us an expression of the Robertson condition directly in terms of the coefficients $H_i^2$ of the metric $g$.
\end{remark}
\noindent
Hence, from (\ref{express}) we immediately obtain that the Helmholtz equation (\ref{Hel}) can be written as
\begin{equation}\label{Helint1}
 \sum_{i=1}^3 \frac{1}{H_i^2} A_i^0 f = (\lambda^2 +1) f,
\end{equation}
where
\begin{equation}\label{defAi0}
A_i^0 = - \partial_i^2 + \frac{1}{2} \partial_i \left( \log(f_i) \right) \partial_i, \quad \textrm{for} \quad i \in \{1,2,3\}. 
\end{equation}
If we multiply Equation (\ref{Helint1}) by $H_1^2$ and if we use that
\[ H_1^2 = s_{11} + s_{21} \frac{s^{21}}{s^{11}} + s_{31} \frac{s^{31}}{s^{11}}, \quad \frac{H_1^2}{H_2^2} = \frac{s^{21}}{s^{11}} \quad \textrm{and} \quad \frac{H_1^2}{H_3^2} = \frac{s^{31}}{s^{11}},\]
we obtain
\begin{equation}\label{Helint2}
 A_1 f + \frac{s^{21}}{s^{11}} A_2 f + \frac{s^{31}}{s^{11}} A_3 f = 0.
\end{equation}
Finally, using the equalities
\[ \frac{s^{21}}{s^{11}} = -s_{12} \frac{s_{33}}{s^{11}} + s_{13} \frac{s_{32}}{s^{11}} \quad \textrm{and} \quad \frac{s^{31}}{s^{11}} = s_{12} \frac{s_{23}}{s^{11}} - s_{13} \frac{s_{22}}{s^{11}},\]
we obtain from (\ref{Helint2}) the equation
\[ A_1 f + s_{12} Lf + s_{13} Hf = 0,\]
where the operators $H$ and $L$ are given by (\ref{defLetH}).
\end{proof}

\begin{remark}
 Since we assumed that $f_2$ and $f_3$ are constant functions equal to $1$ (see Remark \ref{Robsimpl}) we know that
 \[ A_2^0 = - \partial_2^2 \quad \textrm{and} \quad A_3^0 = - \partial_3^2.\]
\end{remark}

\begin{remark}\label{lienopeHL}
 We can make the link between the angular operators $H$ and $L$ and the operators $\hat{P}_2$ and $\hat{P}_3$ related to the existence of Killing $2$-tensors as
 introduced in Theorem \ref{thmKM}. To do this we follow the
 construction given in \cite{KM2}. We thus consider, according to Equation (2.21) of \cite{KM2}, for $i \in \{1,2,3\}$, the operators
 \[ \hat{P}_i = \sum_{j=1}^{3} \frac{s^{ji}}{\det(S)} \left( \partial_{j}^2 - \frac{1}{2} \partial_j \log(f_j) \partial_j \right).\]
 In other words,
  \[ \begin{pmatrix}  \hat{P}_1 \\ \hat{P}_2 \\ \hat{P}_3 \end{pmatrix} = - S^{-1} \begin{pmatrix}  A_1^0 \\ A_2^0 \\ A_3^0 \end{pmatrix},\]
 where $A_i^0$ were defined in (\ref{defAi0}).
 We note that
 \[ \hat{P}_1 = \sum_{j=1}^{3} \frac{s^{j1}}{\det(S)} \left( \partial_{j}^2 - \frac{1}{2} \partial_j \log(f_j) \partial_j \right) = 
 \sum_{j=1}^{3} \frac{1}{H_j^2} \left( \partial_{j}^2 - \frac{1}{2} \partial_j\log(f_j) \partial_j \right) = \Delta_g.\]
 Since,
  \[  \begin{pmatrix}  A_1^0 \\ A_2^0 \\ A_3^0 \end{pmatrix}= - S \begin{pmatrix}  \hat{P}_1 \\ \hat{P}_2 \\ \hat{P}_3 \end{pmatrix},\]
 we see that
 \[ \begin{pmatrix} A_2^0 \\ A_3^0 \end{pmatrix} = -\begin{pmatrix} s_{22} & s_{23} \\ s_{32} & s_{33} \end{pmatrix} \begin{pmatrix} \hat{P}_2 \\ \hat{P}_3 \end{pmatrix} - \begin{pmatrix}  s_{21} \hat{P}_1 \\ s_{31}\hat{P}_1 \end{pmatrix}.\]
 Applied to solutions of the Helmholtz equation (\ref{Hel}) these operators coincide with
 \[ \begin{pmatrix} A_2^0 -(\lambda^2 +1) s_{21} \\ A_3^0 - (\lambda^2 +1) s_{31} \end{pmatrix} =
 -\begin{pmatrix} s_{22} & s_{23} \\ s_{32} & s_{33} \end{pmatrix} \begin{pmatrix} \hat{P}_2 \\ \hat{P}_3 \end{pmatrix}.\]
 Since
 \[ \begin{pmatrix} A_2^0 -(\lambda^2 +1) s_{21} \\ A_3^0 -(\lambda^2 +1) s_{31} \end{pmatrix} = \begin{pmatrix} A_2 \\ A_3 \end{pmatrix},\]
 it follows that
 \[ \begin{pmatrix} \hat{P}_2 \\ \hat{P}_3 \end{pmatrix} = -\frac{1}{s^{11}} \begin{pmatrix} s_{33} & -s_{23} \\ -s_{32} & s_{22} \end{pmatrix} \begin{pmatrix} A_2 \\ A_3 \end{pmatrix}.\]
 In other words, when applied to solution of (\ref{Hel}), we have
 \[ \hat{P}_2 = -\frac{s_{33}}{s^{11}} A_2 + \frac{s_{23}}{s^{11}}A_3 \quad \textrm{and} \quad \hat{P}_3 = \frac{s_{32}}{s^{11}} A_2 - \frac{s_{22}}{s^{11}}A_3,\]
 or
 \[ \hat{P}_2 = L \quad \textrm{and} \quad \hat{P}_3 = H.\]
 We emphasize that the operators $L$ and $H$ and the operators $\hat{P}_2$ and $\hat{P}_3$ respectively coincide only if we apply them to solutions of the Helmholtz 
 equation (\ref{Hel}).
 Moreover, thanks to \cite{KM2} we know that $[\hat{P}_2,\hat{P}_3] =0$. Thus, $L$ and $H$ are commuting operators.
 Moreover, the coupled eigenvalues of the operators $L$ and $H$ correspond to the separation constants of the Helmholtz equation.
\end{remark}

\noindent
The operators $L$ and $H$ are of great interest in our study. In particular we will show that $L$ and $H$ are elliptic operators in the sense of the definition of
ellipticity given in \cite{KKL} which we recall here. Let $a(y,D)$ be a differential operator given in local coordinates by
  \[ (a(y,D)f)(y) = - a^{jk}(y) \partial_j \partial_k f(y) - b^j(y) \partial_j f(y) - c(y)f(y),\]
  where $y=(x^2,x^3)$, $\partial_j = \frac{\partial}{\partial x^j}$, $j \in \{2,3\}$, the coefficients are real and $(a^{jk})$ is a symmetric matrix. The differential 
  operator $a(y,D)$ is then said to be elliptic if the matrix $(a^{jk})$ is positive definite.
  We can first prove the following lemma.

\begin{lemma}\label{propHL}
 The operators $L$ and $H$ satisfy the following properties:
 \begin{enumerate}
  \item $LH = HL$.
  \item $L$ and $H$ are elliptic operators.
  \item $L$ and $H$ are selfadjoint operators on the space $L^2(\mathcal{T}^2, s^{11} dx^2 dx^3)$.
  \item $L$ and $H$ are semibounded operators.
 \end{enumerate}
\end{lemma}

\begin{proof}
 \begin{enumerate}
  \item The proof of the commutativity of the operators $L$ and $H$ is quite easy since $A_2$ and $A_3$ are commuting operators and $s_{22}$ and $s_{23}$ only depend on 
  $x^2$ whereas $s_{32}$ and $s_{33}$ only depend on $x^3$. We note that, thanks to the fact that $\hat{P}_2 = L$ and $\hat{P}_3 = H$, we already know that these operators 
  commute thanks to \cite{KM2}.
  \item Using the definitions of $L$ and $H$ given in (\ref{defLetH}), we obtain that $L$ is an elliptic operator if and only
   if the matrix
   \[ \begin{pmatrix}  -\frac{s_{33}}{s^{11}} & 0 \\ 0 & \frac{s_{23}}{s^{11}} \end{pmatrix}\]
   is positive-definite whereas $H$ is an elliptic operator if and only if the matrix
      \[ \begin{pmatrix}  \frac{s_{32}}{s^{11}} & 0 \\ 0 & -\frac{s_{22}}{s^{11}} \end{pmatrix}\]
   is positive-definite. We now recall that $s_{22}, \, s_{33} < 0$ and $s_{23}, \, s_{32} > 0$ (see condition (\ref{cadre}) in Proposition \ref{propcadre}) and that $s^{11} > 0$ (see Remark (\ref{rkrie})). 
   We can thus conclude that $L$ and $H$ are elliptic operators.
   \item We just study the operator $L$ since the proof is similar for the operator $H$. We first note that, to find the weight $s^{11}$ we can use the exercise 2.19 
   of \cite{KKL} which says that an operator $A$ defined by
   \[ (Af)(y) = (a(y,D)f)(y) = -a^{jk}(y)\partial_j \partial_k f(y) + b^j(y)\partial_j f(y) + c(y) f(y), \]
   is selfadjoint on $L^2(\mathcal{T}_y^2,mg^{\frac{1}{2}}dy)$ if and only if
   \[ a(y,D) f = - \frac{1}{mg^{\frac{1}{2}}} \partial_i \left( mg^{\frac{1}{2}} a^{ij} \partial_j f \right) + qf.\]
   We recall that
   \[ L = \frac{s_{33}}{s^{11}} \partial_2^2 - \frac{s_{23}}{s^{11}} \partial_3^2 + q(x^2,x^3).\]
   Thus,
   \[ \langle Lu , v \rangle = \left\langle \frac{s_{33}}{s^{11}} \partial_2^2 u , v \right\rangle + \left\langle - \frac{s_{23}}{s^{11}} \partial_3^2 u , v \right\rangle + \left\langle qu , v \right\rangle.\]
   Moreover,
   \[ \left\langle \frac{s_{33}}{s^{11}} \partial_2^2 u , v \right\rangle = \int_{\mathcal{T}^2} s_{33} \left( \partial_2^2u \right) v \, dx^2dx^3 =  \int_{\mathcal{T}^2} s_{33}u \left( \partial_2^2v \right) \, dx^2dx^3,\]
   since the boundary terms vanish by periodicity and the function $s_{33}$ does not depend on $x^2$. Thus,
   \[ \left\langle \frac{s_{33}}{s^{11}} \partial_2^2 u , v \right\rangle = \left\langle u , \frac{s_{33}}{s^{11}} \partial_2^2 v \right\rangle.\]
   The second and the third terms can be treat following the same procedure. Finally, we have shown that
   $L$ is selfadjoint on the space $L^2(\mathcal{T}^2, s^{11} dx^2 dx^3)$.
   \item Since the proof is similar for the operator $H$ we just give the proof for the operator $L$.
   \begin{eqnarray*}
          \langle Lu , u \rangle &=& \left\langle \left( -\frac{s_{33}}{s^{11}}A_2 + \frac{s_{23}}{s^{11}} A_3 \right) u , u \right\rangle\\
          &=& \left\langle \frac{s_{33}}{s^{11}} \partial_2^2 u , u \right\rangle + \left\langle - \frac{s_{23}}{s^{11}} \partial_3^2 u , u \right\rangle + (\lambda^2+1) \left\langle  \frac{s_{33}s_{21} - s_{23}s_{31}}{s^{11}} u , u \right\rangle.
         \end{eqnarray*}
  We now study each of these terms.
     \begin{eqnarray*}
          \left\langle \frac{s_{33}}{s^{11}} \partial_2^2u , u \right\rangle &=& \int_{\mathcal{T}^2} \frac{s_{33}}{s^{11}} \left( \partial_2^2 u \right) u s^{11}\, dx^2dx^3\\
&=&  \int_{\mathcal{T}^2} s_{33} \left( \partial_2^2 u \right) u \, dx^2dx^3\\
&=& \underbrace{\left[ s_{33} \left(\partial_2 u\right) u\right]}_{=0 \, \, \textrm{by} \, \, \textrm{periodicity}} + \int_{\mathcal{T}^2} (-s_{33}) (\partial_2 u)^2 \, dx^2dx^3\\
&\geq& 0,
         \end{eqnarray*}
         since $s_{33} < 0$. Similarly,
\[ \left\langle - \frac{s_{23}}{s^{11}} \partial_3^2u , u \right\rangle \geq 0.\]
 At last, since $s_{ij} \in C^{\infty}(\mathcal{T}^2)$ for $i \in \{2,3\}$ and $j \in \{1,2,3\}$, there exists $m \in \R$ such that
 \[ (\lambda^2+1) \left\langle  \frac{s_{33}s_{21} - s_{23}s_{31}}{s^{11}} u , u \right\rangle \geq m \langle u , u \rangle.\]
\end{enumerate}
\end{proof}

\begin{remark}\label{rknum}
 Since the operators $L$ and $H$ commute, there exists a common Hilbertian basis of eigenfunctions of $H$ and $L$. Moreover, the ellipticity property on a 
 compact manifold shows that the spectrum is discrete and the selfadjointness proves that the spectrum is real. The generalized harmonics $\{ Y_m \}_{m \geq 1}$,
 associated with the coupled or joint spectrum $(\mu_m^2,\nu_m^2)$ for $(H,L)$, form a
 Hilbertian basis of $L^2(\mathcal{T}^2, s^{11} dx^2 dx^3)$, i.e.
 \begin{equation}\label{LHvp}
  HY_m = \mu_m^2 Y_m \quad \textrm{and} \quad LY_m = \nu_m^2 Y_m, \quad \forall m \geq 1,
 \end{equation}
 and
 \[  L^2(\mathcal{T}^2,s^{11}dx^2dx^3) = \displaystyle{\bigoplus_{m \geq 1}} \langle Y_m \rangle.\]
 Here, we order the coupled spectrum $(\mu_m^2,\nu_m^2)$ such that
 \begin{enumerate}
  \item Counting multiplicity:
  \[ \mu_1^2 < \mu_2^2 \leq \mu_3^2 \leq \mu_4^2 \leq ... \leq \mu_n^2 \leq ... \to \infty.\]
  \item Starting from $n=1$ and by induction on $n$, for each $n \geq 1$ such that $\mu_n^2$ has multiplicity $k$, i.e. $\mu_n^2 = \mu_{n+1}^2 = ... = \mu_{n+k-1}^2$, we 
  order the corresponding $(\nu_j^2)_{n \leq j \leq n+k-1}$ in increasing order, i.e., counting multiplicity,
  \[ \nu_n^2 \leq \nu_{n+1}^2 \leq ... \leq \nu_{n+k-1}^2.\]
 \end{enumerate}
\end{remark}
\noindent
The toric cylinder's topology implies that the boundary conditions are compatible with the decomposition on the common harmonics $\{ Y_m \}_{m \geq 1}$ of $H$ and $L$.
We thus look for solutions of (\ref{Hel}) under the form
\begin{equation}\label{decf}
  f(x^1,x^2,x^3) = \sum_{m \geq 1} u_m(x^1) Y_m(x^2,x^3).
\end{equation}
We use (\ref{decf}) in (\ref{Helint3}) and we obtain that $u_m$ satisfies, for all $m \geq 1$,
\[ -u''(x^1) + \frac{1}{2} (\log(f_1)(x^1))' u'(x^1) + \left[-(\lambda^2 +1)s_{11}(x^1) + \mu_m^2 s_{12}(x^1) + \nu_m^2 s_{13}(x^1)\right]u(x^1) = 0.\]
Finally, inverting (\ref{LHvp}), we obtain
\begin{equation}\label{systemeeqA2A3}
 \begin{cases}
A_2Y_m = -(s_{22} \mu_m^2 + s_{23} \nu_m^2) Y_m \\
A_3Y_m = -(s_{32} \mu_m^2 + s_{33} \nu_m^2) Y_m
\end{cases}.
\end{equation}

\begin{remark}\label{remmult}
 The harmonics $Y_m(x^2,x^3)$, $m \geq 1$, can be written as a product of a function of the variable $x^2$ and a function of the variable $x^3$. Let $(f_2,g_2)$ and 
 $(f_3,g_3)$ be periodic fundamental systems of solutions associated with the operators $A_2$ and $A_3$ respectively. We can thus write $Y_m(x^2,x^3)$ as
 \[ Y_m(x^2,x^3) = a(x^3) f_2(x^2) + b(x^3) g_2(x^2).\]
 We then apply the operator $A_3$ on this equality and we obtain that
  \[ A_3(Y_m)(x^2,x^3) = A_3(a)(x^3) f_2(x^2) + A_3(b)(x^3) g_2(x^2).\]
 Thus, using that $A_3Y_m = -(s_{32} \mu_m^2 + s_{33} \nu_m^2) Y_m$ and the fact that $(f_2,g_2)$ is a fundamental system of solutions we obtain that
 \[ Y_{m}(x^2,x^3) = af_2(x^2)f_3(x^3) + bf_2(x^2)g_3(x^3) + cg_2(x^2)f_3(x^3) + dg_2(x^2)g_3(x^3),\]
 where
 $a$, $b$, $c$ and $d$ are real constants. Thus, for each coupled 
 eigenvalue $(\mu_m^2,\nu_m^2)$, $m \geq 1$, the corresponding eigenspace for the couple of operator $(H,L)$ is at most of dimension four. However, the diagonalization of the scattering matrix $S_g(\lambda)$
 does not depend on the choice of the harmonics in each eigenspace associated with a coupled eigenvalue $(\mu_m^2,\nu_m^2)$ and we can thus choose as harmonics: $Y_m = f_2f_3$, $Y_m = f_2g_3$, 
 $Y_m = g_2f_3$ and $Y_m = g_2g_3$. We can then assume that $Y_m(x^2,x^3)$ is a product of a function of the variable $x^2$ and a function of the variable $x^3$.
\end{remark}

\begin{lemma}\label{eqsep}
Any solution $u \in H^1(\mathcal{M})$ of $-(\Delta_g +1)u = \lambda^2 u$, can be written as
\[ u = \sum_{m \geq 1} u_m(x^1) Y_m(x^2,x^3),\]
where $Y_m(x^2,x^3) = v_m(x^2) w_m(x^3)$ and
 \[\begin{cases}
-u_m''(x^1) + \frac{1}{2} (\log(f_1)(x^1))' u_m'(x^1) + [-(\lambda^2 +1)s_{11}(x^1) + \mu_m^2 s_{12}(x^1) + \nu_m^2 s_{13}(x^1)]u_m(x^1) = 0\\
-v_m''(x^2) + [-(\lambda^2 +1)s_{21}(x^2) + \mu_m^2 s_{22}(x^2) + \nu_m^2 s_{23}(x^2)]v_m(x^2) = 0\\
-w_m''(x^3) + [-(\lambda^2 +1)s_{31}(x^3) + \mu_m^2 s_{32}(x^3) + \nu_m^2 s_{33}(x^3)]w_m(x^3) = 0
\end{cases}.\]
\end{lemma}
\noindent
From Lemma \ref{eqsep} we can deduce more informations on the eigenvalues $(\mu_m^2)_{m \geq 1}$ and $(\nu_m^2)_{m \geq 1}$. Indeed, we can prove the following Lemma 
which will be useful in the following.

\begin{lemma}\label{distrivp}
 There exist real constants $C_1$, $C_2$, $D_1$ and $D_2$ such that for all $m \geq 1$,
\[ C_1 \mu_m^2 + D_1 \leq \nu_{m}^2 \leq C_2 \mu_m^2 + D_2,\]
where
\[ C_1 = \min \left( -\frac{s_{32}}{s_{33}} \right) > 0 \quad \textrm{and} \quad C_2 = -\min \left(\frac{s_{22}}{s_{23}} \right) > 0.\]
\end{lemma}

\begin{proof}
  We first recall the angular equations of Lemma \ref{eqsep}:
  \begin{equation}\label{eqrad1}
   -v''(x^2) + [-(\lambda^2 +1)s_{21}(x^2) + \mu_m^2 s_{22}(x^2) + \nu_m^2 s_{23}(x^2)]v(x^2) = 0
  \end{equation}
and
  \begin{equation}\label{eqrad2}
-w''(x^3) + [-(\lambda^2 +1)s_{31}(x^3) + \mu_m^2 s_{32}(x^3) + \nu_m^2 s_{33}(x^3)]w(x^3) = 0.
  \end{equation}
  We use a Liouville change of variables in (\ref{eqrad1}) to transform this equation into a Schr\"odinger equation in which $-\nu_m^2$ is the spectral 
  parameter. Thus, we define the diffeomorphism
  \[ X^2 = g_2(x^2) = \int_{0}^{x^2} \sqrt{s_{23}(t)}\, \mathrm dt\]
 and we define
\[v(X^2,\mu_m^2,\nu_m^2) = v(h_2(X^2),\mu_m^2,\nu_m^2),\]
where $h_2 =g_2^{-1}$ is the inverse function of $g_2$. We also introduce a weight function to cancel the first order term. We thus define
\[ V(X^2,\mu_m^2,\nu_m^2) = \left( \frac{1}{s_{23}(h_2(X^2))} \right)^{-\frac{1}{4}} v(h_2(X^2),\mu_m^2,\nu_m^2).\]
After calculation, we obtain that $V(X^2,\mu_m^2,\nu_m^2)$ satisfies, in the variable $X^2$, the Schr\"odinger equation
\begin{equation}\label{eqrad1bis}
 - \ddot{V}(X^2,\mu_m^2,\nu_m^2) + p_{\mu_m^2,2}(X^2,\lambda) V(X^2,\mu_m^2,\nu_m^2) = -\nu_m^2 V(X^2,\mu_m^2,\nu_m^2),
\end{equation}
where,
\begin{equation}\label{potpx2}
p_{\mu_m^2,2}(X^2,\lambda) = -(\lambda^2+1) \frac{s_{21}(X^2)}{s_{23}(X^2)} + \mu_m^2 \frac{s_{22}(X^2)}{s_{23}(X^2)},
\end{equation}
with $s_{21}(X^2) := s_{21}(h_2(X^2))$, $s_{22}(X^2) := s_{22}(h_2(X^2))$ and $s_{23}(X^2) := s_{23}(h_2(X^2))$.
We follow the same procedure for (\ref{eqrad2}) putting
\[ X^3 = g_3(x^3) = \int_{0}^{x^3} \sqrt{-s_{33}(t)}\, \mathrm dt \quad \textrm{and} \quad W(X^3,\mu_m^2,\nu_m^2) = \left( \frac{1}{-s_{33}(h_3(X^3))} \right)^{-\frac{1}{4}} w(h_3(X^3),\mu_m^2,\nu_m^2)\]
and we obtain that $W(X^3)$ satisfies, in the variable $X^3$, the Schr\"odinger equation
\begin{equation}\label{eqrad2bis}
 - \ddot{W}(X^3,\mu_m^2,\nu_m^2) + p_{\mu_m^2,3}(X^3,\lambda) W(X^3,\mu_m^2,\nu_m^2) = \nu_m^2 W(X^3,\mu_m^2,\nu_m^2),
\end{equation}
where,
\begin{equation}\label{potpx3}
p_{\mu_m^2,3}(X^3,\lambda) = (\lambda^2+1) \frac{s_{31}(X^3)}{s_{33}(X^3)} - \mu_m^2 \frac{s_{32}(X^3)}{s_{33}(X^3)},
\end{equation}
with $s_{31}(X^3) := s_{31}(h_3(X^3))$, $s_{32}(X^3) := s_{32}(h_3(X^3))$ and $s_{33}(X^3) := s_{33}(h_3(X^3))$.
 Assume now that $\mu_m^2$ is fixed and look at (\ref{eqrad1bis}) and (\ref{eqrad2bis}) as eigenvalue problems in $\pm \nu_m^2$. We suppose that $\mu_m^2$ has multiplicity 
 $k \geq 1$ and we use the notations given in Remark \ref{rknum}, i.e. that we want to show that
 \[ C_1 \mu_m^2 + D_2 \leq \nu_j^2 \leq C_2 \mu_m^2 + D_2, \quad \forall m \leq j \leq m+k-1,\]
 where $\nu_j^2 \leq \nu_{j+1}^2$ for all $j \in \{m,...,m+k-1\}$.
 We know that the spectra of the operators
 \[ P_2 = - \frac{\mathrm{d}^2}{(\mathrm{d}X^2)^2} + p_{\mu_m^2,2} \quad \textrm{and} \quad P_3 = - \frac{\mathrm{d}^2}{(\mathrm{d}X^3)^2} + p_{\mu_m^2,3}\]
 are included in
 \[ [\min(p_{\mu_m^2,2}), +\infty) \quad \textrm{and} \quad [\min(p_{\mu_m^2,3}), +\infty),\]
 respectively. The first condition gives us that
   \[ - \nu_{j}^2 \geq - C_2 \mu_m^2 - D_2, \quad \textrm{where} \quad -C_2 = \min \left(\frac{s_{22}}{s_{23}} \right) \quad \textrm{and} \quad -D_2 = (\lambda^2+1) \min \left(-\frac{s_{21}}{s_{23}} \right)\]
 and the second one tells us that
  \[ \nu_{j}^2 \geq C_1 \mu_m^2 + D_1, \quad \textrm{where} \quad C_1 = \min \left( -\frac{s_{32}}{s_{33}} \right) \quad \textrm{and} \quad D_1 = (\lambda^2+1) \min \left(\frac{s_{31}}{s_{33}} \right).\]
  Since $(\nu_j^2)_{m \leq j \leq m+k-1}$ is the set of eigenvalues of (\ref{eqrad1}) and (\ref{eqrad2}), for a fixed $\mu_m^2$ of multiplicity $k$ we obtain from these estimates that
  \[ C_1 \mu_m^2 + D_1 \leq \nu_{j}^2 \leq C_2 \mu_m^2 + D_2, \quad \forall m \leq j \leq m+k-1.\]
  In other words, thanks to our numerotation of the coupled spectrum explained in Remark \ref{rknum},
  \[ C_1 \mu_m^2 + D_1 \leq \nu_{m}^2 \leq C_2 \mu_m^2 + D_2, \quad \forall m \geq 1.\]
\end{proof}

\begin{remark}
\begin{enumerate}
 \item Thanks to the condition given in Remark \ref{rkrie},
 \[C_1 = \min \left( -\frac{s_{32}}{s_{33}} \right) < -\min \left(\frac{s_{22}}{s_{23}} \right) = C_2.\]
  \item  The previous Lemma says that the coupled spectrum $\{(\mu_m^2,\nu_m^2), \, \, m \geq 1 \}$ lives in a cone contained in the quadrant $(\R^+)^2$ (up to a possible 
 shift dues to the presence of the constants $D_1$ and $D_2$). 
 Moreover, since the multiplicity of $\mu_m^2$ is finite for all $m \geq 1$, there is a finite number of points of the coupled spectrum on each vertical line.
 We can summarize these facts with the following generic picture:
  \begin{figure}[htbp]
    \center
   \includegraphics[scale=0.25]{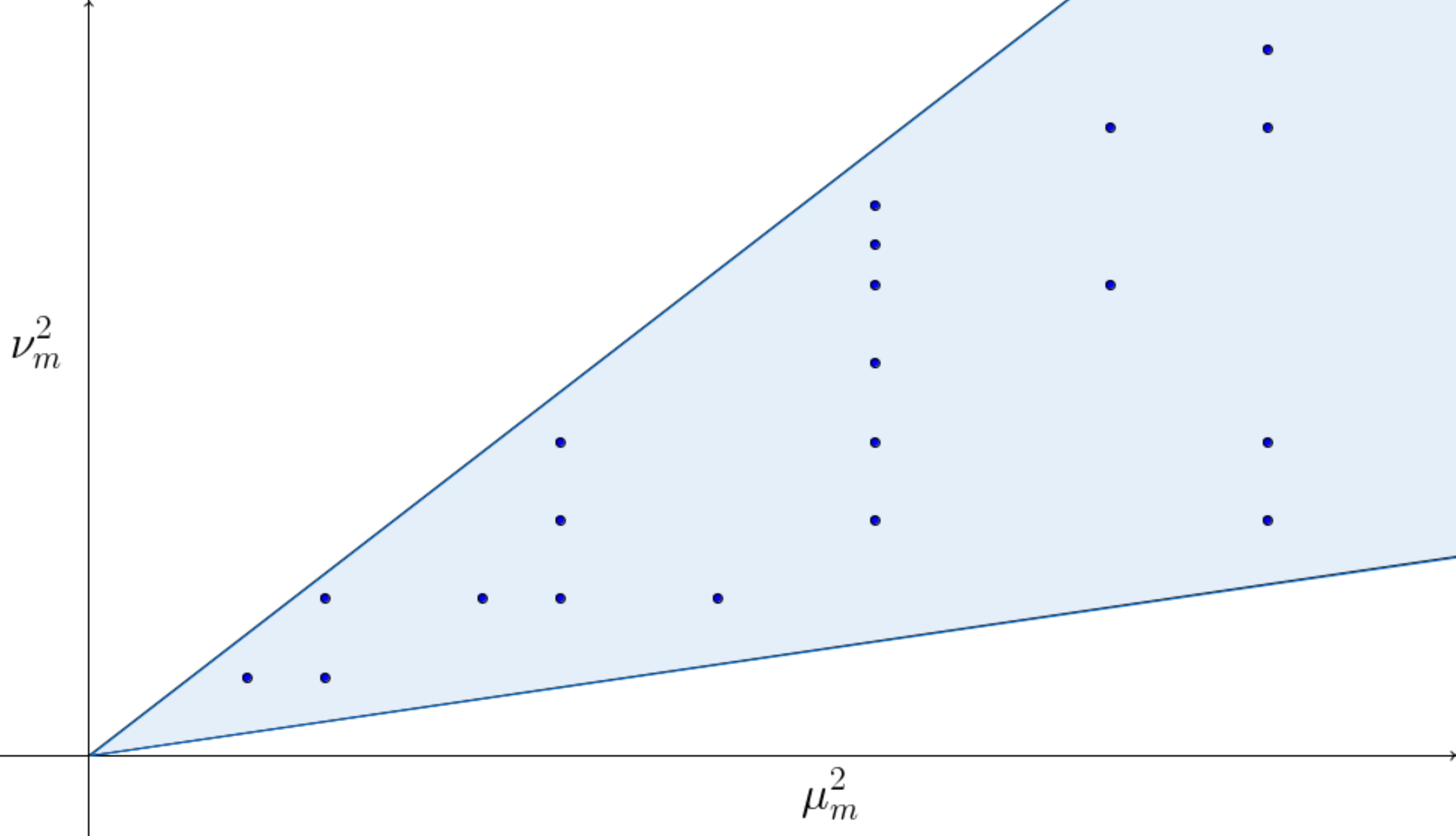}
    \caption{Coupled spectrum}
\end{figure} 
 \item The Weyl law (see \cite{KKL} Theorem 2.21) which says (in two dimensions) that there exists a constant $C$ such that the eigenvalues are equivalent for large $m$ to 
 $Cm$, is satisfied by the eigenvalues $\{\mu_m^2\}_m$ and $\{\nu_m^2\}_m$ but we have to order these
 sequences in an increasing order to use it. However, we ordered the coupled spectrum in such a way that the order for the $(\nu_m^2)$ is not necessary increasing.
\item An eigenvalue of the coupled spectrum $(\mu_m^2,\nu_m^2)$ has at most multiplicity four as it was mentionned in Remark \ref{remmult}.
\end{enumerate}
\end{remark}

\begin{exem}
 We can illustrate the notion of coupled spectrum on the examples given in Example \ref{exemple}.
 \begin{enumerate}
  \item We define the St\"ackel matrix
  \[ S_1 =  \begin{pmatrix} s_{11}(x^1) & s_{12}(x^1) & s_{13}(x^1) \\ 0 & 1 & 0 \\ 0 & 0 & 1 \end{pmatrix}.\]
      In this case $H = -\partial_3^2$ and $L = - \partial_2^2$ and we note that these operators can be obtained by derivation of the Killing vector fields $\partial_2$ and $\partial_3$.
      The coupled spectrum of these operators is $\{(m^2,n^2),$ $(m,n) \in \Z^2 \}$ and we can 
      decompose the space $L^2(\mathcal{T}^2,s^{11} dx^2 dx^3)$ on the basis of generalized harmonics $Y_{mn} = e^{imx^2 + inx^3}$. We note that this coupled spectrum is not
      included in a cone strictly contained in $(\R^+)^2$ but there is no contradiction with Lemma \ref{distrivp} since the St\"ackel matrix $S$
      does not satisfies the condition (\ref{cadre}). However, we can use the invariances of Proposition \ref{propinv} to come down to our framework (this transformation 
      modifies the coupled spectrum). Indeed, we can obtain the St\"ackel matrix
    \[ S_2 \begin{pmatrix} s_{11}(x^1) & s_{12}(x^1) & s_{13}(x^1) \\ a & b & c \\ d & e & f \end{pmatrix},\]
  where $s_{11}$, $s_{12}$ and $s_{13}$ are smooth functions of $x^1$ and $a$, $b$, $c$, $d$, $e$ and $f$ are real constants such that $b,f < 0$ and $c,e > 0$.
  In this case we have
    \[ H = -\frac{s_{32}}{s^{11}} \partial_2^2 + \frac{s_{22}}{s^{11}} \partial_3^2 =  -\frac{e}{bf-ce} \partial_2^2 + \frac{b}{bf-ce} \partial_3^2\]
and
  \[ L = \frac{s_{33}}{s^{11}} \partial_2^2 - \frac{s_{23}}{s^{11}} \partial_3^2 =  \frac{f}{bf-ce} \partial_2^2 - \frac{c}{bf-ce} \partial_3^2.\]
 Thus, the coupled spectrum of the operators $H$ and $L$ can be computed using the same procedure as the one used for $S_1$.
  
       We emphasize that in the case of the St\"ackel matrix $S_1$ the coupled spectrum is in fact uncoupled. We can thus freeze one angular momentum and let the other one 
      move on the integers.
      After the use of the invariance to come down to our framework these vertical or horizontal half-lines are transformed into half-lines contained in our cone of $(\R^+)^2$.
      This allows us to use the usual Complexification of the Angular Momentum method in one dimension on a half-line contained in our cone.
  \item We define the St\"ackel matrix
  \[ S =  \begin{pmatrix} s_{11}(x^1) & s_{12}(x^1) & a s_{12}(x^1) \\ 0 & s_{22}(x^2) & s_{23}(x^2) \\ 0 & s_{32}(x^3) & s_{33}(x^3) \end{pmatrix},\]
    where $s_{11}$ and $s_{12}$ are smooth functions of $x^1$, $s_{22}$ and $s_{23}$ are smooth functions of $x^2$, $s_{32}$ and $s_{33}$ are smooth functions of $x^3$ and 
  $a$ is a real constant. In this case, the Helmholtz equation (\ref{Hel}) can be rewritten as
 \[ A_1 f + s_{12} (L + a H) f = 0.\]
 Therefore, the separation of variables depends only on a single angular operator given by $L + a H$ whose properties can be easily derived from the ones for $H$ and $L$. 
 In particular, the set of angular momenta is given by $\omega_m^2 = \mu_m^2 + \nu_m^2$, $m \geq 1$, and could be used to apply the Complexification of the Angular 
 Momentum method. Note that, even though the spectra $\{\mu_m^2,\nu_m^2\}$ are coupled, only the spectrum $\omega_m^2$ appears in the separated radial equation.
  \item In the case of the St\"ackel matrix
    \[ S =  \begin{pmatrix} s_{1}(x^1)^2 & -s_{1}(x^1) & 1 \\ -s_2(x^2)^2 & s_{2}(x^2) & -1 \\ s_3(x^3)^2 & -s_{3}(x^3) & 1 \end{pmatrix},\]
  where $s_{1}$ is a smooth function of $x^1$, $s_{2}$ is a smooth function of $x^2$ and $s_{3}$ is a smooth function of $x^3$, there is no trivial symmetry. We 
  are thus in the general case and we have to use the general method we develop in this paper.
  \end{enumerate}
\end{exem}

\subsection{A first construction of characteristic and Weyl-Titchmarsh functions}\label{firstdef}

The aim of this section is to define the characteristic and Weyl-Titchmarsh functions for the radial equation choosing $-\mu_m^2$ 
as the spectral parameter. We recall that the radial equation is 
\begin{equation}\label{eqR}
-u'' + \frac{1}{2} (\log(f_1))' u' + \left[-(\lambda^2 +1)s_{11} + \mu_m^2 s_{12} + \nu_m^2 s_{13} \right] u = 0,
\end{equation}
where the functions depend only on $x^1$. We choose $-\mu^2 := -\mu_m^2$ to be the spectral parameter. As mentioned in the Introduction, to do this we make a Liouville change of variables:
\[X^1 = g(x^1) = \int_{0}^{x^1} \sqrt{s_{12}(t)}\, \mathrm dt\]
and we define
\[u(X^1,\mu^2,\nu^2) = u(h(X^1),\mu^2,\nu^2),\]
where $h =g^{-1}$ is the inverse function of $g$ and $\nu^2 := \nu_m^2$. In a second time, to cancel the resulting first order term and obtain a Schr\"odinger equation, we 
introduce a weight function. Precisely, we define
\begin{equation}\label{nouvU}
 U(X^1,\mu^2,\nu^2) = \left( \frac{f_1}{s_{12}}(h(X^1)) \right)^{-\frac{1}{4}} u(h(X^1),\mu^2,\nu^2).
\end{equation}
After calculation, we obtain that $U(X^1,\mu^2,\nu^2)$ satisfies, in the variable $X^1$, the Schr\"odinger equation
\begin{equation}\label{eqmu}
 - \ddot{U}(X^1,\mu^2,\nu^2) + q_{\nu^2}(X^1,\lambda) U(X^1,\mu^2,\nu^2) = -\mu^2 U(X^1,\mu^2,\nu^2),
\end{equation}
where,
\begin{equation}\label{potq}
 q_{\nu^2}(X^1,\lambda) = -(\lambda^2+1) \frac{s_{11}(X^1)}{s_{12}(X^1)} + \nu^2 \frac{s_{13}(X^1)}{s_{12}(X^1)}
 +\frac{1}{16} \left( \dot{\left( \log \left( \frac{f_{1}(X^1)}{s_{12}(X^1)} \right) \right)}\right)^2 - \frac{1}{4} \ddot{\left( \log \left( \frac{f_{1}(X^1)}{s_{12}(X^1)} \right) \right)}.
\end{equation}
with $\dot{f} = \frac{df}{dX^1}$, $f_1(X^1) := f_1(h_1(X^1))$, $s_{11}(X^1) := s_{11}(h_1(X^1))$, $s_{12}(X^1) := s_{12}(h_1(X^1))$ and $s_{13}(X^1) := s_{13}(h_1(X^1))$.

\begin{lemma}
 The potential $q_{\nu^2}$ satisfies, at the end $\{X^1 = 0\}$,
 \[q_{\nu^2}(X^1,\lambda) = - \frac{\lambda^2 + \frac{1}{4}}{(X^1)^2} + q_{0,\nu^2}(X^1,\lambda),\]
 where $X^1 q_{0,\nu^2}(X^1,\lambda) \in L^1\left( 0, \frac{A^1}{2} \right)$ with $A^1 = g(A)$. Similarly, at the end $\{X^1 = A^1\}$,
 \[q_{\nu^2}(X^1,\lambda) = - \frac{\lambda^2 + \frac{1}{4}}{(A^1-X^1)^2} + q_{A^1,\nu^2}(X^1,\lambda),\]
  where $(A^1-X^1) q_{A^1,\nu^2}(X^1,\lambda) \in L^1\left(\frac{A^1}{2},A^1 \right)$.
\end{lemma}

\begin{proof}
 We first note that since $s_{12}(x^1) \sim 1$ when $x^1 \to 0$ we obtain by definition of $X^1$ that $X^1 \sim x^1$, as $x^1 \to 0$. Thus we can use the hyperbolicity 
 conditions (\ref{condhyp}) directly in the variable $X^1$. The lemma is then a consequence of these conditions.
\end{proof}
\noindent
We now follow the paper \cite{DaKaNi} to define the characteristic and the Weyl-Titchmarsh functions associated with Equation (\ref{eqmu}).
To do that, we introduce two fundamental systems of solutions $S_{jn}$, $j \in \{1,2\}$ and $n \in \{0,1\}$, defined by
\begin{enumerate}
 \item When $X^1 \to 0$,
 \begin{equation}\label{S_0}
 S_{10}(X^1,\mu^2,\nu^2) \sim (X^1)^{\frac{1}{2}-i\lambda} \quad \textrm{and} \quad S_{20}(X^1,\mu^2,\nu^2) \sim \frac{1}{2i\lambda} (X^1)^{\frac{1}{2}+i\lambda} 
 \end{equation}
 and when $X^1 \to A^1$,
 \begin{equation}\label{S_1}
 S_{11}(X^1,\mu^2,\nu^2) \sim (A^1-X^1)^{\frac{1}{2}-i\lambda} \quad \textrm{and} \quad S_{21}(X^1,\mu^2,\nu^2) \sim -\frac{1}{2i\lambda} (A^1-X^1)^{\frac{1}{2}+i\lambda}.  
 \end{equation}
 \item $W(S_{1n},S_{2n}) = 1$ for $n \in \{0,1\}$.
 \item For all $X^1 \in (0,A^1)$, $\mu \mapsto S_{jn}(X^1,\mu^2,\nu^2)$ is an entire function for $j \in \{1,2\}$ and $n \in \{0,1\}$.
\end{enumerate}
\noindent
As in \cite{DaKaNi,FY}, we add singular boundary conditions at the ends $\{X^1 = 0\}$ and $\{X^1 = A^1\}$ and we consider (\ref{eqmu}) as an eigenvalue problem. Precisely we
consider the conditions
\begin{equation}\label{condbordsing}
U(u) := W(S_{10},u)_{|X^1 = 0} = 0 \quad \textrm{and} \quad V(u) := W(S_{11},u)_{|X^1 = A^1} = 0,
\end{equation}
 where $W(f,g) = fg'-f'g$ is the Wronskian of $f$ and $g$. Finally, we can define the characteristic functions
\begin{equation}\label{defDelta1}
 \Delta_{q_{\nu^2}}(\mu^2) = W(S_{11},S_{10}) \quad \textrm{and} \quad \delta_{q_{\nu^2}}(\mu^2) = W(S_{11},S_{20})
\end{equation}
and the Weyl-Titchmarsh function 
\begin{equation}\label{defM1}
 M_{q_{\nu^2}}(\mu^2) = -\frac{W(S_{11},S_{20})}{W(S_{11},S_{10})} = -\frac{\delta_{q_{\nu^2}}(\mu^2)}{\Delta_{q_{\nu^2}}(\mu^2)}.
\end{equation}

\begin{remark}
 The Weyl-Titchmarsh function is the unique function such that the solution of (\ref{eqmu}) given by
 \[\Phi(X^1,\mu^2,\nu^2) = S_{10}(X^1,\mu^2,\nu^2) + M_{q_{\nu^2}}(\mu^2) S_{10}(X^1,\mu^2,\nu^2),\]
 satisfies the boundary condition at the end $\{X^1 = A^1\}$.
\end{remark}
\noindent
An immediate consequence of the third condition in the definition of the fundamental systems of solutions is the following lemma.

\begin{lemma}\label{holo1}
 For any fixed $\nu$, the maps
 \[ \mu \mapsto \Delta_{q_{\nu^2}}(\mu^2) \quad \textrm{and} \quad \mu \mapsto \delta_{q_{\nu^2}}(\mu^2)\]
 are entire.
\end{lemma}

In the following (see Subsections \ref{secdef} and \ref{troisdef}), we will define other characteristic and Weyl-Titchmarsh functions which correspond to other choices of spectral parameters which are $-\nu_m^2$ and 
$-(\mu_m^2+ \nu_m^2)$. We study here the influence of these other choices.

\begin{prop}
 Let $\breve{X}^1 = \breve{g}(x^1)$ be a change of variables and $\breve{p}$ be a weight function, then
 \[U(X^1,\mu^2,\nu^2) = \frac{p(h(X^1))}{\breve{p}(\breve{h}(\breve{X}^1))} \breve{U}((\breve{g} \circ h)(X^1),\mu^2,\nu^2)\]
 and
 \[V(X^1,\mu^2,\nu^2) = \frac{p(h(X^1))}{\breve{p}(\breve{h}(\breve{X}^1))} \breve{V}((\breve{g} \circ h)(X^1),\mu^2,\nu^2),\]
 where
 \[ p(h(X^1)) = \left( \frac{f_1}{s_{12}}(h(X^1)) \right)^{-\frac{1}{4}},\]
 \[ \breve{U}(\breve{X}^1,\mu^2,\nu^2) =  \breve{p}(\breve{h}(\breve{X}^1)) u(\breve{h}(\breve{X}^1),\mu^2,\nu^2)   ,\]
 and
  \[ \breve{V}(\breve{X}^1,\mu^2,\nu^2) =  \breve{p}(\breve{h}(\breve{X}^1)) v(\breve{h}(\breve{X}^1),\mu^2,\nu^2).\]
 Moreover,
 \[W_{X^1}(U,V) = \left( \frac{p(h(X^1))}{\breve{p}(\breve{h}(\breve{X}^1))}\right)^{2} \partial_{X^1}(\breve{g} \circ h)(X^1) W_{\breve{X}^1}(\breve{U},\breve{V}).\]
\end{prop}

\begin{coro}\label{coroega}
 Let $\hat{X}^1$ and $\check{X}^1$ be two Liouville change of variables defined by
 \[ \hat{X}^1 = \hat{g}(x^1) = \int_{0}^{x^1} \sqrt{s_{13}(t)}\, \mathrm dt \quad \textrm{and} \quad \check{X}^1 = \check{g}(x^1) = \int_{0}^{x^1} \sqrt{r_{\mu^2,\nu^2}(t)}\, \mathrm dt,\]
 where
 \[ r_{\mu^2,\nu^2}(x^1) := \frac{\mu^2 s_{12}(x^1) + \nu^2 s_{13}(x^1)}{\mu^2 + \nu^2},\]
 and let $\hat{p}$ and $\check{p}$ be two weight functions defined by
 \[\hat{p}(\hat{h}(\hat{X}^1)) = \left( \frac{f_1}{s_{13}}(\hat{h}(\hat{X}^1)) \right)^{-\frac{1}{4}} \quad \textrm{and} \quad \check{p}(\check{h}(\check{X}^1)) = \left( \frac{f_1}{r_{\mu^2,\nu^2}}(\check{h}(\check{X}^1)) \right)^{-\frac{1}{4}}.\]
 Let $\hat{U}$ and $\hat{V}$ be defined as
  \[ \hat{U}(\hat{X}^1,\mu^2,\nu^2) =  \hat{p}(\hat{h}(\hat{X}^1)) u(\hat{h}(\hat{X}^1),\mu^2,\nu^2) \quad \textrm{and} \quad
\hat{V}(\hat{X^1},\mu^2,\nu^2) =  \hat{p}(\hat{h}(\hat{X}^1)) v(\hat{h}(\hat{X}^1),\mu^2,\nu^2).\]
and $\check{U}$ and $\check{V}$ be defined as
  \[ \check{U}(\check{X}^1,\mu^2,\nu^2) =  \check{p}(\check{h}(\check{X}^1)) u(\check{h}(\check{X}^1),\mu^2,\nu^2) \quad \textrm{and} \quad
\check{V}(\check{X^1},\mu^2,\nu^2) =  \check{p}(\check{h}(\check{X}^1)) v(\check{h}(\check{X}^1),\mu^2,\nu^2).\]
 Then,
 \[ W_{X^1}(U,V) = W_{\hat{X}^1}(\hat{U},\hat{V}) = W_{\check{X}^1}(\check{U},\check{V}).\]
\end{coro}

\begin{remark}
 We will use $\hat{X}^1$ and $\hat{p}$ in Subsection \ref{secdef} to obtain holomorphic properties and good estimates in the variable $\nu^2$. Secondly, we will use 
 $\check{X}^1$ and $\check{p}$ in Subsection \ref{troisdef} to show that the characteristic functions are bounded for $(\mu,\nu) \in (i\R)^2$.
\end{remark}

\subsection{Link between the scattering coefficients and the Weyl-Titchmarsh and characteristic functions}\label{link}
%

In this Section, we follow \cite{DaKaNi}, Section 3.3, and we make the link between the transmission and the reflection coefficients, corresponding to the restriction of 
the transmission and the reflection operators on each generalized harmonics,
and the characteristic and Weyl-Titchmarsh functions defined in Subsection \ref{firstdef}. First, we observe that 
the scattering operator defined in Theorem \ref{defscat} leaves invariant the span of each generalized harmonic $Y_m$. Hence, it suffices to calculate the scattering operator on each vector space generated 
by the $Y_m$'s. To do this, we recall from Theorem \ref{defscat} that, given any solution $f = u_m(x^1)Y_m(x^2,x^3) \in \mathcal{B}^{\star}$ of (\ref{Hel}), there exists 
a unique $\psi_m^{(\pm)} = (\psi_{0m}^{\pm},\psi_{1m}^{(\pm)}) \in \mathbb{C}^2$ such that
\begin{align}\label{scatrad}
   u_m(x^1) \simeq& \quad \omega_-(\lambda) \left( \chi_0 (x^{1})^{\frac{1}{2}+i\lambda} \psi_{0m}^{(-)} + \chi_1 (A-x^1)^{\frac{1}{2}+i\lambda} \psi_{1m}^{(-)} \right) \nonumber \\
   &- \omega_+(\lambda) \left( \chi_0 (x^1)^{\frac{1}{2}-i\lambda} \psi_{0m}^{(+)} + \chi_1 (A-x^1)^{\frac{1}{2}-i\lambda} \psi_{1m}^{(+)} \right),
\end{align}
where $\omega_{\pm}$ are given by (\ref{omega}) and the cutoff functions $\chi_0$ and $\chi_1$ are defined in (\ref{defchi}). As in \cite{DaKaNi}, we would like 
to apply this result to the FSS $\{S_{jn}$, $j=1$, $2$, $n=0$, $1 \}$. However we recall that $S_{jn}$ are solutions 
of Equation (\ref{eqmu}) and this Schr\"odinger equation was obtained from the Helmholtz equation (\ref{Hel}) by a change a variables and the introduction 
of a weight function (see (\ref{nouvU})). We thus apply the previous result to the functions
\begin{equation}\label{uS}
 u_{jn} (x^1,\mu^2,\nu^2) = \left( \frac{f_1}{s_{12}}(x^1) \right)^{\frac{1}{4}} S_{jn}(g(x^1),\mu^2,\nu^2),\quad j \in \{1,2\}, \quad n\in \{0,1\}. 
\end{equation}
We first study the behaviour of the weight function at the two ends in the following Lemma.

\begin{lemma}\label{poidsbords}
 When $x^1 \to 0$,
 \[\left( \frac{f_1}{s_{12}}(x^1) \right)^{\frac{1}{4}} = \sqrt{x^1}[1]_{\epsilon_0},\]
 where
 \[ [1]_{\epsilon_0} = 1 + O((1+|\log(x^1)|)^{-1-\epsilon_0}).\]
 The corresponding result at the end $\{x^1 = A\}$ is also true.
\end{lemma}

\begin{proof}
 We first divide the Robertson condition (\ref{Rob1}) by $s_{12}$ and we obtain that
 \[ \frac{f_1}{s_{12}} = \frac{\left(\frac{\det(S)}{s_{12}}\right)^2}{\frac{H_1^2}{s_{12}}H_2^2 H_3^2}.\]
 We use the hyperbolicity conditions given in (\ref{condhypH})-(\ref{condhyp}) and Remark \ref{Robsimpl} to obtain that
 \[ \begin{cases}
\frac{\det(S)}{s_{12}} = \frac{s^{11}}{(x^1)^2}[1]_{\epsilon_0}\\
\frac{H_1^2}{s_{12}} = \frac{1}{(x^1)^2}[1]_{\epsilon_0}\\
 H_2^2 =  \frac{s^{11}}{(x^1)^2} [1]_{\epsilon_0}\\
 H_3^2 =  \frac{s^{11}}{(x^1)^2} [1]_{\epsilon_0}\\
\end{cases}.\]
The Lemma is then a direct consequence of these estimates.
\end{proof}
\noindent
Thanks to (\ref{S_0})-(\ref{S_1}), (\ref{uS}) and Lemma \ref{poidsbords}, we obtain that when $x^1 \to 0$,
 \[ u_{10}(x^1,\mu^2,\nu^2) \sim (x^1)^{1-i\lambda} \quad \textrm{and} \quad u_{20}(x^1,\mu^2,\nu^2) \sim \frac{1}{2i\lambda} (x^1)^{1+i\lambda},\]
 and when $x^1 \to A$,
 \begin{equation}\label{uA}
 u_{11}(x^1,\mu^2,\nu^2) \sim (A-x^1)^{1-i\lambda} \quad \textrm{and} \quad u_{21}(x^1,\mu^2,\nu^2) \sim -\frac{1}{2i\lambda} (A-x^1)^{1+i\lambda}.
 \end{equation}
 We denote by $\psi^{(-)} = (\psi_0^{(-)},\psi_1^{(-)})$ and $\psi^{(+)} = (\psi_0^{(+)},\psi_1^{(+)})$ the constants appearing in Theorem \ref{defscat} corresponding 
 to $u_{10}$. Since $u_{10} \sim  (x^1)^{1-i\lambda}$ when $x^1 \to 0$, we obtain that
\[ \psi_0^{(-)} = 0 \quad \textrm{and} \quad \psi_0^{(+)} = - \frac{1}{\omega_+(\lambda)}.\]
We now write $S_{10}$ as a linear combination of $S_{11}$ and $S_{21}$, i.e.
\[ S_{10} = a_{1}(\mu_m^2,\nu_m^2) S_{11} + b_1(\mu_m^2,\nu_m^2)S_{21},\]
where
\[ a_{1}(\mu_m^2,\nu_m^2) = W(S_{10},S_{21}) \quad \textrm{and} \quad b_{1}(\mu_m^2,\nu_m^2) = W(S_{11},S_{10}).\]
Thus, thanks to (\ref{uS}),
\[ u_{10} = a_{1}(\mu_m^2,\nu_m^2) u_{11} + b_1(\mu_m^2,\nu_m^2)u_{21}.\]
We then obtain, thanks to (\ref{uA}), that
\[ \psi_1^{(-)} = -\frac{b_1(\mu_m^2,\nu_m^2)}{2i\lambda \omega_-(\lambda)} \quad \textrm{and} \quad \psi_1^{(+)} = - \frac{ a_1(\mu_m^2,\nu_m^2)}{\omega_+(\lambda)}.\]
Finally, we have shown that $u_{10}$ satisfies the decomposition of Theorem \ref{defscat} with
\begin{equation}\label{psi}
\psi^{(-)} = \begin{pmatrix}
               0 \\ -\frac{b_1(\mu_m^2,\nu_m^2)}{2i\lambda \omega_-(\lambda)}
              \end{pmatrix} \quad \textrm{and} \quad
              \psi^{(+)} = \begin{pmatrix}
               -\frac{1}{\omega_+(\lambda)} \\ -\frac{a_1(\mu_m^2,\nu_m^2)}{\omega_+(\lambda)}
              \end{pmatrix}.
\end{equation}
We follow the same procedure for $u_{11}$ and we obtain the corresponding vectors
\begin{equation}\label{phi}
\phi^{(-)} = \begin{pmatrix}
               \frac{b_0(\mu_m^2,\nu_m^2)}{2i\lambda \omega_-(\lambda)} \\ 0
              \end{pmatrix} \quad \textrm{and} \quad
              \phi^{(+)} = \begin{pmatrix}
                -\frac{a_0(\mu_m^2,\nu_m^2)}{\omega_+(\lambda)} \\ -\frac{1}{\omega_+(\lambda)}
              \end{pmatrix},
\end{equation}
where $a_0(\mu_m^2,\nu_m^2) = W(S_{11},S_{20})$ and $b_0(\mu_m^2,\nu_m^2) = W(S_{10},S_{11})$. We now recall that for all $\psi_m^{(-)} \in \C^2$ there 
exists a unique vector $\psi_m^{(+)} \in \C^2$ and $u_m(x)Y_m \in \mathcal{B}^{\star}$ satisfying (\ref{Hel}) for which the expansion (\ref{scat}) holds. 
This defines the scattering matrix $S_g(\lambda,\mu_m^2,\nu_m^2)$ as the $2 \times 2$ matrix such that for all $\psi_m^{(-)} \in \C^2$
\begin{equation}\label{psiS}
 \psi_m^{(+)} = S_g(\lambda,\mu_m^2,\nu_m^2) \psi_m^{(-)}.
\end{equation}
Using the notation
\[ S_g(\lambda,\mu_m^2,\nu_m^2) = \begin{pmatrix}
               L(\lambda,\mu_m^2,\nu_m^2) & T_L(\lambda,\mu_m^2,\nu_m^2) \\ T_R(\lambda,\mu_m^2,\nu_m^2) & R(\lambda,\mu_m^2,\nu_m^2)
              \end{pmatrix},\]
and, using the definition (\ref{psiS}) of the partial scattering matrix together with (\ref{psi})-(\ref{phi}), we find
\begin{equation}\label{matscatpart}
  S_g(\lambda,\mu_m^2,\nu_m^2) = \begin{pmatrix}
               -\frac{2i\lambda \omega_-(\lambda)}{\omega_+(\lambda)} \frac{a_0(\mu_m^2,\nu_m^2)}{b_0(\mu_m^2,\nu_m^2)} & \frac{2i\lambda  \omega_-(\lambda)}{\omega_+(\lambda)} \frac{1}{b_1(\mu_m^2,\nu_m^2)} \\ 
               -\frac{2i\lambda \omega_-(\lambda)}{\omega_+(\lambda)} \frac{1}{b_0(\mu_m^2,\nu_m^2)} & \frac{2i\lambda \omega_-(\lambda)}{\omega_+(\lambda)} \frac{a_1(\mu_m^2,\nu_m^2)}{b_1(\mu_m^2,\nu_m^2)} 
              \end{pmatrix}.
\end{equation}
In this expression of the partial scattering matrix, we recognize the usual transmission coefficients $T_L(\lambda,\mu_m^2,\nu_m^2)$ and $T_R(\lambda,\mu_m^2,\nu_m^2)$ and the 
reflection coefficients  $L(\lambda,\mu_m^2,\nu_m^2)$ and $R(\lambda,\mu_m^2,\nu_m^2)$ from the left and the right respectively. Since they are written in terms of Wronskians of the $S_{jn}$, 
$j=1, \, 2$, $n= 0, \, 1$, we can make the link with the characteristic function (\ref{defDelta1}) and the generalized Weyl-Titchmarsh function (\ref{defM1}) as follows.
Noting that,
\[ \Delta_{q_{\nu_m^2}}(\mu_m^2) = b_1(\mu_m^2,\nu_m^2) = - b_0(\mu_m^2,\nu_m^2) \quad \textrm{and} \quad M_{q_{\nu_m^2}}(\mu_m^2) = \frac{a_0(\mu_m^2,\nu_m^2)}{b_0(\mu_m^2,\nu_m^2)},\]
we get,
\begin{equation}\label{lienLM}
 L(\lambda,\mu_m^2,\nu_m^2) = - \frac{2i\lambda \omega_-(\lambda)}{\omega_+(\lambda)} M_{q_{\nu_m^2}}(\mu_m^2)
\end{equation}
and
\begin{equation}\label{lienTdelta}
 T(\lambda,\mu_m^2,\nu_m^2) = T_L(\lambda,\mu_m^2,\nu_m^2) = T_R(\lambda,\mu_m^2,\nu_m^2) = \frac{2i\lambda \omega_-(\lambda)}{\omega_+(\lambda)} \frac{1}{\Delta_{q_{\nu_m^2}}(\mu_m^2)}.
\end{equation}
Finally, using the fact that the scattering operator is unitary (see Theorem \ref{defscat}) we obtain as in \cite{DaKaNi} the equality
\begin{equation}\label{lienRM}
 R(\lambda,\mu_m^2,\nu_m^2) = \frac{2i\lambda \omega_-(\lambda)}{ \omega_+(\lambda)} \frac{\overline{\Delta_{q_{\nu_m^2}}(\mu_m^2)}}{\Delta_{q_{\nu_m^2}}(\mu_m^2)} \overline{M_{q_{\nu_m^2}}(\mu_m^2)}.
\end{equation}


\subsection{A second construction of characteristic and Weyl-Titchmarsh functions}\label{secdef}

In Subsection \ref{firstdef} we defined the characteristic and the Weyl-Titchmarsh functions when $-\mu_m^2$ is the spectral parameter. We can also define these functions when 
we put $-\nu_m^2$ as the spectral parameter. We recall that the radial equation is given by (\ref{eqR}).
To choose $-\nu^2 := -\nu_m^2$ as the spectral parameter we make the Liouville change of variables:
\[\hat{X}^1 = \hat{g}(x^1) = \int_{0}^{x^1} \sqrt{s_{13}(t)}\, \mathrm dt\]
and we define
\[\hat{u}(\hat{X}^1,\mu^2,\nu^2) = u(\hat{h}(\hat{X}^1),\mu^2,\nu^2),\]
where $\hat{h} =\hat{g}^{-1}$ and $\mu^2 := \mu_m^2$. As in Subsection \ref{firstred} we introduce a weight function and we define
\begin{eqnarray*}
 \hat{U}(\hat{X}^1,\mu^2,\nu^2) = \left( \frac{f_1}{s_{13}}(\hat{h}(\hat{X}^1)) \right)^{-\frac{1}{4}} u(\hat{h}(\hat{X}^1),\mu^2,\nu^2).
\end{eqnarray*}
After calculation, we obtain that $\hat{U}(\hat{X}^1,\mu^2,\nu^2)$ satisfies, in the variable $\hat{X}^1$, the Schr\"odinger equation
\begin{equation}\label{Schro2}
 - \ddot{\hat{U}}(\hat{X}^1,\mu^2,\nu^2) + \hat{q}_{\mu^2}(\hat{X}^1,\lambda) U(\hat{X}^1,\mu^2,\nu^2) = -\nu^2 U(\hat{X}^1,\mu^2,\nu^2),
\end{equation}
where,
\begin{align}\label{potq2}
\hat{q}_{\mu^2}(\hat{X}^1,\lambda) =& -(\lambda^2+1) \frac{s_{11}(\hat{X}^1)}{s_{13}(\hat{X}^1)} + \mu^2 \frac{s_{12}(\hat{X}^1)}{s_{13}(\hat{X}^1)} +\frac{1}{16} \left( \dot{ \left( \log \left( \frac{f_{1}(\hat{X}^1)}{s_{13}(\hat{X}^1)} \right) \right) }\right)^2 - \frac{1}{4} \ddot{\left( \log \left( \frac{f_{1}(\hat{X}^1)}{s_{13}(\hat{X}^1)} \right) \right)} \quad.
\end{align}
\noindent
As in Subsection \ref{firstdef} we can prove the following Lemma.

\begin{lemma}
 The potential $\hat{q}_{\mu^2}$ satisfies, at the end $\{ \hat{X}^1 = 0\}$,
 \[\hat{q}_{\mu^2}(\hat{X}^1,\lambda) = - \frac{\lambda^2 + \frac{1}{4}}{(\hat{X}^1)^2} + \hat{q}_{0,\mu^2}(\hat{X}^1,\lambda),\]
 where $\hat{X}^1 \hat{q}_{0,\mu^2}(\hat{X}^1,\lambda) \in L^1\left( 0, \frac{\hat{A}^1}{2} \right)$ with $\hat{A}^1 = \hat{g}(A)$. Similarly, at the end $\{\hat{X}^1 = \hat{A}^1\}$,
 \[\hat{q}_{\mu^2}(\hat{X}^1,\lambda) = - \frac{\lambda^2 + \frac{1}{4}}{(\hat{A}^1-\hat{X}^1)^2} + \hat{q}_{\hat{A}^1,\mu^2}(\hat{X}^1,\lambda),\]
  where $(\hat{A}^1-\hat{X}^1) \hat{q}_{\hat{A}^1,\mu^2}(\hat{X}^1,\lambda) \in L^1\left(\frac{\hat{A}^1}{2},\hat{A}^1 \right)$.
 \end{lemma}
\noindent
We can now follow the procedure of Subsection \ref{firstdef} to define the characteristic and Weyl-Titchmarsh functions corresponding to Equation (\ref{Schro2}) using two 
fondamental systems of solutions. Thus, we can define the characteristic functions
\begin{equation}\label{defDelta2}
 \hat{\Delta}_{\hat{q}_{\mu^2}}(\nu^2) = W(\hat{S}_{11},\hat{S}_{10}) \quad \textrm{and} \quad \hat{\delta}_{\hat{q}_{\mu^2}}(\nu^2) = W(\hat{S}_{11},\hat{S}_{20})
\end{equation}
and the Weyl-Titchmarsh function 
\begin{equation}\label{defM2}
 \hat{M}_{\hat{q}_{\mu^2}}(\nu^2) = -\frac{W(\hat{S}_{11},\hat{S}_{20})}{W(\hat{S}_{11},\hat{S}_{10})} = -\frac{\hat{\delta}_{\hat{q}_{\mu^2}}(\nu^2)}{\hat{\Delta}_{\hat{q}_{\mu^2}}(\nu^2)}.
\end{equation}
\noindent
Thanks to Corollary \ref{coroega} we immediately obtain the following Lemma.
\begin{lemma}\label{egWT1}
 \[ \Delta_{q_{\nu_m^2}}(\mu_m^2) = \hat{\Delta}_{\hat{q}_{\mu_m^2}}(\nu_m^2) \quad \textrm{and} \quad M_{q_{\nu_m^2}}(\mu_m^2) = \hat{M}_{\hat{q}_{\mu_m^2}}(\nu_m^2), \quad \forall m \geq 1.\]
\end{lemma}
\noindent
As in Subsection \ref{firstdef} the characteristic functions satisfies the following lemma.
\begin{lemma}\label{holo2}
 For any fixed $\mu$ the maps
 \[ \nu \mapsto \hat{\Delta}_{\hat{q}_{\mu^2}}(\nu^2) = \Delta_{q_{\nu^2}}(\mu^2) \quad \textrm{and} \quad \nu \mapsto \hat{\delta}_{\hat{q}_{\mu^2}}(\nu^2) = \delta_{q_{\nu^2}}(\mu^2)\]
 are entire.
\end{lemma}


\subsection{A third construction of characteristic and Weyl-Titchmarsh functions and application}\label{troisdef}

The aim of this Subsection is to show that, if we allow the angular momenta to be complex numbers, the characteristic functions $\Delta$ and $\delta$ are bounded on $(i\R)^2$. Thus,
in this Subsection $\mu_m$ and $\nu_m$ are assume to be in $i\R$.
In Subsections \ref{firstdef} and \ref{secdef} we defined the characteristic and the Weyl-Titchmarsh functions with $-\mu_m^2$ and $-\nu_m^2$ as the spectral parameter respectively. 
We now make a third choice of spectral parameter. We recall that the radial equation is given by (\ref{eqR}) and we rewrite this equation as
\[  -u'' + \frac{1}{2} (\log(f_1))' u' -(\lambda^2 +1)s_{11}u = - (\mu_m^2 + \nu_m^2) \left( \frac{\mu_m^2 s_{12} + \nu_m^2 s_{13}}{\mu_m^2 + \nu_m^2}\right)u.\]
We put, for $(y,y') \in \R^2$,
\[ \mu := \mu_m = iy, \quad \nu := \nu_m = iy', \quad \omega^2 := \mu^2 + \nu^2 \quad \textrm{and} \quad r_{\mu^2,\nu^2}(x^1) := \frac{\mu^2 s_{12}(x^1) + \nu^2 s_{13}(x^1)}{\mu^2 + \nu^2}.\]

\begin{remark}\label{estir}
 There exist some positive constants $c_1$ and $c_2$ such that for all $(\mu,\nu) \in (i\R)^2$ and $x^1 \in (0,A)$,
 \[ 0 < c_1 \leq r_{\mu^2,\nu^2}(x^1) \leq c_2 < +\infty.\]
\end{remark}
\noindent
To choose $-\omega^2$ as the spectral parameter we make a Liouville change of variables (that depends on $\mu^2$ and $\nu^2$ and is a kind of average of the previous ones):
\[\check{X}_{\mu^2,\nu^2}^1 = \check{g}_{\mu^2,\nu^2}(x^1) = \int_{0}^{x^1} \sqrt{r_{\mu^2,\nu^2}(t)}\, \mathrm dt.\]
For the sake of clarity, we put
\[ \check{X}^1 := \check{X}_{\mu^2,\nu^2}^1 \quad \textrm{and} \quad \check{g}(x^1) := \check{g}_{\mu^2,\nu^2}(x^1).\]
We define
\[\check{u}(\check{X}^1,\mu^2,\nu^2) = u(\check{h}(\check{X}^1),\mu^2,\nu^2),\]
where $\check{h} = \check{g}^{-1}$. As in Subsection \ref{firstred}, we introduce a weight function and we define
\begin{eqnarray*}
 \check{U}(\check{X}^1,\mu^2,\nu^2) = \left( \frac{f_1}{r_{\mu^2,\nu^2}}(\check{h}(\check{X}^1)) \right)^{-\frac{1}{4}} u(\check{h}(\check{X}^1),\mu^2,\nu^2).
\end{eqnarray*}
After calculation, we obtain that $\check{U}(\check{X}^1,\mu^2,\nu^2)$ satisfies, in the variable $\check{X}^1$, the Schr\"odinger equation
\begin{equation}\label{Schro3}
 - \ddot{\check{U}}(\check{X}^1,\mu^2,\nu^2) + \check{q}_{\mu^2,\nu^2}(\check{X}^1,\lambda) \check{U}(\check{X}^1,\mu^2,\nu^2) = -\omega^2 \check{U}(\check{X}^1,\mu^2,\nu^2),
\end{equation}
where,
\begin{align}\label{potq3}
\check{q}_{\mu^2,\nu^2}(\check{X}^1,\lambda) =& -(\lambda^2+1) \frac{s_{11}(\check{X}^1)}{r_{\mu^2,\nu^2}(\check{X}^1)} +\frac{1}{16} \left( \dot{ \left( \log \left( \frac{f_{1}(\check{X}^1)}{r_{\mu^2,\nu^2}(\check{X}^1)} \right) \right)}\right)^2 - \frac{1}{4} \ddot{\left( \log \left( \frac{f_{1}(\check{X}^1)}{r_{\mu^2,\nu^2}(\check{X}^1)} \right) \right)} \quad.
\end{align}

\begin{lemma}
 The potential $\check{q}_{\mu^2,\nu^2}$ satisfies, at the end $\{ \check{X}^1 = 0\}$,
 \[\check{q}_{\mu^2,\nu^2}(\check{X}^1,\lambda) = - \frac{\lambda^2 + \frac{1}{4}}{(\check{X}^1)^2} + \check{q}_{0,\mu^2,\nu^2}(\check{X}^1,\lambda),\]
 where $\check{X}^1 \check{q}_{0,\mu^2,\nu^2}(\check{X}^1,\lambda) \in L^1\left( 0, \frac{\check{A}^1}{2} \right)$ with $\check{A}^1 = \check{g}(A)$ and 
 $\check{q}_{0,\mu^2,\nu^2}$ is uniformly bounded for $(\mu,\nu) \in (i\R)^2$. Similarly, at the end $\{\check{X}^1 = \check{A}^1\}$,
  \[\check{q}_{\mu^2,\nu^2}(\check{X}^1,\lambda) = - \frac{\lambda^2 + \frac{1}{4}}{(\check{A}^1-\check{X}^1)^2} + \check{q}_{\check{A}^1,\mu^2,\nu^2}(\check{X}^1,\lambda),\]
 where $(\check{A}^1- \check{X}^1) \check{q}_{\check{A}^1,\mu^2,\nu^2}(\check{X}^1,\lambda) \in L^1\left(\frac{\check{A}^1}{2},\check{A}^1 \right)$ with $\check{A}^1 = \check{g}(A)$ and 
 $\check{q}_{\check{A}^1,\mu^2,\nu^2}$ is uniformly bounded for $(\mu,\nu) \in (i\R)^2$. 
\end{lemma}

\begin{remark}\label{estiA}
 Thanks to Remark \ref{estir}, we immediately obtain that there exist some positive constants $A^-$ and $A^+$ such that for all $(\mu,\nu) \in (i \R)^2$,
 \[ A^- \leq \check{A}^1 =: \check{A}_{\mu^2,\nu^2}^1 \leq A^+.\]
\end{remark}
\noindent
Once more, we follow the procedure of Subsection \ref{firstdef} to define the characteristic and Weyl-Titchmarsh functions corresponding to Equation (\ref{Schro3}) using two 
fondamental systems of solutions $\{\check{S}_{j0} \}_{j=1,2}$ and $\{\check{S}_{j1} \}_{j=1,2}$ satisfying the asymptotics (\ref{S_0})-(\ref{S_1}). Thus, we define the characteristic function
\begin{equation}\label{defDelta3}
 \check{\Delta}_{\check{q}_{\mu^2,\nu^2}}(\omega^2) = W(\check{S}_{11},\check{S}_{10}),
\end{equation}
and the Weyl-Titchmarsh function 
\begin{equation}\label{defM3}
 \check{M}_{\check{q}_{\mu^2,\nu^2}}(\omega^2) = -\frac{W(\check{S}_{11},\check{S}_{20})}{W(\check{S}_{11},\check{S}_{10})} =: -\frac{\check{\delta}_{\check{q}_{\mu^2,\nu^2}}(\omega^2)}{\check{\Delta}_{\check{q}_{\mu^2,\nu^2}}(\omega^2)}.
\end{equation}
As in Subsection \ref{secdef}, Lemma \ref{egWT1}, we can use the Corollary \ref{coroega} to prove the following Lemma.

\begin{lemma}\label{eqWT2}
 \[ \Delta_{q_{\nu^2}}(\mu^2) = \check{\Delta}_{\check{q}_{\mu^2,\nu^2}}(\omega^2) \quad \textrm{and} \quad M_{q_{\nu^2}}(\mu^2) = \check{M}_{\check{q}_{\mu^2,\nu^2}}(\omega^2), \quad \forall (\mu,\nu) \in (i\R)^2.\]
\end{lemma}
\noindent
We finish this Subsection following the proof of Proposition 3.2 of \cite{DaKaNi} and proving the following Proposition.

\begin{prop}\label{asym}
 For $\omega = iy$, where $\pm y \geq 0$, when $|\omega| \to \infty$,
 \[ \check{\Delta}_{\check{q}_{\mu^2,\nu^2}}(\omega^2) = \frac{ \Gamma(1-i\lambda)^2}{\pi 2^{2i\lambda}} \omega^{2i\lambda} e^{\pm \lambda \pi} 2 \cosh \left( \omega \check{A}^1 \mp \lambda \pi \right)[1]_{\epsilon},\]
 \[ \check{\delta}_{\check{q}_{\mu^2,\nu^2}}(\omega^2) = \frac{ \Gamma(1-i\lambda) \Gamma(1+i\lambda)}{2i\lambda \pi} 2 \cosh \left( \omega \check{A}^1 \right)[1]_{\epsilon},\]
 \[ \check{M}_{\check{q}_{\mu^2,\nu^2}}(\omega^2) = - \frac{\Gamma(1+i\lambda)^2e^{\mp \lambda \pi} 2^{2i\lambda}}{2i\lambda  \Gamma(1-i\lambda)} \omega^{-2i\lambda} \frac{\cosh \left( \omega \check{A}^1 \right)}{\cosh \left( \omega \check{A}^1 \mp \lambda \pi \right)}[1]_{\epsilon},\]
 where $[1]_{\epsilon} = O \left( \frac{1}{(\log|\omega|)^{\epsilon}} \right)$ when $|\omega| \to \infty$ with $\epsilon = \min(\epsilon_0,\epsilon_1)$.
 \end{prop}

\begin{proof}
 The only difference with Proposition 3.2 of \cite{DaKaNi} is the fact that our potential $q_{\mu^2,\nu^2}$ depends on the angular momenta $\mu^2$ and $\nu^2$. However, since 
 $\check{q}_{0,\mu^2,\nu^2}$ is uniformly bounded for $(\mu,\nu) \in (i\R)^2$, we obtain Proposition \ref{asym} without additional complication.
\end{proof}

\begin{coro}\label{bornedelta}
 There exists $C > 0$ such that for all $(\mu,\nu) \in (i \R)^2$,
 \[ |\Delta_{q_{\nu^2}}(\mu^2)| = |\check{\Delta}_{\check{q}_{\mu^2,\nu^2}}(\omega^2)|  \leq C\]
 and
 \[|\delta_{q_{\nu^2}}(\mu^2)| = |\check{\delta}_{\check{q}_{\mu^2,\nu^2}}(\omega^2)| \leq C.\]
\end{coro}

\begin{proof}
 This corollary is an immediate consequence of Proposition \ref{asym}, Remark \ref{estiA} and the definition of $\omega^2 = \mu^2 + \nu^2 \leq 0$ when $(\mu,\nu) \in (i\R)^2$.
\end{proof}

\section{The inverse problem at fixed energy for the angular equations}\label{invang}

The aim of this Section is to show the uniqueness of the angular part of the St\"ackel matrix i.e. of the second and the third lines. First, we prove that the block
\[ \begin{pmatrix} s_{22} & s_{23} \\ s_{32} & s_{33} \end{pmatrix}\]
is uniquely determined by the knowledge of the scattering matrix at a fixed energy using the fact that the scattering matrices act on the same space and 
the first invariance described in Proposition \ref{propinv}. Secondly, we use the decomposition on the generalized harmonics and the second invariance described in 
Proposition \ref{propinv} to prove the uniqueness of the coefficients $s_{21}$ and $s_{31}$. We finally show the uniqueness of the coupled spectrum which will be useful in the study of 
the radial part.

\subsection{A first reduction and a first uniqueness result}\label{firstred}

We first recall that (see (\ref{nvg})),
\[ g = \frac{(dx^1)^2 + d\Omega_{\mathcal{T}^2}^2 + P(x^1,x^2,x^3,dx^1,dx^2,dx^3)}{(x^1)^2}\]
and
\[ \tilde{g} = \frac{(dx^1)^2 + \tilde{d\Omega}_{\mathcal{T}^2}^2 + \tilde{P}(x^1,x^2,x^3,dx^1,dx^2,dx^3)}{(x^1)^2}.\]
Our main assumption is
\[S_g(\lambda) = S_{\tilde{g}}(\lambda),\]
where the equality holds as operators acting on $L^2(\mathcal{T}^2, dVol_{d\Omega_{\mathcal{T}^2}}; \C^2)$ with
\[ dVol_{d\Omega_{\mathcal{T}^2}} =  \sqrt{\det(d\Omega_{\mathcal{T}^2}^2)}.\]
Thus,
\[ \sqrt{\det(d\Omega_{\mathcal{T}^2}^2)} = \sqrt{\det(\tilde{d\Omega}_{\mathcal{T}^2}^2)},\]
since these operators have to act on the same space. Since,
\[ d\Omega_{\mathcal{T}^2}^2 = s^{11}((dx^2)^2 + (dx^3)^2) \quad \textrm{and} \quad \tilde{d\Omega}_{\mathcal{T}^2}^2 = \tilde{s}^{11}((dx^2)^2 + (dx^3)^2),\]
this equality implies that
\begin{equation}\label{s11}
 s^{11} = \tilde{s}^{11}.
\end{equation}
Using Remark \ref{Robsimpl}, we can obtain more informations from this equality. Indeed, we first note that,
\begin{eqnarray*}
 s^{11} &=& s_{22}s_{33} - s_{23}s_{32} \\
 &=& s_{22}s_{33} - (1+s_{22})(1+s_{33}) \\
 &=& - 1 - s_{22} - s_{33}.
\end{eqnarray*}
Thus, the equality (\ref{s11}), allows us to obtain
\[ s_{22} - \tilde{s}_{22} = \tilde{s}_{33} - s_{33}.\]
Since the left-hand side only depends on the variable $x^2$ and the right-hand side only depends on the variable $x^3$, we can deduce that there exists a constant $c \in \R$ 
such that
\[ s_{22} - \tilde{s}_{22} = c = \tilde{s}_{33} - s_{33}.\]
Using Remark \ref{Robsimpl} again, we can thus write
\[ \begin{pmatrix} s_{22} & s_{23} \\ s_{32} & s_{33} \end{pmatrix} = \begin{pmatrix} \tilde{s}_{22} & \tilde{s}_{23} \\ \tilde{s}_{32} & \tilde{s}_{33} \end{pmatrix} + c \begin{pmatrix} 1 & 1 \\ -1 & -1 \end{pmatrix},\]
or equivalently,
\[ \begin{pmatrix} s_{22} & s_{23} \\ s_{32} & s_{33} \end{pmatrix} = \begin{pmatrix} \tilde{s}_{22} & \tilde{s}_{23} \\ \tilde{s}_{32} & \tilde{s}_{33} \end{pmatrix} G,\]
where
\[G = \begin{pmatrix} 1-c & -c \\ c & 1+c \end{pmatrix},\]
is a constant matrix with determinant equals to $1$. Moreover, as it was mentioned in Proposition \ref{propinv}, if $\hat{S}$ is a second St\"ackel matrix such that
\[\begin{pmatrix} s_{i2} & s_{i3} \end{pmatrix} = \begin{pmatrix} \hat{s}_{i2} & \hat{s}_{i3} \end{pmatrix} G, \quad \forall i \in \{1,2,3\},\]
then
\[ g = \hat{g},\]
since
\[ s^{i1} = \hat{s}^{i1}, \quad \forall i \in \{1,2,3\}.\]
The presence of the matrix $G$ is then due to an invariance of the metric $g$. We can thus assume that $G = I_2$. We conclude that
\begin{equation}\label{egs23}
 \begin{pmatrix} s_{22} & s_{23} \\ s_{32} & s_{33} \end{pmatrix} = \begin{pmatrix} \tilde{s}_{22} & \tilde{s}_{23} \\ \tilde{s}_{32} & \tilde{s}_{33} \end{pmatrix}.
\end{equation}

\subsection{End of the inverse problem for the angular part}\label{secuni}

The aim of this Subsection is to show that the coefficients $s_{21}$ and $s_{31}$ are uniquely defined.
First, since $\{\tilde{Y}_m\}_{m \geq 1}$ is a Hilbertian basis of $L^2(\mathcal{T}^2, s^{11} dx^2 dx^3)$, we can deduce that, for all $m \in \N \setminus \{0\}$, there exists a subset $E_m \subset \N \setminus \{0\}$ such that 
\[ Y_m = \sum_{p \in E_m} c_p \tilde{Y}_p.\]
We recall that, thanks to (\ref{defLetH}),
\[ \begin{pmatrix} H \\ L \end{pmatrix} = \frac{1}{s^{11}} \begin{pmatrix}
               s_{32} & -s_{22} \\ 
               -s_{33} & s_{23}
              \end{pmatrix} \begin{pmatrix} A_2 \\ A_3 \end{pmatrix},\]
where $A_j$, $j \in \{2,3\}$, were defined in (\ref{defAi}). Clearly,
\begin{equation}\label{lienentreAetHL}
 \begin{pmatrix} A_2 \\ A_3 \end{pmatrix} = T \begin{pmatrix} H \\ L \end{pmatrix},
\end{equation}
where
\[ T = - \begin{pmatrix}
               s_{23} & s_{22} \\ 
               s_{33} & s_{32}
              \end{pmatrix}.\]
We recall that
\[ \tilde{T} = - \begin{pmatrix}
               \tilde{s}_{23} & \tilde{s}_{22} \\ 
               \tilde{s}_{33} & \tilde{s}_{32}
              \end{pmatrix}= T.\]
We finally deduce from (\ref{lienentreAetHL}) that
\[ -\begin{pmatrix} \partial_2^2 \\ \partial_3^2 \end{pmatrix} = T \begin{pmatrix} H \\ L \end{pmatrix} + (\lambda^2+1) \begin{pmatrix} s_{21} \\ s_{31} \end{pmatrix}\]
and we then obtain that
\begin{equation}\label{eg1}
 T \begin{pmatrix} H \\ L \end{pmatrix} + (\lambda^2+1) \begin{pmatrix} s_{21} \\ s_{31} \end{pmatrix} = T \begin{pmatrix} \tilde{H} \\ \tilde{L} \end{pmatrix} + (\lambda^2+1) \begin{pmatrix} \tilde{s}_{21} \\ \tilde{s}_{31} \end{pmatrix}.
\end{equation}

\begin{lemma}\label{HLcont}
 For all $m \geq 1$,
 \[ \tilde{H} \left(\sum_{p \in E_m} c_p \tilde{Y}_p \right) = \sum_{p \in E_m} c_p \tilde{H}(\tilde{Y}_p) \quad \textrm{and} \quad \tilde{L}\left(\sum_{p \in E_m} c_p \tilde{Y}_p\right) = \sum_{p \in E_m} c_p \tilde{L}(\tilde{Y}_p).\]
\end{lemma}

\begin{proof}
 We recall that $\tilde{H}$ is selfadjoint. Thus the operator $(\tilde{H}+i)^{-1}$ is bounded. Similarly the operator $(\tilde{H}+i)^{-1} \tilde{H}$ is bounded. Thus,
\begin{eqnarray*}
 (\tilde{H}+i)^{-1} \tilde{H} \left(\sum_{p \in E_m} c_p \tilde{Y}_p \right) &=& \sum_{p \in E_m} c_p (\tilde{H}+i)^{-1} \tilde{H}(\tilde{Y}_p )\\
 &=& (\tilde{H}+i)^{-1} \left( \sum_{p \in E_m} c_p \tilde{H}(\tilde{Y}_p ) \right).
\end{eqnarray*}
Multiplying by $(\tilde{H}+i)$ from the left we obtain the result of the Lemma. Note that the sum $\sum_{p \in E_m} c_p \tilde{H}(\tilde{Y}_p)$ is finite because the coefficients $c_p$ are sufficiently decreasing.
Indeed, we note that $c_p = \langle Y_m , \tilde{Y}_p \rangle$ and we can use integration by parts with the help of the operator $H$ and the Weyl law on the eigenvalues to 
obtain the decay we want.
\end{proof}

\begin{remark}
 If $E_m$, $m \in \N \setminus \{0\}$, are finite then Lemma \ref{HLcont} is obvious. In fact, following the idea of \cite{DaKaNi} Proposition 4.1, 
 we could obtain that this sets are finite using asymptotics of the Weyl-Titchmarsh function.
\end{remark}
\noindent
Applying (\ref{eg1}) to the vector of generalized harmonics
\[ \begin{pmatrix} Y_m \\ Y_m \end{pmatrix} = \begin{pmatrix} \sum_{p \in E_m} c_p\tilde{Y}_p \\ \sum_{p \in E_m} c_p\tilde{Y}_p\end{pmatrix}\]
we obtain, thanks to Lemma \ref{HLcont} and (\ref{LHvp}), that,
\begin{eqnarray*}
  \left( T \begin{pmatrix} H \\ L \end{pmatrix} + (\lambda^2 +1) \begin{pmatrix} s_{21} \\ s_{31} \end{pmatrix} \right) \begin{pmatrix} Y_m \\ Y_m \end{pmatrix}
    &=&  \left( T \begin{pmatrix} \mu_m^2 \\ \nu_m^2 \end{pmatrix} + (\lambda^2 +1) \begin{pmatrix} s_{21} \\ s_{31} \end{pmatrix} \right) \begin{pmatrix} Y_m \\ Y_m \end{pmatrix} \\
 &=&\sum_{p \in E_m} c_p \left( T \begin{pmatrix} \mu_m^2 \\ \nu_m^2 \end{pmatrix} + (\lambda^2 +1) \begin{pmatrix} s_{21} \\ s_{31} \end{pmatrix} \right) \begin{pmatrix} \tilde{Y}_p \\ \tilde{Y}_p \end{pmatrix} \\
  \end{eqnarray*}
and
\begin{eqnarray*}
  \left( T \begin{pmatrix} H \\ L \end{pmatrix} + (\lambda^2 +1) \begin{pmatrix} s_{21} \\ s_{31} \end{pmatrix} \right) \begin{pmatrix} Y_m \\ Y_m \end{pmatrix}
  &=&   \sum_{p \in E_m} c_p  \left( T \begin{pmatrix} \tilde{H} \\ \tilde{L} \end{pmatrix} + (\lambda^2 +1) \begin{pmatrix} \tilde{s}_{21} \\ \tilde{s}_{31} \end{pmatrix} \right) \begin{pmatrix} \tilde{Y}_p \\ \tilde{Y}_p \end{pmatrix} \\
  &=&   \sum_{p \in E_m} c_p  \left( T \begin{pmatrix} \tilde{\mu}_p^2 \\ \tilde{\nu}_p^2 \end{pmatrix} + (\lambda^2 +1) \begin{pmatrix} \tilde{s}_{21} \\ \tilde{s}_{31} \end{pmatrix} \right) \begin{pmatrix} \tilde{Y}_p \\ \tilde{Y}_p \end{pmatrix}.
  \end{eqnarray*}
Hence,
   \[ \sum_{p \in E_m} c_p \left( T \begin{pmatrix} \mu_m^2 \\ \nu_m^2 \end{pmatrix} + (\lambda^2 +1) \begin{pmatrix} s_{21} \\ s_{31} \end{pmatrix} \right) \begin{pmatrix} \tilde{Y}_p \\ \tilde{Y}_p \end{pmatrix}
   = \sum_{p \in E_m} c_p  \left( T \begin{pmatrix} \tilde{\mu}_p^2 \\ \tilde{\nu}_p^2 \end{pmatrix} + (\lambda^2 +1) \begin{pmatrix} \tilde{s}_{21} \\ \tilde{s}_{31} \end{pmatrix} \right) \begin{pmatrix} \tilde{Y}_p \\ \tilde{Y}_p \end{pmatrix}.\]
Since $\{\tilde{Y}_p\}_{p \geq 1}$ is a Hilbertian basis we deduce from this equality that for all $m \geq 1$, for all $p \in E_m$,
\begin{equation}\label{eqmp}
  T \begin{pmatrix} \mu_m^2 \\ \nu_m^2 \end{pmatrix} + (\lambda^2 +1) \begin{pmatrix} s_{21} \\ s_{31} \end{pmatrix} 
  = T \begin{pmatrix} \tilde{\mu}_p^2 \\ \tilde{\nu}_p^2 \end{pmatrix} + (\lambda^2 +1) \begin{pmatrix} \tilde{s}_{21} \\ \tilde{s}_{31} \end{pmatrix}.
\end{equation}
We deduce from (\ref{eqmp}) that,
\[ \begin{pmatrix}  \tilde{\mu}_{p}^2 - \mu_{m}^2 \\ \tilde{\nu}_{p}^2 - \nu_{m}^2 \end{pmatrix} = (\lambda^2 +1) T^{-1} \begin{pmatrix} s_{21} - \tilde{s}_{21} \\ s_{31} - \tilde{s}_{31} \end{pmatrix}.\]
Since the right-hand side is independent of $m$ and $p$, we can deduce from this equality that there exists a constant vector
\[ \begin{pmatrix}  c_1 \\ c_2 \end{pmatrix},\]
such that
\[ \begin{pmatrix}  \tilde{\mu}_{p}^2  \\ \tilde{\nu}_{p}^2 \end{pmatrix} = \begin{pmatrix}  \mu_{m}^2  \\ \nu_{m}^2 \end{pmatrix} + \begin{pmatrix}  c_1 \\ c_2 \end{pmatrix},\]
and
\begin{equation}\label{s23ci}
  \begin{pmatrix}
               s_{21} \\ 
               s_{31}
              \end{pmatrix} = \begin{pmatrix}
               \tilde{s}_{21} \\ 
               \tilde{s}_{31}  \end{pmatrix} + \frac{1}{\lambda^2+1} T \begin{pmatrix}  c_1 \\ c_2 \end{pmatrix}.
\end{equation}

\noindent
From (\ref{s23ci}), we immediately deduce that
\begin{equation}\label{eqcol1}
 \begin{cases}
s_{21}(x^2) = \tilde{s}_{21}(x^2) - C_1 s_{23}(x^2) - C_2 s_{22}(x^2) \\
s_{31}(x^3) = \tilde{s}_{31}(x^3) - C_1 s_{33}(x^3) - C_2 s_{32}(x^3)
\end{cases},
\end{equation}
where $C_i = \frac{c_i}{\lambda^2+1}$ for $i \in \{1,2\}$.
We recall that 
\[g = \sum_{i=1}^3 H_i^2 (dx^i)^2 \quad \textrm{with} \quad H_i^2 = \frac{\det(S)}{s^{i1}} \quad \forall i \in \{1,2,3\}.\]
Since the minors $s^{i1}$ only depend on the second and the third column, they don't change under the transformation given in (\ref{eqcol1}).
Thus, as mentioned in the Introduction, Proposition \ref{propinv}, recalling that the determinant of a matrix doesn't 
change if we add to the first column a linear combination of the second and the third columns, we conclude that the equalities (\ref{eqcol1}) describe an invariance of the 
metric $g$ under the definition of the St\"ackel matrix $S$. We can then choose $C_i=0$, $i \in \{1,2\}$, i.e. $c_1 = c_2 =0$. Finally, we have shown that
\begin{equation}\label{col1}
  \begin{pmatrix}
               s_{21} \\ 
               s_{31}
              \end{pmatrix} = \begin{pmatrix}
               \tilde{s}_{21} \\ 
               \tilde{s}_{31}  \end{pmatrix}.
\end{equation}
From the definition of the operators $L$ and $H$ given by (\ref{defLetH}), we deduce from (\ref{egs23}) and (\ref{col1}) that
\[ H = \tilde{H} \quad \textrm{and} \quad L = \tilde{L}.\]
We then conclude that these operators have the same eigenfunctions, i.e. we can choose 
\begin{equation}\label{exthyp}
Y_m = \tilde{Y}_m, \quad \forall m \geq 1
\end{equation}
and the same coupled spectrum
\begin{equation}\label{egmom}
 \begin{pmatrix}  \mu_{m}^2  \\ \nu_{m}^2 \end{pmatrix} = \begin{pmatrix}  \tilde{\mu}_{m}^2  \\ \tilde{\nu}_{m}^2 \end{pmatrix}, \quad \forall m \geq 1.
\end{equation}

\section{The inverse problem at fixed energy for the radial equation}\label{invrad}

The aim of this Section is to show that the radial part of the St\"ackel matrix, i.e. the first line, is uniquely determined by the knowledge of the scattering matrix.
We first use a multivariable version of the 
Complex Angular Momentum method to extend the equality of the Weyl-Titchmarsh functions (valid on the coupled spectrum) to complex angular momenta. In a second time, we use the Borg-Marchenko Theorem (see 
for instance \cite{FY,KST}) to obtain the uniqueness of the quotients
\[ \frac{s_{11}}{s_{12}} \quad \textrm{and} \quad \frac{s_{11}}{s_{13}}.\]

\subsection{Complexification of the Angular Momenta}

We recall that, thanks to our main assumption in Theorem \ref{main}, (\ref{exthyp})-(\ref{egmom}) and (\ref{lienLM}), we know that,
\begin{equation}\label{hypdep}
 M(\mu_m^2,\nu_m^2) = \tilde{M}(\mu_m^2,\nu_m^2), \quad \forall m \geq 1,
\end{equation}
where
\[M(\mu_m^2,\nu_m^2) = M_{q_{\nu_m^2}}(\mu_m^2) = M_{\hat{q}_{\mu_m^2}}(\nu_m^2)\]
and
\[\tilde{M}(\mu_m^2,\nu_m^2) = M_{\tilde{q}_{\nu_m^2}}(\mu_m^2) = M_{\tilde{\hat{q}}_{\mu_m^2}}(\nu_m^2).\]
The aim of this Subsection is to show that
\begin{equation}\label{CAMeq}
 M(\mu^2,\nu^2) = \tilde{M}(\mu^2,\nu^2), \quad \forall (\mu,\nu) \in \C^2 \setminus P, 
\end{equation}
where $P$ is the set of points $(\mu,\nu) \in \C^2$ such that $(\mu^2,\nu^2)$ is a pole of $M$ and $\tilde{M}$, or equivalently which is a zero of $\Delta$ and $\tilde{\Delta}$.
Usually, in the Complexification of the Angular Momentum method there is only one angular momentum which we complexify using
uniqueness results for holomorphic 
functions in certain classes. In \cite{DNK2}, there are two \emph{independent} angular momenta and the authors are able to use the Complexification of the Angular Momentum method for only 
one angular momentum. Here, we cannot fix one angular momentum and let the other one belong to $\C$ since the two angular momenta are not independent (see 
Lemma \ref{distrivp}). We thus have to complexify simultaneously the two angular momenta and we then need to use uniqueness results for multivariable holomorphic functions.
Therefore, to obtain (\ref{CAMeq}) we want to use the following result given in \cite{Ber,Bl}.

\begin{theorem}\label{bloom}
 Let $K$ be an open cone in $\R^2$ with vertex the origin and $T(K) = \{ z \in \C^2,$ $\mathrm{Re}(z) \in K \}$. Suppose that $f$ is a bounded and an analytic function on $T(K)$. Let $E$ be a 
 discrete subset of $K$ such that for some constant $h > 0$, $|e_1 - e_2| \geq h$, for all $(e_1,e_2) \in E$. Let $n(r) = \# E \cap B(0,r)$. Assume that $f$ vanishes on $E$. 
 Then $f$ is identically zero if
 \[ \varlimsup \frac{n(r)}{r^2} > 0, \quad r \to + \infty.\]
\end{theorem}
\noindent
We first introduce the function
\[\begin{array}{ccccl}
\psi & : & \C^2 & \to & \C \\
 & & (\mu,\nu) & \mapsto & \tilde{\Delta}(\mu^2,\nu^2)\delta(\mu^2,\nu^2) - \Delta(\mu^2,\nu^2)\tilde{\delta}(\mu^2,\nu^2)\\
\end{array},\]
with,
\[\delta(\mu^2,\nu^2) = \delta_{q_{\nu^2}}(\mu^2) = \delta_{\hat{q}_{\mu^2}}(\nu^2),\]
where $\delta_{q_{\nu^2}}(\mu^2)$ and $\delta_{\hat{q}_{\mu^2}}(\nu^2)$ were defined in (\ref{defDelta1}) and
\[\Delta(\mu^2,\nu^2) = \Delta_{q_{\nu^2}}(\mu^2) = \Delta_{\hat{q}_{\mu^2}}(\nu^2),\]
where $\Delta_{q_{\nu^2}}(\mu^2)$ and $\Delta_{\hat{q}_{\mu^2}}(\nu^2)$ were defined in (\ref{defDelta2}). Our aim is then to show that $\psi$ is identically zero.

\begin{lemma}\label{lem1}
 The map $\psi$ is entire as a function of two complex variables.
\end{lemma}

\begin{proof}
 We use Hartogs Theorem (see for instance \cite{Horman}) which states that a function of two complex variables is holomorphic if and only if this function is 
 holomorphic with respect to each variable separately. Indeed, thanks to Lemmas \ref{holo1} and \ref{holo2}, we can then immediately conclude.
\end{proof}
\noindent
To use Theorem \ref{bloom} we need the following estimate on the function $\psi$.

\begin{lemma}\label{lem2}
There exist some positive constants $C$, $A$ and $B$ such that
 \[| \psi(\mu,\nu) | \leq C e^{A |\mathrm{Re}(\mu)| + B| \mathrm{Re}(\nu)|}, \quad \forall (\mu,\nu) \in \C^2.\]
\end{lemma}

\begin{proof}
 The proof of this Lemma consists in four steps. \\
 \underline{Step 1:} We claim that for all fixed $\nu \in \C$ there exists a constant $C_1(\nu)$ such that for all $\mu \in \C$
 \[| \psi(\mu,\nu) | \leq C_1(\nu) e^{A |\mathrm{Re}(\mu)|}.\]
 To obtain this estimate we study the solutions $S_{j0}$ and $S_{j1}$ defined in Subsection \ref{firstdef}.
 First, we show that for $j \in \{1,2\}$,
 \[ |S_{j0}(X^1,\mu^2,\nu^2) |\leq C(\nu) \frac{e^{|\mathrm{Re}(\mu)|X^1}}{|\mu|^{\frac{1}{2}}},\]
 \[ |S_{j0}'(X^1,\mu^2,\nu^2) |\leq C(\nu) {|\mu|^{\frac{1}{2}}} e^{|\mathrm{Re}(\mu)|X^1},\]
 \[ |S_{j1}(X^1,\mu^2,\nu^2) |\leq C(\nu) \frac{e^{|\mathrm{Re}(\mu)|(A^1-X^1)}}{|\mu|^{\frac{1}{2}}},\]
 \[ |S_{j1}'(X^1,\mu^2,\nu^2) |\leq C(\nu) {|\mu|^{\frac{1}{2}}} e^{|\mathrm{Re}(\mu)|(A^1-X^1)}.\]
 As in \cite{DaKaNi}, we can show by an iterative procedure that
 \begin{equation}\label{estimSj0}
  |S_{j0}(X^1,\mu^2,\nu^2)| \leq C \left( \frac{X^1}{1+|\mu| X^1} \right)^{\frac{1}{2}} e^{|\mathrm{Re}(\mu)|X^1} \exp \left( \int_{0}^{X^1} \frac{ t |q_{0,\nu^2}(t)|}{1+|\mu|t} \, dt \right).
 \end{equation}
 Recall now that, thanks to the asymptotically hyperbolic structure, we have for all $X^1 \in (0,X_0^1)$, where $X_0^1 \in (0,A^1)$ is fixed,
 \[ t |q_{0,\nu^2}(t)| \leq \frac{C(1+\nu^2)}{t(1+|\log(t)|)^{1+\epsilon_0}}, \quad \forall t \in (0,X).\]
 Thus, as shown in Subsection 3.1 of \cite{DaKaNi},
 \begin{equation}\label{estintq}
   \int_{0}^{X^1} \frac{ t |q_{0,\nu^2}(t)|}{1+|\mu|t} \, dt \leq (1+\nu^2) O \left( \frac{1}{(\log(|\mu|))^{\epsilon_0}} \right).
 \end{equation}
 We can then conclude
  \[ |S_{j0}(X^1,\mu^2,\nu^2)| \leq C(\nu) \frac{e^{|\mathrm{Re}(\mu)|X^1}}{|\mu|^{\frac{1}{2}}}.\]
 The result on $S_{j0}'(X^1,\mu^2,\nu^2)$ is obtained similarly using the estimate on the derivative of the Green kernel given in Proposition 3.1 of \cite{DaKaNi}. By symmetry, we also 
 obtain the corresponding estimates on $S_{j1}(X^1,\mu^2,\nu^2)$ and $S_{j1}'(X^1,\mu^2,\nu^2)$. We can then conclude that
 \[ \Delta(\mu^2,\nu^2) = \Delta_{q_{\nu^2}}(\mu^2) = W(S_{11},S_{10})\]
 and
  \[ \delta(\mu^2,\nu^2) = \delta_{q_{\nu^2}}(\mu^2) = W(S_{11},S_{20}),\]
 satisfy
 \[ |\Delta(\mu^2,\nu^2)| \leq C_1(\nu) e^{A| \mathrm{Re}(\mu)|} \quad \textrm{and} \quad |\delta(\mu^2,\nu^2)| \leq C_1(\nu) e^{A| \mathrm{Re}(\mu)|}, \quad \forall (\mu,\nu) \in \C^2.\]
 Finally, we have shown that
  \[| \psi(\mu,\nu) | \leq C_1(\nu) e^{A |\mathrm{Re}(\mu)|}, \quad \forall (\mu,\nu) \in \C^2.\]
\underline{Step 2:} We can also show that for all fixed $\mu \in \C$ there exists a constant $C_2(\mu)$ such that for all $\nu \in \C$
 \[| \psi(\mu,\nu) | \leq C_2(\mu) e^{\hat{A} |\mathrm{Re}(\nu)|}.\]
 To obtain this estimate we use the strategy of the first step on Equation (\ref{Schro2}) with potential (\ref{potq2}) introduced in Subsection \ref{secdef}.\\
\underline{Step 3:} Thanks to Corollary \ref{bornedelta}, there exists a constant $C$ such that for all $(y,y') \in \R^2$,
 \begin{equation}\label{borne}
 | \psi(iy,iy') | \leq C.
 \end{equation}
 \underline{Step 4:} We finish the proof of the Lemma by the use of the Phragm\'en-Lindel\"of Theorem (see \cite{Bo} Theorem 1.4.3).
 We first fix $\nu \in i\R$. Thus, the application $\mu \mapsto \psi(\mu,\nu)$ satisfies
 \[ \begin{cases}
|\psi(\mu,\nu)| \leq C_1(\nu) e^{A|\mathrm{Re}(\mu)|}, \quad \forall \mu \in \C, \quad \quad \textrm{(Step} \, \,\textrm{1)}\\
|\psi(\mu,\nu)| \leq C, \quad \forall \mu \in i\R, \quad \quad \textrm{(Step} \, \,\textrm{3)}
\end{cases}.\]
Thanks to the Phragm\'en-Lindel\"of Theorem, we deduce from these equalities that
\[|\psi(\mu,\nu)| \leq C e^{A|\mathrm{Re}(\mu)|}, \quad \forall (\mu,\nu) \in \C \times i \R.\]
We now fix $\mu \in \C$, then the application $\nu \mapsto \psi(\mu,\nu)$ satisfies
 \[ \begin{cases}
|\psi(\mu,\nu)| \leq C_2(\mu) e^{B|\mathrm{Re}(\nu)|}, \quad \forall \nu \in \C, \quad \quad \textrm{(Step} \, \,\textrm{1)}\\
|\psi(\mu,\nu)| \leq C e^{A|\mathrm{Re}(\mu)|}, \quad \forall \nu \in i\R
\end{cases}.\]
Thus, using once more the Phragm\'en-Lindel\"of Theorem, we obtain
 \[| \psi(\mu,\nu) | \leq C e^{A |\mathrm{Re}(\mu)| + B| \mathrm{Re}(\nu)|}, \quad \forall (\mu,\nu) \in \C^2.\]
\end{proof}
\noindent
We apply Theorem \ref{bloom} with
\[ K = (\R^+)^2 \quad \textrm{and} \quad F(\mu,\nu)= \psi(\mu,\nu) e^{-A\mu - B \nu}.\]

\begin{lemma}\label{lem3}
 The application $F$ is bounded and holomorphic on
 \[ T((\R^+)^2) = \{(\mu,\nu) \in \C^2, \, \, (\mathrm{Re}(\mu),\mathrm{Re}(\nu)) \in \R^+ \times \R^+ \}.\]
\end{lemma}

\begin{proof}
 This lemma is an immediate consequence of Lemmas \ref{lem1} and \ref{lem2}.
\end{proof}
\noindent
We now recall that $(\mu_m^2,\nu_m^2)$, $m \geq 1$, denotes the coupled spectrum of the operators $H$ and $L$. We note that $\mu_m^2$ and $\nu_m^2$ tend to $+ \infty$, 
as $m \to + \infty$. Therefore, there exists $M \geq 1$ such that $\mu_m^2 \geq 0$ and $\nu_m^2 \geq 0$ for all $m \geq M$. We then set
\[ E_M = \{(|\mu_m|,|\nu_m|), \, \, m \geq M \}.\]
Thanks to Equation (\ref{hypdep}), we note that, the application $F$ satisfies
 \[ F(\mu_m,\nu_m) = 0, \quad \forall m \geq M,\]
 since
\begin{equation}\label{psi0}
 \psi(\mu_m,\nu_m) = 0, \quad \forall m \geq M.
\end{equation}
Moreover, since the characteristic functions are, by definition, even functions with respect to $\mu$ and $\nu$, we obtain that
 \[ F(|\mu_m|,|\nu_m|) = F(\mu_m,\nu_m) = 0, \quad \forall m \geq M,\]
i.e. that $F$ vanishes on $E_M$.

\begin{remark}
We emphasize that $E_M$ denotes the set of eigenvalues counted with multiplicity (which is at most 4). Since we need a separation property given in the following Lemma
to apply the Bloom's Theorem, we have to consider a new set, also denoted by $E_M$, which corresponds to the previous set of eigenvalues counted without multiplicity.
To obtain this separation property on the coupled spectrum $E_M$, we also need to restrict our analysis to a suitable cone given in the following Lemma.
\end{remark}

\begin{lemma}\label{separationBloom}
We set
\begin{equation}\label{definitionducone}
\mathcal{C} = \{(\mu^2,\theta^2 \mu^2), \quad c_1 + \epsilon \leq \theta^2 \leq c_2 - \epsilon\}, \quad 0 < \epsilon << 1,
\end{equation}
where
\[ c_1 = \max \left( -\frac{s_{32}}{s_{33}} \right) \quad \textrm{and} \quad c_2 = \min \left( -\frac{s_{22}}{s_{23}} \right).\]
In that case, there exists $h > 0$ such that $|e_1 - e_2| \geq h$ for all $(e_1,e_2) \in (E_M \cap \mathcal{C})^2$, $e_1 \neq e_2$.
\end{lemma}

\begin{proof}
 See Appendix \ref{apsep}.
\end{proof}

\begin{remark}\label{rkconerestreint}
 We note that, as we have shown in Lemma \ref{distrivp}, we know that there exist real constants $C_1$, $C_2$, $D_1$ and $D_2$ such that for all $m \geq 1$,
\[ C_1 \mu_m^2 + D_1 \leq \nu_{m}^2 \leq C_2 \mu_m^2 + D_2,\]
where
\[ C_1 = \min \left( -\frac{s_{32}}{s_{33}} \right) > 0 \quad \textrm{and} \quad C_2 = -\min \left(\frac{s_{22}}{s_{23}} \right) > 0.\]
We then easily obtain that
\[ 0 < C_1 \leq c_1 < c_2 \leq C_2.\]
\end{remark}

\begin{lemma}\label{number}
 We set
 \[n(r) = \# E_M \cap B(0,r) \cap \mathcal{C},\]
 where $\mathcal{C}$ was defined in (\ref{definitionducone}), then
  \[ \varlimsup \frac{n(r)}{r^2} > 0, \quad r \to + \infty.\]
\end{lemma}

\begin{proof}
 See Appendix \ref{ap}.
\end{proof}

\begin{remark}
 We emphasize that the number of points of the coupled spectrum $n(r)$ we use to apply the Bloom's Theorem is not exactly the one we compute in the framework of Colin de 
 Verdi\`ere. Indeed, Colin de Verdi\`ere computes the number of points of the coupled spectrum counting multiplicity whereas $n(r)$ denotes the number of points of the 
 coupled spectrum counting without multiplicity. However, as we have seen before (see Remark \ref{remmult}) the multiplicity of a coupled eigenvalue is at most $4$. Therefore, 
 $n(r)$ is greater than the quarter of the number computed in the work of Colin de Verdi\`ere and is thus still of order $r^2$.
\end{remark}

\noindent
We can then use Theorem \ref{bloom} on the cone $\mathcal{C}$ to conclude that
\[ F(\mu,\nu) = 0, \quad \forall (\mu,\nu) \in \C^2.\]
From this equality we immediately deduce that
\[M(\mu^2,\nu^2) = \tilde{M}(\mu^2,\nu^2), \quad \forall (\mu,\nu) \in \C^2 \setminus P,\]
where $P$ is the set of points $(\mu,\nu) \in \C^2$ such that $(\mu^2,\nu^2)$ is a zero of the characteristic functions $\Delta$ and $\tilde{\Delta}$.

\subsection{Inverse problem for the radial part}

With the help of a multivariable version of the Complex Angular Momentum method we have shown
\[M(\mu^2,\nu^2) = \tilde{M}(\mu^2,\nu^2), \quad \forall (\mu,\nu) \in \C^2 \setminus P, \]
where $P$ is the set of points $(\mu,\nu) \in \C^2$ such that $(\mu^2,\nu^2)$ is a zero of the characteristic function $\Delta$. By definition, it means that
\[  M_{q_{\nu^2}}(\mu^2) =  M_{\tilde{q}_{\nu^2}}(\mu^2), \quad \forall (\mu,\nu) \in \C^2\setminus P,\]
where $M_{q_{\nu^2}}(\mu^2)$ was defined in (\ref{defM1}). We can thus use the celebrated Borg-Marchenko Theorem in the form given in \cite{DaKaNi,FY} and recalled in the
Introduction (see (\ref{BorgMar1})-(\ref{BorgMar2})) to obtain that
\[  q_{\nu_m^2}(X^1,\lambda) =  \tilde{q}_{\nu_m^2}(X^1,\lambda), \quad \forall m \geq 1, \quad \forall X^1 \in (0,A^1).\]
Thanks to (\ref{potq}), and since the previous equality is true for all $m \geq 1$, we then have, for all $X^1 \in (0,A^1)$,
\begin{equation}\label{eq1}
\frac{s_{13}(X^1)}{s_{12}(X^1)} = \frac{\tilde{s}_{13}(X^1)}{\tilde{s}_{12}(X^1)}
\end{equation}
and
\begin{align}\label{eq3}
-&(\lambda^2+1) \frac{s_{11}(X^1)}{s_{12}(X^1)} +\frac{1}{16} \left( \dot{\left( \log \left( \frac{f_{1}(X^1)}{s_{12}(X^1)} \right) \right)}\right)^2 - \frac{1}{4} \ddot{\left( \log \left( \frac{f_{1}(X^1)}{s_{12}(X^1)} \right) \right)} \nonumber\\
=& -(\lambda^2+1) \frac{\tilde{s}_{11}(X^1)}{\tilde{s}_{12}(X^1)}+\frac{1}{16} \left( \dot{ \left( \log \left( \frac{\tilde{f}_{1}(X^1)}{\tilde{s}_{12}(X^1)} \right) \right)}\right)^2 - \frac{1}{4} \ddot{\left( \log \left( \frac{\tilde{f}_{1}(X^1)}{\tilde{s}_{12}(X^1)} \right) \right)} \quad.
\end{align}
We want to rewrite this equation as a Cauchy problem for a second order non-linear differential equation with boundary conditions at the end $\{X^1 = 0\}$. To do that, we put
\[f = \frac{s_{11}}{s_{12}}, \quad h = \frac{s_{12}}{f_1}, \quad l = \frac{s_{13}}{s_{12}} = \tilde{l} \quad \textrm{and} \quad u = \left( \frac{h}{\tilde{h}} \right)^{\frac{1}{4}}.\]
We can thus rewrite (\ref{eq3}) into the form
\begin{equation}\label{eq4}
u'' + \frac{1}{2} (\log(\tilde{h}))' u' + (\lambda^2 + 1) (\tilde{f} - f)u = 0.
\end{equation}
Using the Robertson condition (\ref{Rob1}) we can write
\begin{equation}\label{ega}
 \frac{s_{11}}{s_{12}} = f = - \frac{s^{12}}{s^{11}} - l \frac{s^{13}}{s^{11}} + h(ls_{32} - s_{33})(s_{23}-ls_{22}).
\end{equation}
\begin{remark}
 Thanks to (\ref{ega}), we see that we can write $\frac{s_{11}}{s_{12}}$ as a function of $\frac{s_{13}}{s_{12}}$ and $\frac{f_1}{s_{12}}$, i.e.
 \begin{equation}\label{rk}
 \frac{s_{11}}{s_{12}} = \Phi \left( \frac{s_{13}}{s_{12}}, \frac{f_{1}}{s_{12}} \right) \quad \textrm{and} \quad \frac{\tilde{s}_{11}}{\tilde{s}_{12}} = \Phi \left( \frac{s_{13}}{s_{12}}, \frac{\tilde{f}_{1}}{\tilde{s}_{12}} \right),
 \end{equation}
 where
 \[ \Phi(X,Y) = - \frac{s^{12}}{s^{11}} - X \frac{s^{13}}{s^{11}} + \frac{1}{Y}(Xs_{32} - s_{33})(s_{23}-Xs_{22}).\]
 Thus, to show that
 \[\frac{s_{11}}{s_{12}} = \frac{\tilde{s}_{11}}{\tilde{s}_{12}},\]
 it is sufficient by (\ref{eq1}) to prove that
  \[\frac{f_{1}}{s_{12}} = \frac{\tilde{f}_{1}}{\tilde{s}_{12}}.\]
\end{remark}
\noindent
From (\ref{ega}), we deduce that
\begin{eqnarray*}
 f - \tilde{f} &=&  (h - \tilde{h}) (ls_{32} - s_{33})(s_{23}-ls_{22})\\
 &=& \tilde{h}(u^4 - 1)(ls_{32} - s_{33})(s_{23}-ls_{22}).
\end{eqnarray*}
Finally, using this last equality, we can rewrite (\ref{eq4}) as
\begin{equation}\label{eq5}
u'' + \frac{1}{2} (\log(\tilde{h}))' u' + (\lambda^2 + 1) \tilde{h}(ls_{32} - s_{33})(s_{23}-ls_{22}) (u^5 - u)= 0.
\end{equation}

\begin{lemma}
 The function $u$ defined by
 \[u = \left( \frac{h}{\tilde{h}} \right)^{\frac{1}{4}},\]
 satisfies
 \[u(0) = 1 \quad \textrm{and} \quad u'(0) = 0.\]
\end{lemma}

\begin{proof}
 The proof is a consequence of the fact that the asymptotically hyperbolic structures given in the third point of Definition \ref{defAH} are the same on the two manifolds.
\end{proof}
\noindent
We thus study the Cauchy problem
\begin{equation}\label{Cauchy}
 \begin{cases}
u'' + \frac{1}{2} (\log(\tilde{h}))' u' + (\lambda^2 + 1) \tilde{h}(ls_{32} - s_{33})(s_{23}-ls_{22}) (u^5 - u)= 0 \\
u(0) = 1 \quad \textrm{and} \quad u'(0) = 0
\end{cases}.
\end{equation}
We immediately note that $u = 1$ is a solution of (\ref{Cauchy}). By uniqueness of the Cauchy problem for the ODE (\ref{Cauchy}) we conclude that $u = 1$. We then have shown that
  \[\frac{f_{1}}{s_{12}} = \frac{\tilde{f}_{1}}{\tilde{s}_{12}}\]
and, using (\ref{eq1}) and (\ref{rk}), we can conclude that
\begin{equation}\label{egalite}
 \frac{s_{11}}{s_{12}} = \frac{\tilde{s}_{11}}{\tilde{s}_{12}} \quad \textrm{and} \quad \frac{s_{11}}{s_{13}} = \frac{\tilde{s}_{11}}{\tilde{s}_{13}}.
\end{equation}

\section{Resolution of the inverse problem}\label{conclusion}

We can now finish the resolution of our inverse problem. We first note that
\[ g = \sum_{i=1}^3 H_i^2 (dx^i)^2 = \psi^{\star}g',\]
where
\[ g' = \frac{H_1^2}{s_{12}} (dX^1)^2 + H_2^2 (dx^2)^2 + H_3^2 (dx^3)^2,\]
where $\psi$ is the diffeomorphism (equal to the identity at the compactified ends $\{x^1 = 0 \}$ and $\{x^1 = A \}$) corresponding to the Liouville change of variables in the first variable
  \[ X^1 =  \int_0^{x^1} \sqrt{s_{12}(s)} \, ds.\]
Similarly,
\[ \tilde{g} = \sum_{i=1}^3 \tilde{H}_i^2 (dx^i)^2 = \tilde{\psi}^{\star}\tilde{g}',\]
where
\[ \tilde{g}' = \frac{\tilde{H}_1^2}{\tilde{s}_{12}} (d\tilde{X}^1)^2 + \tilde{H}_2^2 (dx^2)^2 + \tilde{H}_3^2 (dx^3)^2,\]
where $\tilde{\psi}$ is the diffeomorphism (equal to the identity at the compactified ends $\{x^1 = 0 \}$ and $\{x^1 = A \}$) corresponding to the same Liouville change of variables in the first variable for the second 
manifold. We note that, thanks to the Borg-Marchenko Theorem, we can identify
\[ A^1 = \tilde{A}^1.\]
We now note that, thanks to (\ref{egalite}),
\[ \frac{H_1^2}{s_{12}} = \frac{\det(S)}{s_{12}s^{11}} = \frac{s_{11}}{s_{12}} + \frac{s^{12}}{s^{11}} + \frac{s_{13}}{s_{12}} \frac{s^{13}}{s^{11}} = \frac{\tilde{H}_1^2}{\tilde{s}_{12}},\]
\[ H_2^2 = \frac{\det(S)}{s^{21}} = \frac{\frac{s_{11}}{s_{12}}s^{11}  + s^{12} + \frac{s_{13}}{s_{12}}s^{13}}{ \frac{s_{13}}{s_{12}}s_{32} - s_{33}} = \tilde{H}_2^2\]
and
\[ H_3^2 = \frac{\det(S)}{s^{31}} = \frac{\frac{s_{11}}{s_{12}}s^{11}  + s^{12} + \frac{s_{13}}{s_{12}}s^{13}}{s_{23} - \frac{s_{13}}{s_{12}}s_{22}} = \tilde{H}_3^2.\]
We can then deduce from these equalities that
\[ g' = \tilde{g}'.\]
Finally, we have shown that there exists a diffeomorphism $\Psi := \psi^{-1} \tilde{\psi}$ such that
\[ \tilde{g} = \Psi^{\star} g,\]
where $\Psi$ is the identity at the two ends.

\newpage
\appendix
\section{Proof of Proposition \ref{propcadre}}\label{ap1}

The aim of this Appendix is to prove Proposition \ref{propcadre} which we recall here.

\begin{prop}
 Let $S$ be a St\"ackel matrix with corresponding metric $g_S$. There exists a St\"ackel matrix $\hat{S}$
 with $g_{\hat{S}} = g_S$
 and such that
 \begin{equation}\tag{C}
 \begin{cases}
 \hat{s}_{12}(x^1) > 0 \quad \textrm{and} \quad \hat{s}_{13}(x^1) > 0, \quad \forall x^1\\
 \hat{s}_{22}(x^2) < 0 \quad \textrm{and} \quad \hat{s}_{23}(x^2) > 0, \quad \forall x^2\\
 \hat{s}_{32}(x^3) > 0 \quad \textrm{and} \quad \hat{s}_{33}(x^3) < 0, \quad \forall x^3\\
 \lim\limits_{x^1 \to 0} s_{12}(x^1) = \lim\limits_{x^1 \to 0} s_{13}(x^1) = 1
\end{cases}.
\end{equation}
\end{prop}

\begin{proof}
 The proof of this Proposition consists in three steps and uses the riemannian structure and the invariances of the metric described in Proposition \ref{propinv}.
 We first show that the coefficients of the second and the third columns are non-negative or non-positive functions. Secondly, we show that these coefficients can be
 assumed to be positive or negative functions. Finally, we show that we can find a St\"ackel matrix with the same associated metric and satisfying the condition
 (\ref{cadre}).\\
 
 \noindent
 \underline{Step 1:} We claim that for all $(i,j) \in \{1,2,3\} \times \{2,3\}$, $s_{ij} \geq 0$ or $s_{ij} \leq 0$. Since the proof is similar for the third column we just 
 give the proof for the second one. First, if one of the functions $s_{12}$, $s_{22}$ and $s_{32}$ is identically zero, the two others cannot vanish on their intervals 
 of definition since the minors $s^{11}$, $s^{21}$ and $s^{31}$ cannot vanish. Thus, in this case we immediately obtain that $s_{i2} \geq 0$ or $s_{i2} \leq 0$ for all 
 $i \in \{1,2,3\}$. We can thus assume that there exists a triplet 
 $(x_0^1,x_0^2,x_0^3)$ such that $s_{12}(x_0^1) \neq 0$, $s_{22}(x_0^2) \neq 0$ and $s_{32}(x_0^3) \neq 0$.
 Without loss of generality, we assume that $\det(S) > 0$ and $s^{i1} > 0$ for all $i \in \{1,2,3\}$.
 From the positivity property of the minors we can deduce that, according to the sign of the quantities $s_{12}(x_0^1)$, $s_{22}(x_0^2)$ and $s_{32}(x_0^3)$,
 \begin{itemize}
  \item \underline{$s_{12}(x_0^1) > 0$, $s_{22}(x_0^2) > 0$ and $s_{32}(x_0^3) > 0$:} This case is impossible since the minors $s^{11}$, $s^{21}$ and $s^{31}$ cannot be all 
  positive.
  \item \underline{$s_{12}(x_0^1) > 0$, $s_{22}(x_0^2) < 0$ and $s_{32}(x_0^3) > 0$:}
  \begin{equation}\label{eqriem1}
   \frac{s_{33}(x_0^3)}{s_{32}(x_0^3)} < \frac{s_{23}(x_0^2)}{s_{22}(x_0^2)} < \frac{s_{13}(x_0^1)}{s_{12}(x_0^1)}.
  \end{equation}
  \item \underline{$s_{12}(x_0^1) > 0$, $s_{22}(x_0^2) > 0$ and $s_{32}(x_0^3) < 0$:}
  \begin{equation}\label{eqriem2}
  \frac{s_{13}(x_0^1)}{s_{12}(x_0^1)} <  \frac{s_{33}(x_0^3)}{s_{32}(x_0^3)} < \frac{s_{23}(x_0^2)}{s_{22}(x_0^2)}.
  \end{equation}
  \item \underline{$s_{12}(x_0^1) > 0$, $s_{22}(x_0^2) < 0$ and $s_{32}(x_0^3) > 0$:}
  \begin{equation}\label{eqriem3}
    \frac{s_{23}(x_0^2)}{s_{22}(x_0^2)} < \frac{s_{13}(x_0^1)}{s_{12}(x_0^1)} < \frac{s_{33}(x_0^3)}{s_{32}(x_0^3)}.
  \end{equation}
 \end{itemize}
 Since the four cases corresponding to the case $s_{12}(x_0^1) < 0$ are similar, we just treat the four cases above. Assume, for instance, that there exists $\alpha_0^2$ such that
 $s_{22}(\alpha_0^2) = 0$. We want to show that $s_{22}$ does not change of sign. We denote by $I$ the maximal interval (possibly reduced to $\alpha_0^2$) containing $\alpha_0^2$ such that
 $s_{22}(x^2) = 0$ for all $x^2 \in I$. Since the minors $s^{11}$ and $s^{31}$ are non-vanishing quantities, the functions $s_{12}$ and $s_{32}$ can then not vanish. Thus, there 
 exists two real constants $c_1$ and $c_2$ such that
 \begin{equation}\label{quobor}
   c_1 \leq \frac{s_{13}}{s_{12}} \leq c_2\quad \textrm{and} \quad c_1 \leq \frac{s_{33}}{s_{32}} \leq c_2 ,
 \end{equation}
 i.e. these quotients are bounded. Moreover, $s_{23}(x^2) \neq 0$ for all $x^2 \in I$ and by continuity there exists an interval $J$ such that $I \subsetneq J$ and
 $s_{23}(x^2) \neq 0$ for all $x^2 \in J$. If we assume that $s_{22}$ changes sign in a neighbourhood of $I$ we obtain that for all $\epsilon > 0$ there exist 
 $y_0^2 \in J$ and $y_1^2 \in J$ such that
 \[ 0 < s_{22}(y_0^2)  < \epsilon \quad \textrm{and} \quad -\epsilon < s_{22}(y_1^2)  < 0.\]
 Thus, for all $M > 0$ there exist, $y_0^2 \in J$ and $y_1^2 \in J$ such that
  \[ \frac{s_{23}(y_0^2)}{s_{22}(y_0^2)} > M \quad \textrm{and} \quad \frac{s_{23}(y_1^2)}{s_{22}(y_1^2)} < - M.\]
 We thus obtain a contradiction between (\ref{quobor}) and each of the equalities (\ref{eqriem1}), (\ref{eqriem2}) and (\ref{eqriem3}). We can then conclude that $s_{22}(x^2) \geq 0$ or 
 $s_{22}(x^2) \leq 0$. The proof is similar for $s_{12}$ and $s_{32}$.\\
 
 \noindent
 \underline{Step 2:} We show, thanks to the first invariance given in Proposition \ref{propinv}, that there exists a St\"ackel matrix having the same associated metric
 as $S$ and such that for all $(i,j) \in \{1,2,3\} \times \{2,3\}$, $s_{ij} > 0$ or $s_{ij} < 0$. We recall that there is at most one vanishing function $s_{ij}$, 
 $(i,j) \in \{1,2,3\} \times \{2,3\}$, per column since the minors
 $s^{11}$, $s^{21}$ and $s^{31}$ are non-zero quantities. We assume that one coefficient of the second column vanishes. By symmetry, we can assume that this is $s_{12}$, 
 i.e. that $s_{12}(x_0^1) = 0$ at one point $x_0^1$. We first assume that $s_{23}$ and $s_{33}$ do not vanish. In this case, there exists a real $a \geq 1$ such that
 \[ |s_{23}| < a |s_{22}| \quad \textrm{and} \quad |s_{33}| < a |s_{32}|\]
 and a real constant $b \geq 1$ such that
 \[ |s_{22}| < b |s_{23}| \quad \textrm{and} \quad |s_{32}| < b |s_{33}|.\]
 We now search a $2 \times 2$ constant invertible matrix $G$ such that the coefficients of the new St\"ackel matrix, obtained by the transformation given in the first point
 of Proposition \ref{propinv}, are positive or negative. For instance, if $s_{12}$ and $s_{13}$ have the same sign, we put
 \[ G =  \begin{pmatrix} a & 1  \\ 1 & b \end{pmatrix}\]
 and we thus obtain a new St\"ackel matrix whose second and third columns are
 \[ \begin{pmatrix} as_{12} + s_{13} & s_{12} + bs_{13}  \\ as_{22} + s_{23} & s_{22} + bs_{23}  \\ as_{32} + s_{33} & s_{32} + bs_{33} \end{pmatrix}.\]
 We can easily show that these six components are positive or negative (we recall that $s_{12}$ and $s_{13}$ cannot vanish simultaneously).
 However, if $s_{12}$ and $s_{13}$ have different signs, we put
 \[ G =  \begin{pmatrix} a & -1  \\ -1 & b \end{pmatrix}\]
 and we also obtain positive or negative components. If $s_{23}$ or $s_{33}$ vanish we just have to choose the suitable constants $a$ and $b$ using the fact there is at most 
 one vanishing function in the third column.\\
 
 \noindent
 \underline{Step 3:} Finally, we show, thanks to the first invariance given in Proposition \ref{propinv} and the riemannian structure, that there exists a St\"ackel matrix having the same associated
 metric as $S$ and satisfying the condition (\ref{cadre}) of Definition \ref{defAH}. We recall that thanks, to the second step, we can assume that the St\"ackel matrix $S$ satisfies $s_{ij} > 0$ or $s_{ij} < 0$ 
 for all $(i,j) \in \{1,2,3\} \times \{2,3\}$. We recall that the metric $g$ is riemannian if and only if $\det(S)$, $s^{11}$, $s^{21}$ and $s^{31}$ have the same sign. 
 Without loss of generality, we assume that these quantities are all positive. We recall that according to the sign of the functions $s_{12}$, $s_{22}$ 
 and $s_{32}$ the inequalities (\ref{eqriem1})-(\ref{eqriem3}) are satisfied. We thus have to treat different cases according to the sign of the components of the St\"ackel 
 matrix. We first want to obtain the sign conditions in (\ref{cadre}). Since the proof are similar in the other cases, we just give the proof in the case
\[ s_{12} > 0, \quad s_{22} < 0 \quad \textrm{and} \quad s_{32} > 0.\]
We then give, in each case, the matrix $G \in GL_2(\R)$ such that the transformation given in the first invariance of Proposition \ref{propinv} provides us the signs we want.
        \begin{itemize}
         \item \underline{If $s_{13} > 0$, $s_{23} < 0$ and $s_{33} > 0$:} We put
          \[ G =  \begin{pmatrix} 1 & -1  \\ 0 & b \end{pmatrix},\]
          where
          \[ \frac{s_{13}}{s_{12}} < b < \frac{s_{23}}{s_{22}} < \frac{s_{33}}{s_{32}},\]
          and we obtain the required signs. Indeed, we obtain that the second and the third column of the new St\"ackel matrix are given by
 \[ \begin{pmatrix} s_{12} & -s_{12} + bs_{13}  \\ s_{22} & -s_{22} + bs_{23}  \\ s_{32}  & -s_{32} + bs_{33} \end{pmatrix}\]
 which has the desired signs thanks to our choice of constant $b$.
         \item \underline{If $s_{13} > 0$, $s_{23} > 0$ and $s_{33} < 0$:} We put $G = I_2$.
         \item \underline{If $s_{13} > 0$, $s_{23} < 0$ and $s_{33} < 0$:} We put
          \[ G =  \begin{pmatrix} 1 & -1  \\ 0 & b \end{pmatrix},\]
          where
          \[ \frac{s_{33}}{s_{32}} < \frac{s_{13}}{s_{12}} < b < \frac{s_{23}}{s_{22}}.\]
        \end{itemize}
        As previously, the case $s_{13} < 0$ is similar and we thus omit its proof. Up to this point, we proved that we can assume that  
  \begin{equation}
 \begin{cases}
 s_{12}(x^1) > 0 \quad \textrm{and} \quad s_{13}(x^1) > 0, \quad \forall x^1\\
 s_{22}(x^2) < 0 \quad \textrm{and} \quad s_{23}(x^2) > 0, \quad \forall x^2\\
 s_{32}(x^3) > 0 \quad \textrm{and} \quad s_{33}(x^3) < 0, \quad \forall x^3
\end{cases}.
\end{equation}
Finally, we just have to use once more the invariance with respect to the multiplication of the second and the third column by an invertible constant $2 \times 2$ matrix
$G$ to obtain that we can assume that
 \[\lim\limits_{x^1 \to 0} s_{12}(x^1) = \lim\limits_{x^1 \to 0} s_{13}(x^1) = 1.\]
Indeed, we just have to set
\[ G =  \begin{pmatrix} \frac{1}{\alpha} & 0  \\ 0 & \frac{1}{\beta} \end{pmatrix},\]
where,
\[ \alpha = \lim\limits_{x^1 \to 0} s_{12}(x^1) > 0 \quad \textrm{and} \quad \beta = \lim\limits_{x^1 \to 0} s_{13}(x^1) > 0.\]
The result then follows.
\end{proof}
\newpage
\section{Proof of Lemma \ref{separationBloom}}\label{apsep}

The aim of this Appendix is to prove Lemma \ref{separationBloom} which we recall here.

\begin{lemma}
We set
\[ E_M = \{(|\mu_m|,|\nu_m|), \, \, m \geq M \}\]
and
\[ \mathcal{C} = \{(\mu^2,\theta^2 \mu^2), \quad c_1 + \epsilon \leq \theta^2 \leq c_2 + \epsilon \}, \quad 0 < \epsilon << 1,\]
where
\[ c_1 = \max \left( -\frac{s_{32}}{s_{33}} \right) \quad \textrm{and} \quad c_2 = \min \left( -\frac{s_{22}}{s_{23}} \right).\]
In that case, there exists $h > 0$ such that $|e_1 - e_2| \geq h$ for all $(e_1,e_2) \in (E_M \cap \mathcal{C})^2$, $e_1 \neq e_2$.
\end{lemma}

\begin{proof}
 We recall that the coupled spectrum was defined in Remark \ref{rknum} by
 \begin{equation}\label{equation1ap2}
   HY_m = \mu_m^2 Y_m \quad \textrm{and} \quad LY_m = \nu_m^2 Y_m, \quad \forall m \geq 1,
 \end{equation}
  where $H$ and $L$ are commuting, elliptic and selfadjoint operators of order two. Writing $Y_m(x^2,x^3) = v_m(x^2) w_m(x^3)$, we obtain that (\ref{equation1ap2}) is 
  equivalent to
  \begin{eqnarray}\label{equation2ap2}
-v_m''(x^2) + \left[-(\lambda^2 +1)s_{21}(x^2) + \mu_m^2 s_{22}(x^2) + \nu_m^2 s_{23}(x^2)\right]v_m(x^2) = 0,
\end{eqnarray}
and
\begin{eqnarray} \label{equation3ap2}
-w_m''(x^3) + \left[-(\lambda^2 +1)s_{31}(x^3) + \mu_m^2 s_{32}(x^3) + \nu_m^2 s_{33}(x^3)\right]w_m(x^3) = 0,
  \end{eqnarray}
where $v_m$ and $w_m$ are periodic functions, i.e.
  \begin{eqnarray}\label{equation4ap2}
  \begin{cases}
   v_m(0) = v_m(B) \quad \textrm{and} \quad v_m'(0) = v_m'(B) \\
   w_m(0) = w_m(C) \quad \textrm{and} \quad w_m'(0) = w_m'(C) \\
  \end{cases}.
  \end{eqnarray}
We first consider Equation (\ref{equation2ap2}) which we rewrite as
\[ -v'' - (\lambda^2 + 1) s_{21} v = \mu^2 \left[ -s_{22} - \theta^2 s_{23} \right]v,\]
where $v := v_m$, $\mu^2 := \mu_m^2$, $\nu^2 := \nu_m^2$ and 
\[\theta^2 := \frac{\nu^2}{\mu^2}.\]
In the following we will consider Schr\"odinger equations associated with (\ref{equation2ap2})-(\ref{equation3ap2})
whose spectral parameter is $\mu^2$ which tends to $+ \infty$. Moreover, these equations depend on the parameter $\theta^2$ which is always bounded 
in a suitable cone we introduce now. We recall that, as we have shown in Lemma \ref{distrivp}, there exist real constants $C_1$, $C_2$, $D_1$ and $D_2$ such that for
all $m \geq 1$,
\[ C_1 \mu_m^2 + D_1 \leq \nu_{m}^2 \leq C_2 \mu_m^2 + D_2,\]
where
\[ C_1 = \min \left( -\frac{s_{32}}{s_{33}} \right) > 0 \quad \textrm{and} \quad C_2 = -\min \left(\frac{s_{22}}{s_{23}} \right) > 0.\]
 Let $\epsilon > 0$ be fixed, we then consider $\theta^2$ such that
  \begin{eqnarray}\label{equation4bisap2}
c_1 + \frac{D_1}{\mu^2} + \epsilon \leq \theta^2 \leq c_2 + \frac{D_2}{\mu^2} - \epsilon,
\end{eqnarray}
where
\[ c_1 = \max \left( -\frac{s_{32}}{s_{33}} \right) \quad \textrm{and} \quad c_2 = \min \left( -\frac{s_{22}}{s_{23}} \right).\]
We note that
\[ 0 < C_1 \leq c_1 < c_2 \leq C_2.\]
This implies that, for sufficiently large $\mu^2$, there exists $\delta > 0$ such that
  \begin{eqnarray}\label{equation5ap2}
 -s_{22} - \theta^2 s_{23} \geq  \left(\epsilon - \frac{D_2}{\mu^2} \right) s_{23} \geq \delta > 0
  \end{eqnarray}
and
\begin{eqnarray} \label{equation6ap2}
 -s_{32} - \theta^2 s_{33} \geq \left( \epsilon + \frac{D_1}{\mu^2} \right) (-s_{33}) \geq \delta > 0.
  \end{eqnarray}
For such a $\theta^2$, we can thus proceed to the Liouville change of variables
\[ X^2 = \int_0^{x^2} \sqrt{-s_{22}(t) - \theta^2 s_{23}(t)} \, dt,\]
in Equation (\ref{equation2ap2}). This new variable thus satisfies $X^2 \in [0,\tilde{B}(\theta^2)]$, where
\begin{equation}\label{equation6bisap2}
 \tilde{B}(\theta^2) = \int_0^{B} \sqrt{-s_{22}(t) - \theta^2 s_{23}(t)} \, dt. 
\end{equation}
Finally, we set
\[ V(X^2) = \left[-s_{22}(x^2(X^2)) - \theta^2 s_{23}(x^2(X^2)) \right]^{\frac{1}{4}} v(x^2(X^2)).\]
This new function then satisfies in the variable $X^2$ the Schr\"odinger equation
  \begin{eqnarray}\label{equation7ap2}
-\ddot{V}^2(X^2) + Q_{\theta^2}(X^2) V(X^2) = \mu^2 V(X^2),
\end{eqnarray}
where $\mu^2$ is the spectral parameter, $Q_{\theta^2}(X^2)$ is uniformly bounded with respect to $\theta^2$ satisfying (\ref{equation4bisap2}) and for such a 
$\theta^2$,
\[ Q_{\theta^2}(X^2) = O(1).\]

\noindent
We now search the couples $(\mu^2,\theta^2)$ such that (\ref{equation7ap2}) admits periodic solutions. We define $\{C_0,S_0\}$ and $\{C_1,S_1\}$ the usual Fondamental Systems of 
Solutions of (\ref{equation7ap2}), i.e.
\[ C_0(0) = 1, \quad \dot{C}_0(0) = 0, \quad S_0(0) = 0 \quad \textrm{and} \quad \dot{S}_0(0) = 1,\]
and
\[ C_1(\tilde{B}) = 1, \quad \dot{C}_1(\tilde{B}) = 0, \quad S_1(\tilde{B}) = 0 \quad \textrm{and} \quad \dot{S}_1(\tilde{B}) = 1.\]
We recall that these functions are analytic and even with respect to $\mu$. We write the solutions $V$ of (\ref{equation7ap2}) as
\[ V = \alpha C_0 + \beta S_0 = \gamma C_1 + \delta S_1,\]
where $\alpha$, $\beta$, $\gamma$ and $\delta$ are real constants. Thus,
\[ V(0) = \alpha, \quad \dot{V}(0) = 0, \quad V(\tilde{B}) = \gamma \quad \textrm{and} \quad \dot{V}(\tilde{B}) = \delta.\]
$V$ is then a periodic function if and only if
\[ V(0) = V(\tilde{B}) \quad \Leftrightarrow \quad \alpha = \gamma \quad \Leftrightarrow \quad W(V,S_0) = W(V,S_1)\]
and
\[ \dot{V}(0) = \dot{V}(\tilde{B}) \quad \Leftrightarrow \quad \beta = \delta \quad \Leftrightarrow \quad W(C_0,V) = W(C_1,V),\]
where $W(f,g) = fg'-f'g$ denotes the Wronskian of two functions $f$ and $g$. In other words, $V$ is a periodic solution of (\ref{equation7ap2}) if and only if
  \begin{eqnarray}\label{equation8ap2}
 W(V,S_0-S_1) = W(C_0-C_1,V) = 0.
\end{eqnarray}
We thus add to the Equation (\ref{equation7ap2}) the boundary conditions (\ref{equation8ap2}) and we define the corresponding characteristic functions. In other words, 
we defined
\[ \Delta_1(\mu^2,\theta^2) = W(C_0-C_1,S_0-S_1) = 2 - W(C_0,S_1) - W(C_1,S_0).\]
We emphasize that $\Delta_1(\mu^2,\theta^2)$ vanishes if and only if there exists a periodic solution of (\ref{equation7ap2}) for $(\mu^2,\theta^2)$. The asymptotics 
of $W(C_0,S_1)$ and $W(C_1,S_0)$ are well known (see for instance \cite{DNK2,FY2}). Indeed, we know that
  \begin{eqnarray}\label{equation9ap2}
 W(C_0,S_1) = \cos \left( \mu \tilde{B}(\theta^2) \right) \times \left( 1 + O \left( \frac{1}{\mu} \right) \right)
\end{eqnarray}
and
  \begin{eqnarray}\label{equation10ap2}
 W(C_1,S_0) = \cos \left( \mu \tilde{B}(\theta^2) \right) \times \left( 1 + O \left( \frac{1}{\mu} \right) \right),
\end{eqnarray}
where $\mu = \sqrt{\mu^2}$ (we do not have to precise the sign of $\mu$ since the characteristic functions are even functions). We then obtain that
  \begin{eqnarray}\label{equation11ap2}
  \Delta_1(\mu^2,\theta^2) = 0 \quad \Leftrightarrow \quad  2 - 2 \cos \left( \mu \tilde{B}(\theta^2) \right) + O \left( \frac{1}{\mu} \right) = 0.
\end{eqnarray}
Using the Rouch\'e's Theorem (see for instance \cite{FY2}) we can then deduce that the couples $(\mu^2,\theta^2)$ satisfying (\ref{equation11ap2})
are close for large $\mu$ to the couples $(\mu^2,\theta^2)$ satisfying
  \[ 2 - 2 \cos \left( \mu \tilde{B}(\theta^2) \right) = 0 \quad \Leftrightarrow \quad \cos \left( \mu \tilde{B}(\theta^2) \right)  = 1.\]
The solutions of this last equation are
\[ \mu = \frac{2m\pi}{\tilde{B}(\theta^2)}, \quad m \in \Z,\]
for $\theta^2$ satisfying (\ref{equation4bisap2}) and $m$ sufficiently large. Finally, we recall that $\tilde{B}(\theta^2)$ is given by (\ref{equation6bisap2}).
Thus, since $s_{23}$ is a positive function, the map $\tilde{B}$ is strictly decreasing with respect to $\theta^2 \in [c_1 + \epsilon,c_2 - \epsilon]$. The map $\frac{1}{\tilde{B}(\theta^2)}$
is then strictly increasing. We can summarize these facts on the following picture:
\begin{figure}[htbp]\label{Figure1ap2}
    \center
   \includegraphics[scale=0.25]{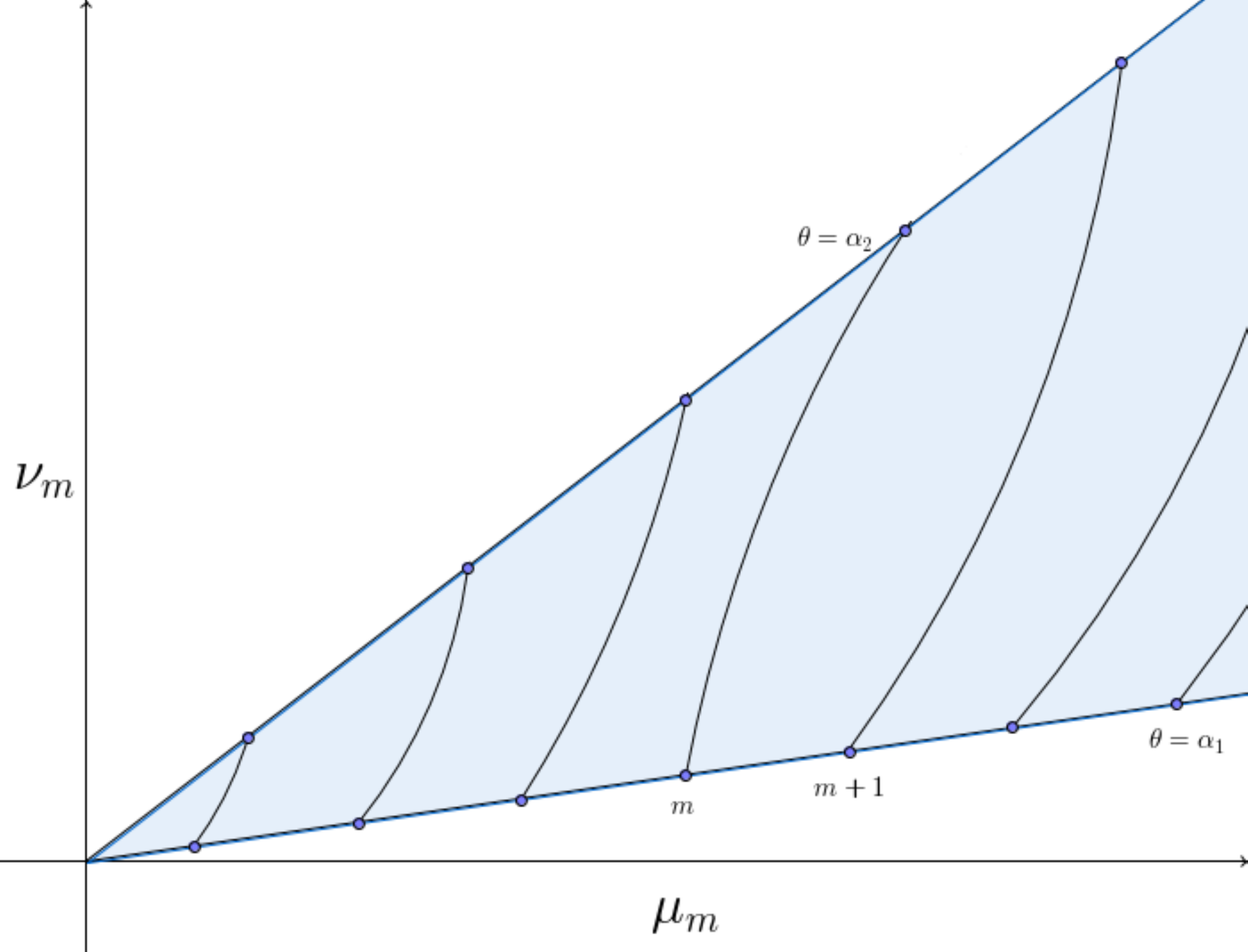}
    \caption{First approximation of the coupled spectrum}
\end{figure} 

\noindent
We do the same analysis on the Equation (\ref{equation3ap2}). We recall that if $\theta^2$ satisfies (\ref{equation4bisap2}) then the inequality (\ref{equation6ap2}) is 
satisfied for $\mu^2$ sufficiently large. We can thus set
\[ X^3 = \int_0^{x^3} \sqrt{-s_{32}(t) - \theta^2 s_{33}(t)} \, dt.\]
This new variable satisfies $X^3 \in [0,\tilde{C}(\theta^2)]$, where
\begin{equation}\label{equation11bisap}
 \tilde{C}(\theta^2) = \int_0^{C} \sqrt{-s_{32}(t) - \theta^2 s_{33}(t)} \, dt. 
\end{equation}
We then set
\[ W(X^3) = \left[ -s_{32}(x^3(X^3)) - \theta^2 s_{33}(x^3(X^3)) \right]^{\frac{1}{4}} w(x^3(X^3)).\]
This function then satisfies, in the variable $X^3$,  the Schr\"odinger equation
  \begin{eqnarray}\label{equation12ap2}
-\ddot{W}^2(X^3) + \tilde{Q}_{\theta^2}(X^3) W(X^3) = \mu^2 W(X^3), \quad \textrm{where} \quad \tilde{Q}_{\theta^2}(X^3) = O(1),
\end{eqnarray}
for $\theta^2$ satisfying (\ref{equation4bisap2}) and $\mu^2$ sufficiently large. As previously, we obtain that (\ref{equation12ap2}) has a periodic solution if and only
if
\[ \Delta_2(\mu^2,\theta^2) := 2 - W(C_0,S_1) - W(C_1,S_0) = 0.\]
Thanks to the asymptotics (\ref{equation9ap2})-(\ref{equation10ap2}) we obtain that
\[ \Delta_2(\mu^2,\theta^2) = 0 \quad \Leftrightarrow \quad 2 - 2 \cos\left( \mu \tilde{C}(\theta^2) \right) + O\left( \frac{1}{\mu} \right) = 0.\]
 Using once more the Rouch\'e's Theorem, we obtain that the couples $(\mu^2,\theta^2)$ satisfying the previous equality are close for large $\mu$ to the couples 
 satisfying
 \[ \cos\left( \mu \tilde{C}(\theta^2) \right) = 1, \quad \textrm{i.e.} \quad \mu^2 = \frac{2 \pi k}{\tilde{C}(\theta^2)}, \quad k \in \Z,\]
 where $k$ is sufficiently large and $\theta^2$ satisfies Equation (\ref{equation4bisap2}). We recall that $\tilde{C}(\theta^2)$ is given by 
 (\ref{equation11bisap}).
Since $s_{33}$ is a negative function, the map $\tilde{C}$ is strictly increasing for $\theta^2 \in [c_1 + \epsilon,c_2 - \epsilon]$.
The map $\frac{1}{\tilde{C}(\theta^2)}$ is then strictly decreasing. We can summarize these facts on the following picture:
\begin{figure}[htbp]
    \center
   \includegraphics[scale=0.25]{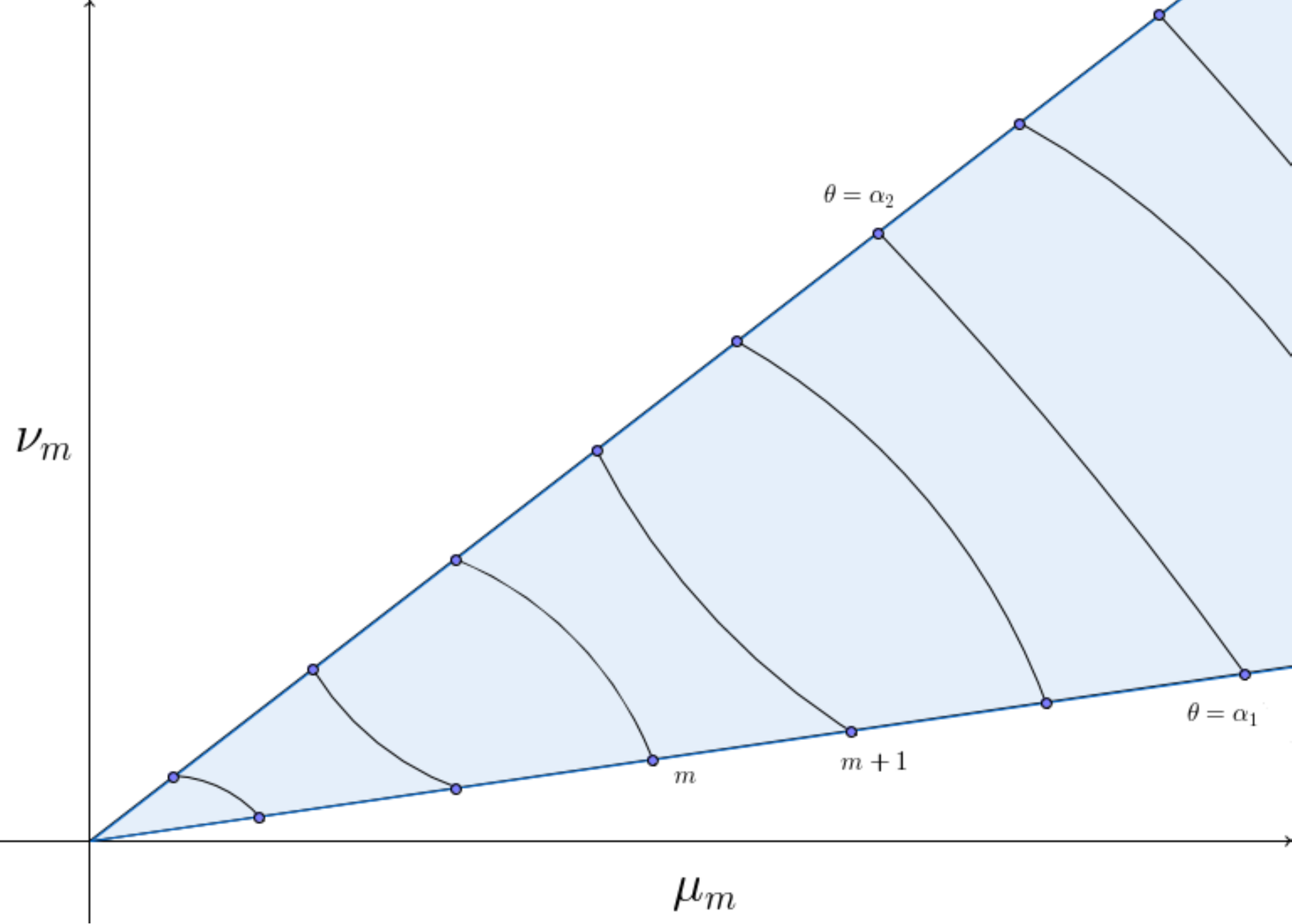}
    \caption{Second approximation of the coupled spectrum}
\end{figure}

\noindent
The coupled spectrum $\Lambda = \{(\mu_m^2,\nu_m^2),$ $m \geq 1\}$, or equivalently the coupled spectrum $(\mu_m^2,\theta_m^2)$, is then given by
\[ \Lambda = \{ \Delta_1(\mu^2,\theta^2) = 0 \} \cap \{ \Delta_2(\mu^2,\theta^2) = 0 \},\]
since for all $(\mu_m^2,\nu_m^2) \in \Lambda$, there exists simultaneously a periodic solution of (\ref{equation7ap2}) and a periodic solution of (\ref{equation12ap2}).
Using the previous two figures we obtain the following one on which the coupled spectrum corresponds to the intersection between the previous curves:
  \begin{figure}[htbp]
    \center
   \includegraphics[scale=0.25]{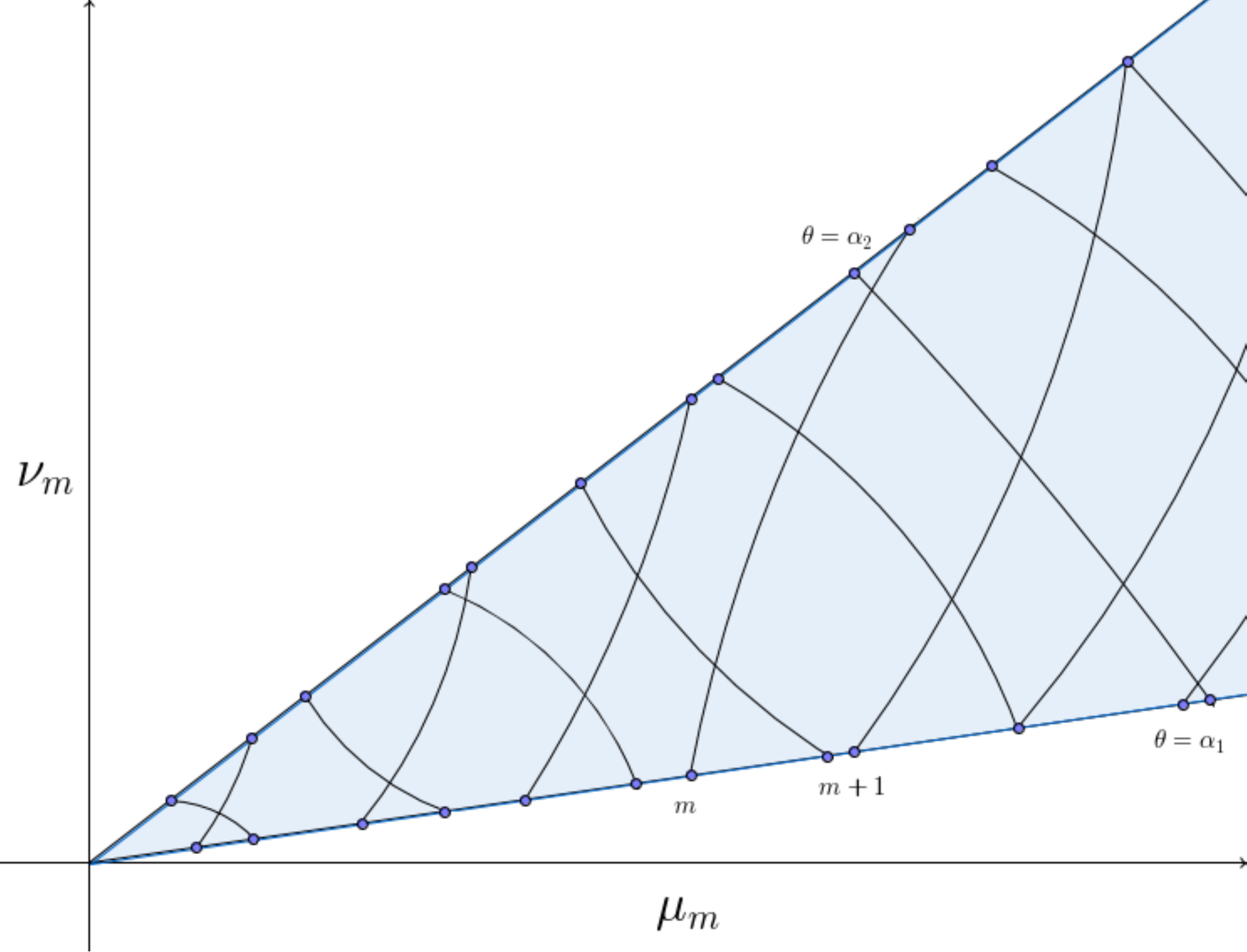}
    \caption{The coupled spectrum}
\end{figure}

\noindent
We now want to use this particular structure of the coupled spectrum to prove Lemma \ref{separationBloom}. We work on the plane $(\mu,\theta)$ and we set $\nu = \theta \mu$, 
with $0 < \alpha_1 \leq \theta \leq \alpha_2$, where
\[ \alpha_1 = \sqrt{c_1^2 + \epsilon} \quad \textrm{and} \quad \alpha_2 = \sqrt{c_2^2 - \epsilon},\]
with $\epsilon > 0$. We recall that for large $m$ we can approximate $\mu_m$ by
\[ \mu_m = \frac{2m \pi}{\tilde{B}(\theta^2)},\]
where
\[ \tilde{B}(\theta^2) = \int_0^{B} \sqrt{-s_{22}(t) - \theta^2 s_{23}(t)} \, dt.\]
We first want to show that the curves drawn in the first Figure are uniformly separated. In other words, we show that there exists $\delta > 0$ such that the distance 
between two successive curves is greater than $\delta$. Precisely, we want to show that there exists $\delta > 0$ such that for large $m$ and for all
$(\theta_1,\theta_2) \in [\alpha_1,\alpha_2]^2$,
\begin{equation}\label{equation13ap2}
 |\mu_{m+1}(\theta_2) - \mu_m(\theta_1)| + |\theta_2 \mu_{m+1}(\theta_2) - \theta_1 \mu_m(\theta_1)| \geq \delta.
\end{equation}
If we note,
\[ d = |\mu_{m+1}(\theta_2) - \mu_m(\theta_1)|,\]
we immediately obtain that (\ref{equation13ap2}) is equivalent to
\begin{equation}\label{equation14ap2}
 d + |d \theta_2 + (\theta_2 - \theta_1) \mu_m(\theta_1)| \geq \delta.
\end{equation}
We now use the mean value Theorem on the map $\frac{1}{\tilde{B}(\theta^2)}$ and we thus obtain
\[ \frac{1}{\tilde{B}(\theta_2^2)} = \frac{1}{\tilde{B}(\theta_1^2)} + e(\xi)(\theta_2^2 - \theta_1^2),\]
where
\[ e(\xi) = - \frac{\tilde{B}'(\xi^2)}{\tilde{B}(\xi^2)^2} > 0,\]
with $\xi \in (\theta_1,\theta_2)$. Actually, we can show that there exist two positive constants $e_1$ and $e_2$ such that
\[ 0 < e_1 \leq e(\xi) \leq e_2, \quad \forall \xi \in [\alpha_1,\alpha_2].\]
We then easily obtain that
\[ d = \left| \frac{2\pi}{\tilde{B}(\theta_1^2)} + 2(m+1)\pi e(\xi)(\theta_1+\theta_2)(\theta_1 - \theta_2) \right|.\]
Using the triangle inequality we thus obtain that
\begin{equation}\label{equation15ap2}
 2(m+1)\pi e(\xi)(\theta_1+\theta_2)|\theta_1 - \theta_2| \geq \frac{2\pi}{\tilde{B}(\theta_1^2)} - d.
\end{equation}
We thus have to study different cases.\\
 
\noindent
\underline{Case 1:} If
 \[d \geq \frac{2\pi}{\tilde{B}(\theta_1^2)},\]
we easily obtain
\[  |\mu_{m+1}(\theta_2) - \mu_m(\theta_1)| + |\theta_2 \mu_{m+1}(\theta_2) - \theta_1 \mu_m(\theta_1)| \geq d \geq \frac{2\pi}{\tilde{B}(\theta_1^2)}.\]
 
\noindent
\underline{Case 2:} If
 \[d < \frac{2\pi}{\tilde{B}(\theta_1^2)},\]
then (\ref{equation15ap2}) gives us
\[ |\theta_1 - \theta_2| > \frac{2\pi -d\tilde{B}(\theta_1^2)}{2(m+1)\pi e(\xi)(\theta_1+\theta_2)\tilde{B}(\theta_1^2)}.\]
Thus,
\begin{eqnarray*}
 \mu_m(\theta_1) |\theta_1 - \theta_2| &=& \frac{2m\pi}{\tilde{B}(\theta_1^2)} |\theta_1 - \theta_2| \\
 &>& \frac{m}{m+1} \frac{2\pi -d\tilde{B}(\theta_1^2)}{e(\xi)(\theta_1+\theta_2)\tilde{B}(\theta_1^2)^2}\\
 &>&  \frac{2\pi -d\tilde{B}(\alpha_1^2)}{4 e_2 \alpha_2 \tilde{B}(\alpha_1^2)^2}.
\end{eqnarray*}
We note that
\[ d\theta_2 <  \frac{2\pi -d\tilde{B}(\alpha_1^2)}{4 e_2 \alpha_2 \tilde{B}(\alpha_1^2)^2} \quad \Leftrightarrow \quad d <  \frac{2\pi}{(4 \theta_2 e_2 \alpha_2\tilde{B}(\alpha_1^2) +1)\tilde{B}(\alpha_1^2)}.\]
If
\[ d > \frac{2\pi}{(4\theta_2  e_2 \alpha_2\tilde{B}(\alpha_1^2) +1)\tilde{B}(\alpha_1^2)},\]
then as in the Case $1$, we easily obtain
\[  |\mu_{m+1}(\theta_2) - \mu_m(\theta_1)| + |\theta_2 \mu_{m+1}(\theta_2) - \theta_1 \mu_m(\theta_1)| \geq d \geq \delta.\]
If
\[ d < \frac{2\pi}{(4\theta_2  e_2 \alpha_2\tilde{B}(\alpha_1^2) +1)\tilde{B}(\alpha_1^2)},\]
we then obtain
\begin{eqnarray*}
 && |\mu_{m+1}(\theta_2) - \mu_m(\theta_1)| + |\theta_2 \mu_{m+1}(\theta_2) - \theta_1 \mu_m(\theta_1)| \\
 &=&  d + |d \theta_2 + (\theta_2 - \theta_1) \mu_m(\theta_1)| \\
 &=& d + |\theta_2 - \theta_1| \mu_m(\theta_1) - d \theta_2 \\
 &>& d + \frac{2\pi -d\tilde{B}(\alpha_1^2)}{4 e_2 \alpha_2 \tilde{B}(\alpha_1^2)^2} - d \theta_2 \\
 &=& \frac{\pi}{2 e_2 \alpha_2 \tilde{B}(\alpha_1^2)^2} + d \left( 1 - \frac{1}{4 e_2 \alpha_2 \tilde{B}(\alpha_1^2)} - \theta_2 \right).
\end{eqnarray*}
We note that there exists $d_0 > 0$ such that for all $d < d_0$,
\[  d \left( 1 - \frac{1}{4 e_2 \alpha_2 \tilde{B}(\alpha_1^2)} - \theta_2 \right) > - \frac{\pi}{4 e_2 \alpha_2 \tilde{B}(\alpha_1^2)^2}.\]
Thus, for all $d < d_0$, we immediately obtain
\[ |\mu_{m+1}(\theta_2) - \mu_m(\theta_1)| + |\theta_2 \mu_{m+1}(\theta_2) - \theta_1 \mu_m(\theta_1)| \geq \frac{\pi}{4 e_2 \alpha_2 \tilde{B}(\alpha_1^2)^2} \geq \delta.\]
Moreover, if $d \geq d_0$ we conclude as in the Case 1.\\

\noindent
We thus have shown that the curves of the first Figure are uniformly separated. Since, the same analysis is also true for the second Figure we have 
shown Lemma \ref{separationBloom}. 
 
\end{proof}

\newpage
\section{Proof of Lemma \ref{number}}\label{ap}

The aim of this Appendix is to prove Lemma \ref{number} which we recall here.

\begin{lemma}
 We set
 \[n(r) = \# E_M \cap B(0,r) \cap \mathcal{C}, \]
 where
  \[ E_M = \{(|\mu_m|,|\nu_m|), \, \, m \geq 1 \},\]
  without multiplicity. Then,
  \[ \varlimsup \frac{n(r)}{r^2} > 0, \quad r \to + \infty.\]
\end{lemma}

\begin{proof}
 To prove the Lemma we use the work of Colin de Verdi\`ere on the coupled spectrum of commuting pseudodifferential operators in \cite{CdV1,CdV2}.
 We recall that the operators $L$ and $H$ are defined by (\ref{defLetH}) and satisfy (\ref{LHvp}). Since we proved in Lemma \ref{propHL} that $H$ and $L$ are semibounded 
 operators, we can say that there exists $M \in \R$ such that $L + M$ and $H + M$ are positive operators. We set
 \[ P_1 = \sqrt{L + M} \quad \textrm{and} \quad P_2 = \sqrt{H+M}.\]
 The operators $P_1$ and $P_2$ are commuting, selfadjoint pseudodifferential operators of order $1$ such that $P_1^2 + P_2^2$ is an elliptic operator. These operators are thus in the framework 
 of \cite{CdV1}. The principal symbol of $P_1$ and $P_2$ are given by
 \begin{equation}\label{symbprinc}
 p_1(x,\xi) = \sqrt{ -\frac{s_{33}}{s^{11}} \xi_2^2 + \frac{s_{23}}{s^{11}}\xi_3^2} \quad \textrm{and} \quad  p_2(x,\xi) = \sqrt{ \frac{s_{32}}{s^{11}} \xi_2^2 - \frac{s_{22}}{s^{11}}\xi_3^2},
\end{equation}
 respectively. We put
 \[ p(x,\xi) = (p_1(x,\xi),p_2(x,\xi)),\]
 where $x :=(x^2,x^3)$, $\xi := (\xi_2,\xi_3)$ and $(x,\xi)$ is a point on the cotangent bundle of $\mathcal{T}^2$, i.e. $T^{\star} \mathcal{T}^2$. We will apply Theorem 0.7 of \cite{CdV1} to $P_1$ and $P_2$. We recall here this result adapted to our framework.
 \begin{theorem}\label{cdv}
  Let $C$ be a cone of $\dot{\R}^2 = \R^2 \setminus \{(0,0) \}$, with piecewise $C^1$ boundary such that $\partial C \cap W = \emptyset$, where $\partial C$ is the boundary
  of $C$ and $W$ is the set of critical values of $p$. We then have
   \[ \# \{ \lambda \in C \cap \Lambda, \, \, |\lambda| \leq r \} = \frac{1}{4\pi^2} \mathrm{vol}_{\Omega} \left( p^{-1}(C \cap B(0,r)) \right) + O(r),\]
 where $\Lambda$ is the coupled spectrum of $P_1$ and $P_2$ and $\Omega = dx^2 \wedge dx^3 \wedge d\xi_2 \wedge d\xi_3$.
 \end{theorem}
 \noindent
 Thus, to use Theorem \ref{cdv}, we have to determine the set $W$ of critical values of $p$. We first have to determine the critical points of $p$ i.e. the points for 
 which the differential of $p$ is not onto. The differential of $p$ is given by (we omit the variables)
 \[ Dp(x,\xi) = \frac{-1}{4p_1 p_2} \begin{pmatrix}  \partial_2 \left(\frac{s_{33}}{s^{11}} \right) \xi_2^2 - \partial_2 \left(\frac{s_{23}}{s^{11}} \right) \xi_3^2 & \partial_3 \left(\frac{s_{33}}{s^{11}} \right) \xi_2^2 - \partial_3 \left(\frac{s_{23}}{s^{11}} \right) \xi_3^2 & 2\frac{s_{33}}{s^{11}} \xi_2 & -2\frac{s_{23}}{s^{11}} \xi_3 \\
   -\partial_2 \left(\frac{s_{32}}{s^{11}} \right) \xi_2^2 + \partial_2 \left(\frac{s_{22}}{s^{11}} \right) \xi_3^2 & - \partial_3 \left(\frac{s_{32}}{s^{11}} \right) \xi_2^2 + \partial_3 \left(\frac{s_{22}}{s^{11}} \right) \xi_3^2 & -2\frac{s_{32}}{s^{11}} \xi_2 & 2\frac{s_{22}}{s^{11}} \xi_3 \\ \end{pmatrix}.\]
 We compute the six $2 \times 2$ minors of this matrix and we search the points $(x,\xi)$ for which all these minors vanish. After calculation, we obtain that $(x,\xi)$ is a critical
 point of $p$ if and only if the four following conditions are satisfied:
 \[ \begin{cases}
\xi_2 \xi_3 = 0\\
\xi_3 \partial_2(s_{22})(\xi_2^2 + \xi_3^2) = 0\\
\xi_2 \partial_3(s_{33})(\xi_2^2 + \xi_3^2) = 0\\
\partial_2(s_{22}) \partial_3(s_{33})(\xi_2^2 + \xi_3^2)^2 = 0
\end{cases}.\]
Thus, there are four cases to study according to the vanishing of $\partial_2(s_{22})$ and $\partial_3(s_{33})$. We finally obtain that
\[ W = \begin{cases}
(0,0) \quad \textrm{if} \quad \partial_2(s_{22}) \neq 0 \quad \textrm{and} \quad \partial_3(s_{33}) \neq 0\\
\mathcal{D}_1 \quad \textrm{if} \quad \partial_2(s_{22}) = 0 \quad \textrm{and} \quad \partial_3(s_{33}) \neq 0\\
\mathcal{D}_2 \quad \textrm{if} \quad \partial_2(s_{22}) \neq 0 \quad \textrm{and} \quad \partial_3(s_{33}) = 0\\
\mathcal{D}_1 \cup \mathcal{D}_2 \quad \textrm{if} \quad \partial_2(s_{22}) = 0 \quad \textrm{and} \quad \partial_3(s_{33}) = 0\\
\end{cases},\]
where
\[ \mathcal{D}_1 = \{t(\sqrt{s_{23}},\sqrt{-s_{22}}), \, \, t \geq 0 \} \quad \textrm{and} \quad \mathcal{D}_2 = \{t(\sqrt{-s_{33}},\sqrt{s_{32}}), \, \, t \geq 0 \},\]
where $s_{22}$, $s_{23} = s_{22} +1$, $s_{33}$ and $s_{32} = s_{33} +1$ are constants according to the case we study.
 We now recall that in Theorem \ref{cdv}, we have to choose a cone $C$ such that $\partial C \cap W = \emptyset$ and we want to study the set
 \[p^{-1}(C \cap B(0,r)) = p^{-1}(C) \cap p^{-1}(B(0,r)).\]
 Let $r > 0$, we first study the set $p^{-1}(B(0,r))$. We recall that there exists a constant $c_1 > 0$ such that
 \[ \max \left( -\frac{s_{33}}{s^{11}}, \frac{s_{23}}{s^{11}},\frac{s_{32}}{s^{11}},-\frac{s_{22}}{s^{11}} \right) \leq c_1.\]
 Thus, if $(\xi_2,\xi_3) \in B \left( 0, \frac{r}{\sqrt{2c_1}} \right)$ and $(x^2,x^3) \in \mathcal{T}^2$, then
 \[ \Vert p(x,\xi) \Vert = \sqrt{p_1(x,\xi) + p_2(x,\xi)} \leq \sqrt{2c_1(\xi_2^2 + \xi_3^2)} \leq r.\]
 We deduce from this fact that
 \begin{equation}\label{imreciboule}
   \mathcal{T}^2 \times B\left( 0, \frac{r}{\sqrt{2c_1}} \right) \subset p^{-1}(B(0,r)).
 \end{equation}
 We now study the set $p^{-1}(C)$. We have to divide our study in four cases as we have seen before.\\
 
 \noindent
 \underline{Case 1:} $\partial_2(s_{22}) \neq 0$ and $\partial_3(s_{33}) \neq 0$. In this case we just have to avoid the point $\{(0,0)\}$. We consider the cone
 \[ C = \{(x,y) \in \R^2 \, \, \mathrm{such} \,\, \mathrm{that} \,\, \epsilon \leq x, \, \, \epsilon \leq y \}, \quad \epsilon > 0.\]
   \begin{figure}[h]
    \center
   \includegraphics[scale=0.25]{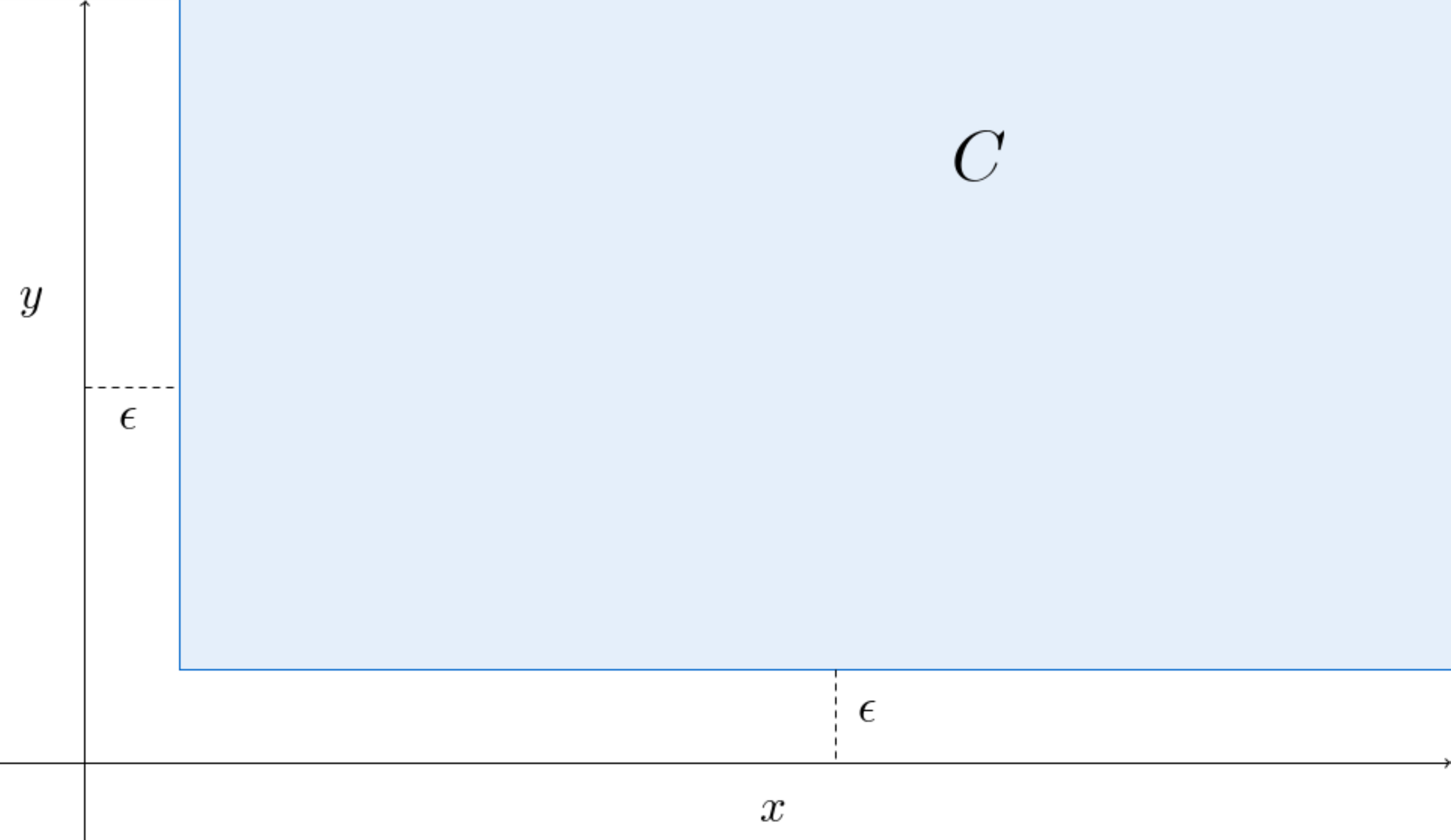}
    \caption{Case 1}
\end{figure}\\
 By definition
 \[ p^{-1}(C) = \{(x,\xi) \in \mathcal{T}^2 \times \R^2, \, \, \epsilon \leq p_1(x,\xi), \, \, \epsilon \leq p_2(x,\xi) \}\]
 and since there exists $c_2 > 0$ such that
 \[ c_2 \leq \min \left( -\frac{s_{33}}{s^{11}}, \frac{s_{23}}{s^{11}},\frac{s_{32}}{s^{11}},-\frac{s_{22}}{s^{11}} \right),\]
there exists $\eta > 0$ such that
\[ \mathcal{T}^2 \times (\R^2 \setminus B(0,\eta)) \subset p^{-1}(C).\]
  \underline{Case 2:} $\partial_2(s_{22}) = 0$ and $\partial_3(s_{33}) \neq 0$. We have to avoid the half-line $\mathcal{D}_1$ which has slope
  $\beta_1 = \sqrt{\frac{-s_{22}}{s_{23}}}$. We consider the cone
   \[ C = \{(x,y) \in \R^2 \, \, \mathrm{such} \,\, \mathrm{that} \,\, \epsilon \leq x, \, \, \epsilon \leq y \leq \beta_1 x - \epsilon\}, \quad \epsilon > 0.\]
      \begin{figure}[h]
    \center
   \includegraphics[scale=0.25]{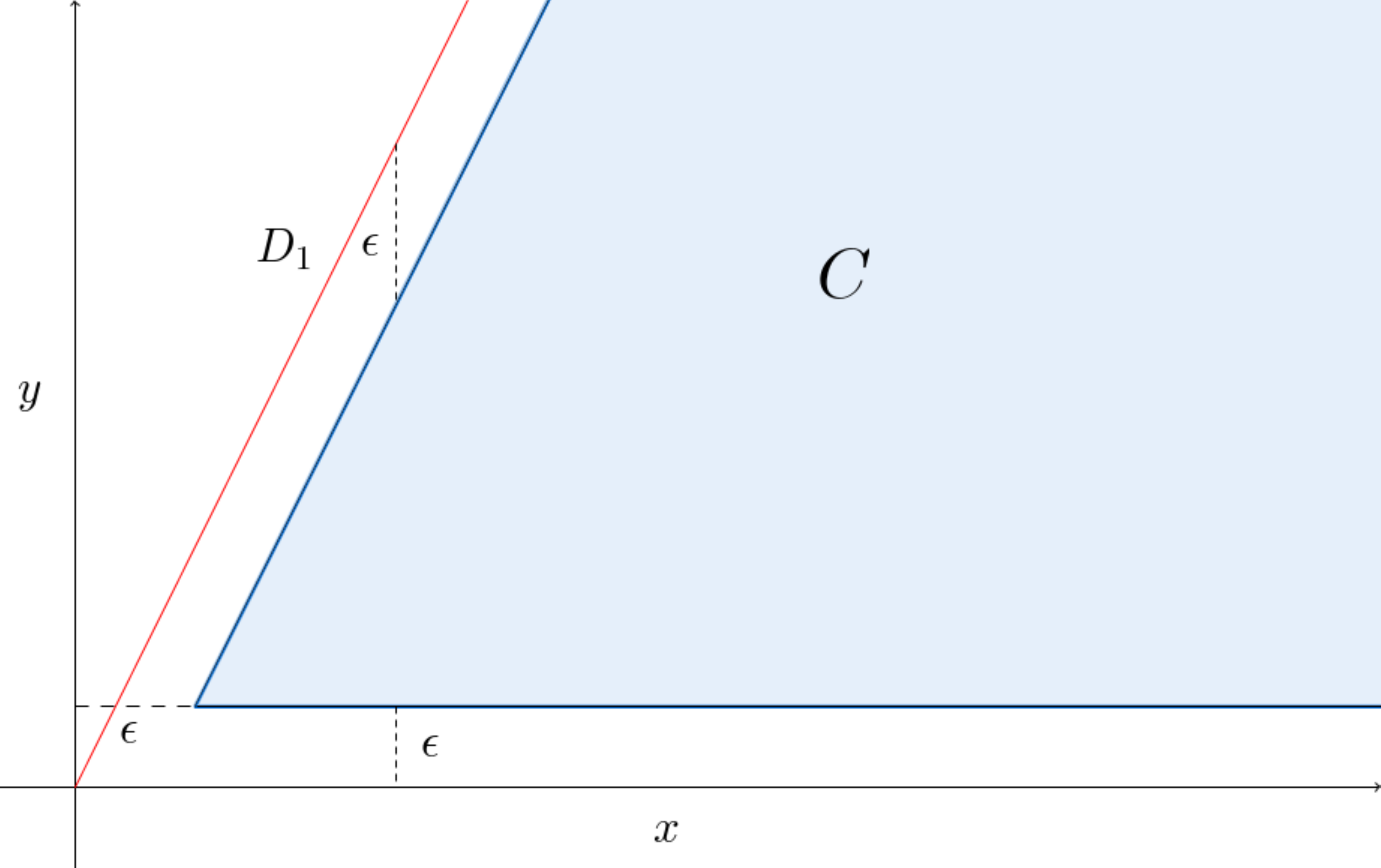}
    \caption{Case 2}
\end{figure}\\
 As in the first case, there is $\eta > 0$ such that
 \[ p_1(x,\xi) \geq \epsilon, \quad \forall (x,\xi) \in \mathcal{T}^2 \times (\R^2 \setminus B(0,\eta))\]
 and
  \[ p_2(x,\xi) \geq \epsilon, \quad \forall (x,\xi) \in \mathcal{T}^2 \times (\R^2 \setminus B(0,\eta)).\]
The last condition can be rewritten as
\begin{eqnarray*}
 p_2(x,\xi) \leq \beta_1 p_1(x,\xi) - \epsilon &\Leftrightarrow& \sqrt{ \frac{s_{32}}{s^{11}} \xi_2^2 - \frac{s_{22}}{s^{11}}\xi_3^2} \leq \sqrt{\frac{-s_{22}}{s_{23}}} \sqrt{ -\frac{s_{33}}{s^{11}} \xi_2^2 + \frac{s_{23}}{s^{11}}\xi_3^2} - \epsilon \\
 &\Leftrightarrow& \sqrt{ \frac{s_{32}}{s^{11}} \xi_2^2 - \frac{s_{22}}{s^{11}}\xi_3^2} \leq \sqrt{ \frac{s_{22}s_{33}}{s_{23}s^{11}} \xi_2^2 - \frac{s_{22}}{s^{11}}\xi_3^2} - \epsilon. 
\end{eqnarray*}
We recall that, thanks to the condition given in Remark \ref{rkrie},
\[\frac{s_{22} s_{33}}{s_{23}} > s_{32}.\]
Thus, there exists $\epsilon > 0$ small enough such that
  \[  p_2(x,\xi) \leq \beta_1 p_1(x,\xi) - \epsilon, \quad \forall (x,\xi) \in \mathcal{T}^2 \times \R^2.\]
  Finally, we have shown that for such an $\epsilon$, there exists $\eta > 0$ such that
  \[ \mathcal{T}^2 \times (\R^2 \setminus B(0,\eta)) \subset p^{-1}(C).\]
  \underline{Case 3:} $\partial_2(s_{22}) \neq 0$ and $\partial_3(s_{33}) = 0$. We have to avoid the half-line $\mathcal{D}_2$ which 
  has slope $\beta_2 = \sqrt{-\frac{s_{32}}{s_{33}}}$. We consider the cone 
     \[ C = \{(x,y) \in \R^2 \, \, \mathrm{such} \,\, \mathrm{that} \,\, \epsilon \leq x, \, \, \beta_2 x + \epsilon \leq y \}, \quad \epsilon > 0,\]
      \begin{figure}[h]
    \center
   \includegraphics[scale=0.25]{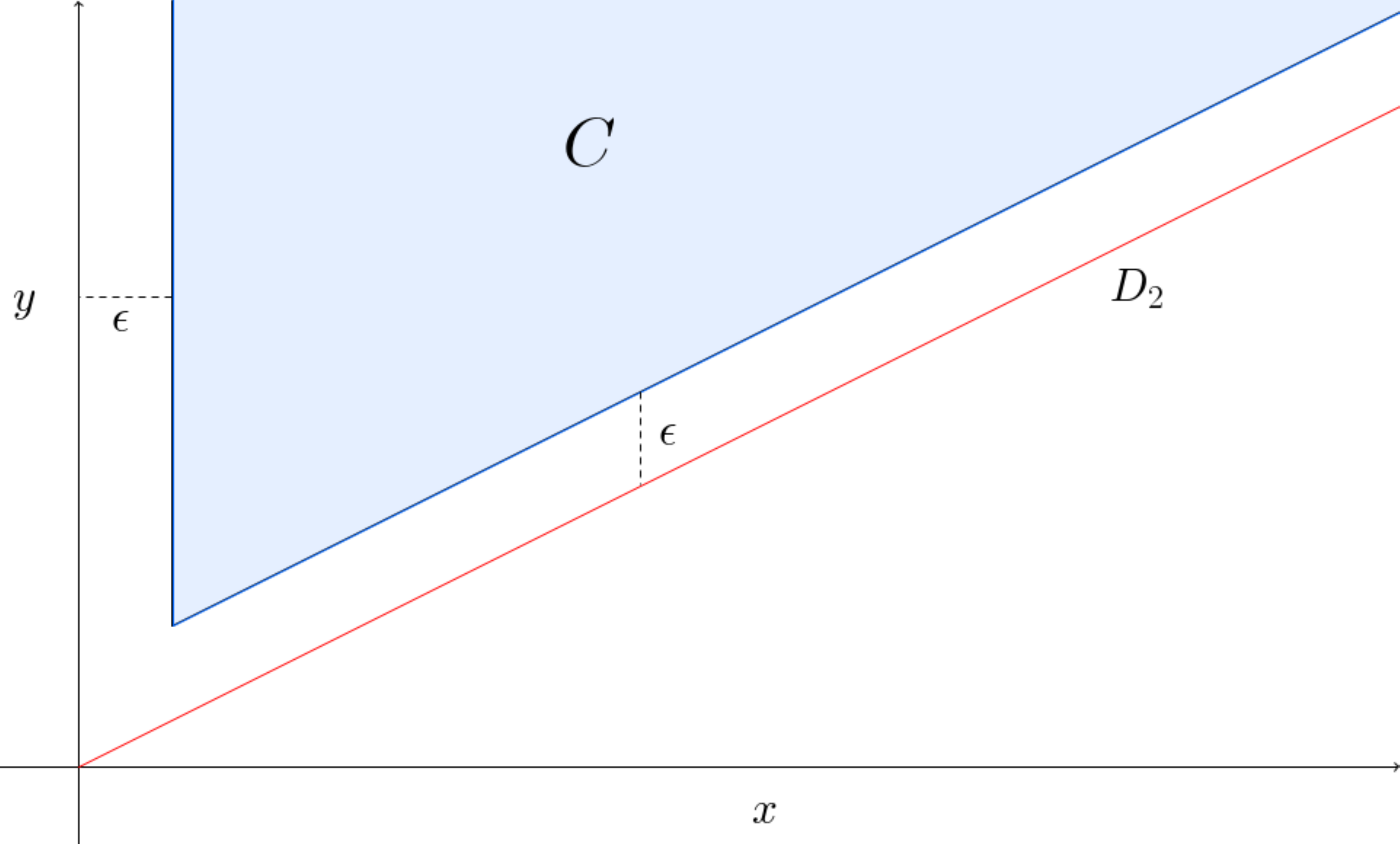}
    \caption{Case 3}
\end{figure}\\
and we show, as in the second case, that for $\epsilon > 0$ small enough there exists $\eta > 0$ such that
  \[ \mathcal{T}^2 \times (\R^2 \setminus B(0,\eta)) \subset p^{-1}(C).\]
  \underline{Case 4:} $\partial_2(s_{22}) = 0$ and $\partial_3(s_{33}) = 0$. We have to avoid $\mathcal{D}_1 \cup \mathcal{D}_2$ which have slopes
  $\alpha_1$ and $\alpha_2$ respectively. We consider the cone
   \[ C = \{(x,y) \in \R^2 \, \, \mathrm{such} \,\, \mathrm{that} \,\, \epsilon \leq x, \, \, \beta_2 x + \epsilon \leq y \leq \beta_1 x - \epsilon\}, \quad \epsilon > 0.\]
        \begin{figure}[h]
    \center
   \includegraphics[scale=0.25]{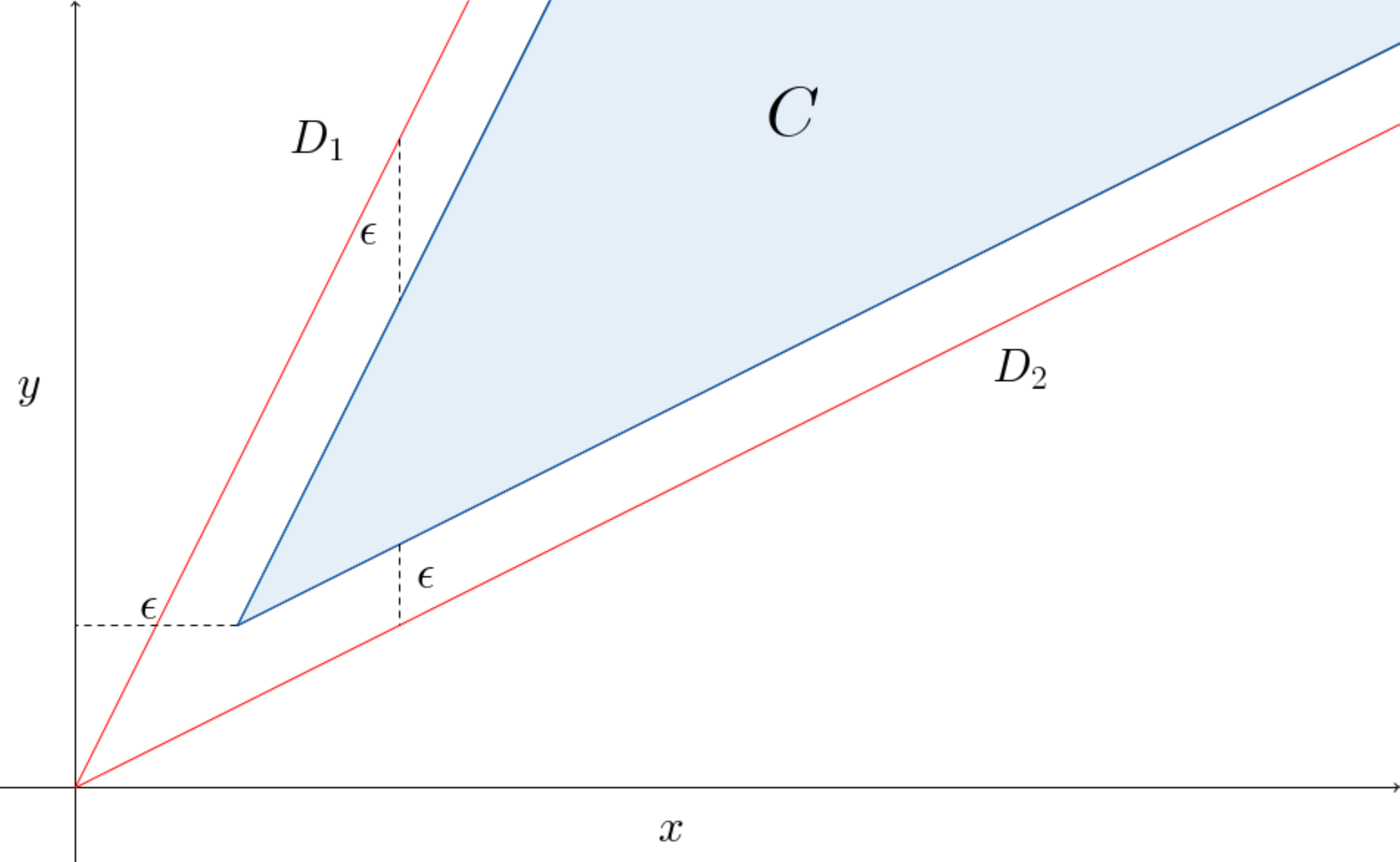}
    \caption{Case 4}
\end{figure}\\
 As in the first case, there is $\eta > 0$ such that
 \[ p_1(x,\xi) \geq \epsilon, \quad \forall (x,\xi) \in \mathcal{T}^2 \times (\R^2 \setminus B(0,\eta))\]
 and, as in the second and the third cases, there exists $\epsilon >0$ small enough such that
  \[  p_2(x,\xi) \leq \beta_1 p_1(x,\xi) - \epsilon, \quad \forall (x,\xi) \in \mathcal{T}^2 \times \R^2\]
and
  \[  \beta_2 p_1(x,\xi) + \epsilon \leq p_2(x,\xi), \quad \forall (x,\xi) \in \mathcal{T}^2 \times \R^2.\]
Thus, for $\epsilon > 0$ small enough, there exists $\eta > 0$ such that
  \[ \mathcal{T}^2 \times (\R^2 \setminus B(0,\eta)) \subset p^{-1}(C).\]
In conclusion, we have shown that in any cases there exists $\eta > 0$ such that
\begin{equation}\label{imrecicone}
  \mathcal{T}^2 \times (\R^2 \setminus B(0,\eta)) \subset p^{-1}(C).
\end{equation}
Moreover, in each case the cone $\mathcal{C}$ defined in (\ref{definitionducone}) is, by definition, included in the cone $C$ we considered and we can thus apply Theorem
\ref{cdv} to this cone. Therefore, thanks to (\ref{imreciboule})-(\ref{imrecicone}) we thus have shown that for $r>0$ large enough
 \begin{equation}\label{inclucdv}
   \mathcal{T}^2 \times \left( B \left( 0, \frac{r}{\sqrt{2c_1}} \right) \setminus B(0,\eta) \right) \subset p^{-1}(\mathcal{C}) \cap p^{-1}(B(0,r)) = p^{-1}(\mathcal{C} \cap B(0,r)).
 \end{equation}
From the inclusion (\ref{inclucdv}) we can deduce that there exists a constant $c > 0$ such that
\[ cr^2 \leq \frac{1}{4\pi^2} \mathrm{vol}_{\Omega} \left( B \left( 0, \frac{r}{\sqrt{2c_1}} \right) \setminus B(0,\eta) \right) \leq \frac{1}{4\pi^2} \textrm{vol}_{\Omega} \left( p^{-1}(\mathcal{C} \cap B(0,r)) \right).\]
Thanks to Theorem \ref{cdv}, we can then conclude that there exists $c>0$ such that
\[ \# \{ \lambda \in \mathcal{C} \cap \Lambda, \, \, |\lambda| \leq r \} \geq cr^2.\]
Finally, we recall that
\[ \Lambda = \{ (\sqrt{\mu_m^2 + M}, \sqrt{\nu_m^2 + M}), \quad m \geq 1 \}\]
and we note that, thanks to the fact that $\mu_m^2 \to +\infty$ and $\nu_m^2 \to +\infty$, as $m \to + \infty$,
\[\sqrt{\mu_m^2 + M} \sim |\mu_m| \quad \textrm{and} \quad \sqrt{\nu_m^2 + M} \sim |\nu_m|, \quad m \to + \infty.\]
We recall that
\[ n(r) = \# \{ \lambda \in \mathcal{C} \cap E_M , \, \, |\lambda| \leq r \},\]
without multiplicity, whereas the result obtained before was computed counting multiplicity. However, the multiplicity of the coupled eigenvalues is at most $4$ (see Remark
\ref{remmult}). Thus, even if it means divide by $4$, we can conclude that
  \[ \varlimsup \frac{n(r)}{r^2} > 0, \quad r \to +\infty.\]
\end{proof}

\vspace{0,5cm}

\noindent
\textit{Acknowledgments:} This paper was initiated by Thierry Daud\'e and Fran\c{c}ois Nicoleau while the author was visiting Thierry Daud\'e and Niky Kamran at McGill University
during his PhD. The author wants to deeply thank Thierry Daud\'e and Fran\c{c}ois Nicoleau for their help and their support and Niky Kamran for his hospitality and 
encouragement. The author also would like to warmly thank Georgi Popov and Thomas Wallez for useful discussions on the study of the coupled spectrum. Finally, the author also
would like to thank his friends and colleagues Valentin Samoyeau and Pierre Vidotto for their support.

{}


\newpage
\begin{thebibliography}{99}
\bibitem{AR} {\sc de Alfaro V., Regge T.},
{\em Potential scattering},
North-Holland Publishing Co., Amsterdam; Interscience Publishers Division John Wiley ans Sons, Inc., (1965).
\bibitem{Arn} {\sc Arnold V.I.},
{\em Mathematical methods of classical mechanics},
Springer Science \& Business Media, Vol. 60, (2013).
\bibitem{Benen} {\sc Benenti S.},
{\em Separability in Riemannian manifolds},
Preprint (2015), ArXiv:1512.07833.
\bibitem{Benen2} {\sc Benenti S.},
{\em Orthogonal separable dynamical systems},
Math. Publ., Silesian Univ. Opava, $\mathbf{1}$, (1993), 163-184.
\bibitem{BCR1} {\sc Benenti S., Chanu C., Rastelli G.},
{\em Remarks on the connection between the additive separation of the Hamilton-Jacobi equation and the multiplicative separation of the Schr\"odinger equation. I. The completeness
and Robertson condition},
Journal of Mathematical Physics $\mathbf{43}$, 5183, (2002).
\bibitem{BCR2} {\sc Benenti S., Chanu C., Rastelli G.},
{\em Remarks on the connection between the additive separation of the Hamilton-Jacobi equation and the multiplicative separation of the Schr\"odinger equation. II. First integrals
and symmetry operators},
Journal of Mathematical Physics $\mathbf{43}$, 5223, (2002).
\bibitem{BenF} {\sc Benenti S., Francaviglia M.},
{\em The theory of separability of the Hamilton-Jacobi equation and its applications to general relativity},
in General relativity and gravitation, Vol. 1, Plenum, New York-London, (1980), 393-439.
\bibitem{Be} {\sc Bennewitz C.},
{\em A proof of the local Borg-Marchenko Theorem},
Comm. Math. Phys. $\mathbf{211}$, (2001), 131-132.
\bibitem{Ber} {\sc Berndtsson B.},
{\em Zeros of analytic functions of several variables},
Ark. Mat. {\bf 16} (1978), no. 2, 251-262.
\bibitem{Bl} {\sc Bloom T.},
{\em A spanning set for $\mathcal{C}(I^n)$},
Trans. Amer. Math. Soc. 321, (1990), no.2, 741-759.
\bibitem{Bo} {\sc Boas R.P.},
{\em Entire functions},
Academic Press Inc., New York (1954).
\bibitem{BM} {\sc Bolsinov A. V., Matveev. V. S.},
{\em Local normal forms for geodesically equivalent pseudo-Riemannian metrics},
Trans. Amer. Math. Soc. $\mathbf{9}$, (2015), 6719-6749.
\bibitem{Bort} {\sc Borthwick D.},
{\em Spectral theory for infinite-area hyperbolic surface},
Birkh\"auser, Boston-Basel-Berlin, (2007).
\bibitem{BP} {\sc Borthwick D., Perry P. A.},
{\em Inverse scattering results for manifolds hyperbolic near infinity},
J. of Geom. Anal. $\mathbf{21}$, No. 2, (2011), 305-333.
\bibitem{Bou} {\sc Bouclet J.M.},
{\em Absence of eigenvalue at the bottom of the continuous spectrum on asymptotically hyperbolic manifolds},
Ann. Global Anal. Geom. 44, no 2, (2013), 115-136.
\bibitem{ChPa} {\sc Chalendar I., Partington J.},
{\em Multivariable approximate Carleman-type theorems for complex measures},
Ann. Probab. 35 (2007), no.1, 384-396.
\bibitem{CR} {\sc Chanu C., Rastelli G.},
{\em Fixed energy R-separation for Schr\"odinger equation},
Int. J. Geom. Methods Mod. Phys. 3 (2006), no. 3, 489-508. 
\bibitem{CdV1} {\sc Colin de Verdi\`ere Y.},
{\em Spectre conjoint d'op\'erateurs pseudo-diff\'erentiels qui commutent. I. Le cas non int\'egrable},
Duke Math. J. 46 (1979), no. 1, 169-182.
\bibitem{CdV2} {\sc Colin de Verdi\`ere Y.},
{\em Spectre conjoint d'op\'erateurs pseudo-diff\'erentiels qui commutent. II. Le cas int\'egrable},
Math. Z. 171 (1980), no. 1, 51-73.
\bibitem{Collins} {\sc Collins P.D.B.},
{\em An introduction to Regge theory and high energy physics},
Cambridge Monographs on Mathematical Physics, Vol. 4, (1977).
\bibitem{Pap1} {\sc Daud\'{e} T., Gobin D., Nicoleau F.},
{\em Local inverse scattering results at fixed energy in spherically symmetric asymptotically hyperbolic manifolds},
To appear in Inverse Problems and Imaging, Preprint (2013), ArXiv:1310.0733.
\bibitem{DaKaNi} {\sc Daud\'e T., Kamran N., Nicoleau F.},
{\em Inverse scattering at fixed energy on asymptotically hyperbolic Liouville surfaces},
To appear in Inverse problems, Preprint (2015), ArXiv:1409.6229.
\bibitem{DNK2} {\sc Daud\'e T., Kamran N., Nicoleau F.},
{\em Non-uniqueness results for the anisotropic Calderon problem with data measured on disjoint sets},
Preprint (2015), ArXiv:1510.06559.
\bibitem{DN3} {\sc Daud\'{e} T., Nicoleau F.}
{\em Direct and inverse scattering at fixed energy for massless charged Dirac fields by Kerr-Newman-De Sitter black holes},
Preprint (2013), ArXiv:1307.2842, to be published in Memoirs of the AMS.
\bibitem{DN} {\sc Daud\'{e} T., Nicoleau F.}
{\em Inverse scattering at fixed energy in de Sitter-Reissner-Nordstr\"{o}m black holes},
Ann. Henri Poincar\'{e} 12, (2011), 1-47.
\bibitem{DN4} {\sc Daud\'{e} T., Nicoleau F.}
{\em Local inverse scattering at a fixed energy for radial Schr\"odinger operators and localization of the Regge poles},
Preprint (2015), ArXiv:1502.02276, to be published in Annales Henri Poincar\'e.
\bibitem{DKSU} {\sc Dos Santos Ferreira D., Kenig C. E., Salo M., Uhlmann G.}
{\em Limiting Carleman weights and anisotropic inverse problems},
Invent. Math., $\mathbf{178}$ no. 1, (2009), 119-171.
\bibitem{DKLS} {\sc Dos Santos Ferreira D., Kurylev Y., Lassas M., Salo M.}
{\em The Calder\'on problem in transversally anisotropic geometries},
Preprint (2013), ArXiv:1305.1273, to be published in Journal of European Mathematical Society.
\bibitem{Eis} {\sc Eisenhart L.P.},
{\em Separable systems of Staeckel},
Annals of Math., 35 (1934), no. 2, 284D305.
\bibitem{Eis2} {\sc Eisenhart L.P.},
{\em Riemannian geometry},
Princeton University Press, Princeton, NJ (1997).
\bibitem{FY2} {\sc Freiling G., Yurko V.},
{\em Inverse Sturm-Liouville problems and their applications},
NOVA Science Publishers New York, (2001).
\bibitem{FY} {\sc Freiling G., Yurko V.},
{\em Inverse problems for differential operators with singular boundary conditions},
Math. Nachr. {\bf 278} no. 12-13, (2005), 1561-1578.
\bibitem{GS} {\sc Gesztesy F., Simon B.},
\emph{On local Borg-Marchenko uniqueness results},
Comm. Math. Phys. $\mathbf{211}$, (2000), 273-287.
\bibitem{G} {\sc Gobin D.},
\emph{Inverse scattering at fixed energy for massive charged Dirac fields in de Sitter-Reissner-Nordstr\"om black holes},
Inverse Problems $\mathbf{31}$, (2015).
\bibitem{GuiSB} {\sc Guillarmou C., S\'a Barreto A.},
\emph{Inverse problems for Einstein manifolds},
Inverse Probl. Imaging $\mathbf{3}$, (2009), 1-15.
\bibitem{GSB} {\sc Guillarmou C., S\'a Barreto A.},
\emph{Scattering and inverse scattering on ACH manifolds},
J. Reine Angew. Math. $\mathbf{622}$, (2008), 1-55.
\bibitem{GT2} {\sc Guillarmou C., Tzou L.},
\emph{Calder\'on inverse problem with partial data on Riemann surfaces},
Duke Math. J., $\mathbf{158}$, (2011), 83-120.
\bibitem{GT} {\sc Guillarmou C., Tzou L.},
\emph{The Calderon inverse problem in two dimensions. Inverse problems and applications: Inside Out. II},
119-166, Math. Sci. Res. Inst. Publ. $\mathbf{60}$, Cambridge University Press, Cambridge, (2013).
\bibitem{Horman} {\sc Hormander L.},
\emph{An introduction to complex analysis in several variables},
Elsevier, $\mathbf{7}$, (1973).
\bibitem{IK} {\sc Isozaki H., Kurylev J.},
{\em Introduction to spectral theory and inverse problems on asymptotically hyperbolic manifolds},
MSJ Memoirs $\mathbf{32}$, Preprint (2014), ArXiv:1102.5382.
\bibitem{IKL} {\sc Isozaki H., Kurylev Y., Lassas M.},
{\em Conic singularities, generalized scattering matrix, and inverse scattering on asymptotically hyperbolic surfaces},
Journal fur die reine und angewandte Mathematik (Crelle's journal), (2011).
\bibitem{JSB} {\sc Joshi M.S., S\'a Barreto A.},
{\em Inverse scattering on asymptotically hyperbolic manifolds},
Acta Mathematica, 184, (2000), 41-86.
\bibitem{KM1} {\sc Kalnins E.G., Miller Jr. W.},
{\em Intrinsic characterisation of orthogonal R separation for Laplace equations},
J. Phys. A 15 (1982), no. 9, 2699-2709.
\bibitem{KM3} {\sc Kalnins E.G., Miller Jr. W.},
{\em Killing tensors and variable separation for Hamilton-Jacobi and Helmholtz equations},
SIAM J. Math. Anal. $\mathbf{11}$, (1980), 1011-1026.
\bibitem{KM2} {\sc Kalnins E.G., Miller Jr. W.},
{\em The theory of orthogonal R-separation for Helmholtz equations},
Adv. in Math. 51 (1984), no. 1, 91-106.
\bibitem{KKL} {\sc Katchalov A., Kurylev Y., Lassas M.},
{\em Inverse boundary spectral problems},
Chapman and Hall/CRC (2001) 290 pp.
\bibitem{KS2} {\sc Kenig C., Salo M.},
{\em The Calder\'on problem with partial data on manifolds and applications},
Anal. PDE, $\mathbf{6}$ no. 8, (2013), 2003-2048.
\bibitem{KS} {\sc Kenig C., Salo M.},
{\em Recent progress in the Calderon problem with partial data},
Contemp. Math. $\mathbf{615}$, (2014), 193-222.
\bibitem{Koo} {\sc Koornwinder T.H.},
{\em A precise definition of separation of variables},
Proc. Fourth Scheveningen Conf., Scheveningen, (1979), Lecture Notes in Math., Springer, Berlin (1980).
\bibitem{KST} {\sc Kostenko A., Sakhnovich A., Teschl G.},
{\em Weyl-Titchmarsh theory for Schr\"odinger operators with strongly singular potentials},
Int. Math. Res. Not. $\mathbf{2012}$, (2012), 1699-1747.
\bibitem{LTU} {\sc Lassas M., Taylor M., Uhlmann G.},
{\em The Dirichlet-to-Neumann map for complete Riemannian manifolds with boundary},
Comm. Anal. Geom., $\mathbf{11}$, (2003), 207-221.
\bibitem{LU} {\sc Lassas M., Uhlmann G.},
{\em On determining a Riemannian manifold from the Dirichlet-to-Neumann map},
Ann. Sci. \'Ecole Norm. Sup. (4), $\mathbf{34}$, (2001), 771-787.
\bibitem{Leb} {\sc Lebedev N. N.},
{\em Special functions and their applications},
Prentice-Hall, Englewood Cliffs (1965).
\bibitem{LeeU} {\sc Lee J. M., Uhlmann G.},
{\em Determining anisotropic real-analytic conductivities by boundary measurements},
Comm. Pure Appl. Math., $\mathbf{42}$ no. 8, (1989), 1097-1112.
\bibitem{LC} {\sc Levi-Civita T.},
{\em Sulla integrazione della equazione di Hamilton-Jacobi per separazione di variabili},
Math. Ann., $\mathbf{59}$, (1904), 383-397.
\bibitem{Mara} {\sc Marazzi L.},
{\em Inverse scattering on conformally compact manifolds},
Inverse Probl. Imaging, $\mathbf{3}$, (2009), 537-550.
\bibitem{MM} {\sc Mazzeo, R., Melrose R. B.},
{\em Meromorphic extension of the resolvent on complete spaces with asymptotically constant negative curvature},
J. Funct. Anal. $\mathbf{75}$, (1987), 260-310.
\bibitem{Mel} {\sc Melrose R. B.},
{\em Geometric scattering theory},
Cambridge University Press, Cambridge, (1995).
\bibitem{Mi} {\sc Miller Jr. W.},
{\em Mechanisms for variable separation in partial differential equations and their relationship to group theory},
Symmetries and nonlinear phenomena $\mathbf{9}$, (1988), 188-221.
\bibitem{Mu} {\sc Musin, I. Kh.},
{\em On the Fourier-Laplace representation of analytic functions in tube domains},
Collect. Math. 45 (1994), no. 3, 301-308.
\bibitem{New} {\sc Newton R.G.},
{\em Scattering theory of waves and particles},
Dover Publications, Mineola (2002) (Reprint of the 1982 second edition New York: Springer, with list of errata prepared for this edition by the author)
\bibitem{Ra} {\sc Ramm A.G.}
{\em An inverse scattering problem with part of the fixed-energy phase shifts},
Comm. Math. Phys. 207(1), (1999), 231-247.
\bibitem{Re} {\sc Regge T.}
{\em Introduction to complex orbita momenta},
Nuevo Cimento XIV(5), (1959), 951-976.
\bibitem{Rob} {\sc Robertson H. P.}
{\em Bemerkung \"uber separierbare Systeme in der Wellenmechanik},
Math. Ann., $\mathbf{98}$, (1928), 749-752.
\bibitem{SB} {\sc S\'a Barreto A.},
{\em Radiation fields, scattering, and inverse scattering on asymptotically hyperbolic manifolds},
Duke Math. J., $\mathbf{129}$, (2005), 407-480.
\bibitem{S} {\sc Salo M.},
{\em The Calder\'on problem on Riemannian manifolds, Inverse problems and applications: inside out. II},
Math. Sci. Res. Inst. Publ., $\mathbf{60}$, Cambridge Univ. Press, Cambridge, (2013), 167-247.
\bibitem{Sha} {\sc \v{S}apovalov, V. N.},
{\em St\"ackel spaces},
Sibirsk. Mat. Zh., $\mathbf{20}$, (1979), 1117-1130, 1168.
\bibitem{Sta1} {\sc St\"ackel P.},
{\em \"Uber die integration der Hamilton-Jacobischen differentialgleichung mittelst separation der variabel, Habilitationsschirft},
Ph.D. thesis Halle, (1891).
\bibitem{Sta} {\sc St\"ackel P.},
{\em Ueber die Bewegung eines Punktes in einer $n$-fachen Mannigfaltigkeit},
Math. Ann., $\mathbf{42}$, (1893), 537-563.
\bibitem{Te} {\sc Teschl G.},
{\em Mathematical Methods in Quantum Mechanics},
Graduate Studies in Mathematics Vol. 99, AMS Providence, Rhode Island, (2009).
\bibitem{U} {\sc Uhlmann G.},
{\em Electrical impedance tomography and Calderon's problem},
Inverse Problems $\mathbf{25}$, (2009), 123011, 39p.
\end{thebibliography}
\end{document}